\UseRawInputEncoding
\documentclass[11pt]{amsart}
\usepackage{upkg}

\title{A positive combinatorial formula for
the double Edelman---Greene coefficients}

\author{Jack Chen-An Chou}
\address{Jack Chen-An Chou, Department of Mathematics, University of Florida, Gainesville, FL 32611.}
\email{c.chou@ufl.edu}

\author{Tianyi Yu}
\address{Tianyi Yu, Laboratoire d'Alg\`ebre, de Combinatoire et d'Informatique Math\'ematique, Universit\'e du Qu\'ebec \`a Montr\'eal, Montreal QC, Canada}
\email{yu.tianyi@uqam.ca}

\begin{document}

\begin{abstract}
Lam, Lee, and Shimozono introduced the double Stanley symmetric functions in their study of the equivariant geometry of the affine Grassmannian.
They proved that the associated double Edelman--Greene coefficients, 
the double Schur expansion coefficients of these functions,
are positive, a result later refined by Anderson.
They further asked for a combinatorial proof of this positivity.
In this paper, we provide the first such proof, together with a combinatorial formula that manifests the finer positivity established by Anderson.
Our formula is built from two combinatorial models: bumpless pipedreams and increasing chains in the Bruhat order.
The proof relies on three key ingredients: a correspondence between these two models, a natural subdivision of bumpless pipedreams, and a symmetry property of increasing chains.

\end{abstract}

\maketitle

\section{Introduction}
\label{S: Intro}

We give a combinatorial formula for the double Edelman--Greene coefficients introduced by Lam, Lee and Shimozono~\cite{LLS}, exhibiting a positivity that has been established geometrically~\cites{LLS, And} but not combinatorially.
These coefficients govern the double Schur
expansion of double Stanley functions. 

\subsection*{From Schubert polynomials
to double Stanley functions}
A \definition{permutation}
is a bijection from $\Z$ to $\Z$
with finitely many non-fixed points.
Let $S_\Z$ be the set of permutations 
and $S_n$ be the subset of $S_\Z$
whose non-fixed points are in $[n] := \{1, \dots. n\}$.
For $w \in S_n$, 
Lascoux and Sch\"utzenberger~\cite{LS:Schub} 
introduced 
the \definition{Schubert polynomial} $\fS_w(\x)
\in \Q[x_1, \dots, x_{n-1}]$. 
These polynomials
represent the cohomology classes of Schubert varieties in the flag variety. 
Their equivariant analogue, the \definition{double Schubert polynomial} 
$\fS_w(\x;\y)$, represents Schubert classes in the torus-equivariant cohomology ring.
In their foundational work, Lam, Lee, and Shimozono~\cite{LLS} introduced the
\definition{back stable double Schubert function} $\bfS_w(\x; \y)$ for any $w \in S_\Z$
as the ``back stable limit'' of 
certain specializations of double Schubert polynomials.
The function $\bfS_w(\x; \y)$ lives in the ring $\Lambda(\x || \y) \otimes_{\Q[\y]} \Q[\x; \y]$,
where $\Lambda(\x || \y)$ is the 
\definition{ring of double symmetric functions} and $\Q[\x; \y]$ is the polynomial
ring in $x_i, y_i$ with $i \in \Z$.
The ring $\Lambda(\x || \y)$ has a basis
consisting of the 
\definition{double Schur functions} 
$s_\lambda(\x || \y)$ indexed by partitions.  
These functions were studied by Molev~\cite{Mol}
and recover the classical Schur functions
by specializing $\y\mapsto 0$.
Each partition $\lambda$ 
is associated with a $0$-Grassmannian 
permutation $w_\lambda$ 
(see Definition~\ref{D: parititons to 0-Grass}).
Following \cite{LLS}*{Prop.~4.11},
we may define $s_\lambda(\x || \y)$ 
as $\bfS_{w_\lambda}(\x;\y)$.

For $w \in S_\Z$, 
one may expand $\bfS_w(\x; \y)$ into the double Schur basis:
\[
\bfS_w(\x; \y) = \sum_{\lambda} a^w_\lambda(\x; \y) \, s_\lambda(\x || \y),
\]
with coefficients $a^w_\lambda(\x; \y) \in \Q(\x;\y)$.
The \definition{double Stanley function} $F_w(\x ||\y)$ may be defined as
\[
F_w(\x || \y) := \sum_{\lambda} j^w_\lambda(\y) \, s_\lambda(\x || \y),
\qquad
j^w_\lambda(\y) := a^w_\lambda(\x;  \y)\big|_{\x\mapsto \y}
= a^w_\lambda(\y; \y).
\]
They recover the classical \definition{Stanley symmetric function}~\cite{StanSym} under
the specialization $\y \mapsto 0$.
The polynomials $j^w_\lambda(\y)$
are called the \definition{double Edelman--Greene coefficients}.

\subsection*{Positivity of double Edelman--Greene coefficients}
One main result established by Lam, Lee and Shimozono~\cite{LLS}*{Thm.~4.22}
is that the double Edelman--Greene coefficients
are positive, meaning that
$j^w_\lambda(\y) \in \Z_{\ge 0}[\, y_i - y_j : i, j \in \Z, i\prec j \,]$,
where $\prec$ denotes the total order on $\Z$:
\[
1 \prec 2 \prec 3 \prec \cdots \prec -2 \prec -1 \prec 0.
\]
We call this \definition{Lam--Lee--Shimozono (LLS) positive}.
Their proof is geometric:
they identify $F_w(\x \,\|\, \y)$ with the equivariant Schubert class
of the affine flag variety and use equivariant localization
to express the structure constants as positive combinations of $y_i-y_j$. 
This positivity
extends the classical result~\cites{EG,LS85}:
the Edelman--Greene coefficients, which can be recovered by $\y \mapsto 0$ in $j^w_\lambda(\y)$, 
are in $\Z_{\geq 0}$.
More recently, by analyzing the structure of
equivariant Schubert classes in the affine Grassmannian,
Anderson~\cite{And} refined
the LLS-positivity by imposing 
additional constraints on the power of 
$y_i - y_j$ in each monomial.
We rephrase this positivity following the terminology of~\cite{GHY}.
The forms $(y_i - y_j)$
with $i\prec j$ are classified 
into three types:
\[
\text{Type 1: } 0<i<j, \quad
\text{Type 2: } i<j\le 0, \quad
\text{Type 3: } j\le 0 < i.
\]

\begin{thm}[{\cite{And}}]
\label{T: Anderson}
The double Edelman--Greene coefficient $j^w_\lambda(\y)$
can be written as a polynomial of the linear forms $(y_i - y_j)$ with $i \prec j$ with coefficient in $\Z_{\geq 0}$.
In each monomial, 
each type~1 or type~2 factor has 
degree at most $1$,
and each type~3 factor has 
degree at most $2$.
\end{thm}

As far as the authors are aware, it remains an open problem to give a combinatorial proof of Theorem~\ref{T: Anderson} and to produce a combinatorial formula for $j^w_\lambda(\y)$ that directly exhibits Anderson's refined positivity, or even the original LLS-positivity.
The case when $w$ is vexillary (i.e.\ $2143$-avoiding) is resolved
by Gregory, Hamaker and Yu~\cite{GHY}
using a tableaux formula.
The goal of this paper is 
to resolve the general case of this problem:
\begin{prob}
\label{Pb: main}
Give a combinatorial formula
for $j^w_\lambda(\y)$
that satisfies the condition
in Theorem~\ref{T: Anderson}
\end{prob}

Our approach is completely combinatorial and our formula 
can be obtained by combining~\eqref{EQ: intro a},~\eqref{EQ: Intro Diag cut} and~\eqref{EQ: intro final TBPD}.
It involves two combinatorial objects:
bumpless pipedreams and increasing chains. 

\subsection*{Bumpless pipedreams}
Lam, Lee, and Shimozono~\cite{LLS} introduced \definition{bumpless pipedreams (BPDs)} as a combinatorial model for double Schubert polynomials and their back stable analogues.  
There has been a recent 
surge of research on 
these new combinatorial objects~\cites{GH, Huang,Huang2,HSY,KW,W,Yu}.

Consider the integer lattice $\Z\times\Z$, viewed as a grid whose row (resp. col) indices increase from top (resp. left) to bottom (resp. right).  
An element $(r,c)\in\Z\times\Z$ corresponds to the cell in row~$r$ and column~$c$ and   
a subset $\A\subseteq\Z\times\Z$ is therefore identified with a set of cells.
A BPD $D$ on $\A$ is a map
\[
D:\A\longrightarrow
\{\btile,\htile,\vtile,\rtile,\jtile,\ptile\},
\]
satisfying certain conditions. 
The tiles are decorated by pipes, 
so we may trace the pipes from top and right
to bottom and left.
Each BPD $D$ is assigned a \definition{weight}
$$
\wt(D):=\prod_{(i,j): D(i,j) = \btile} (x_i-y_j). 
$$
\begin{thm}[\cites{LLS, W}]
\label{T: LLS BPD}
For $u\in S_n$ and $w \in S_\Z$,
\[
\fS_u(\x;\y)=\sum_{D\in\BPD_n(u)}\wt(D),\textrm{ and} 
\quad
\bfS_w(\x;\y)=\sum_{D\in\bBPD(w)}\wt(D),
\]
where $\BPD_n(w)$ and $\bBPD(w)$ denote certain sets of BPDs on $[n]\times[n]$ and $\Z\times\Z$, respectively
(see Definition~\ref{D: BPD for Schubert}).
\end{thm}
We take this result as the definition of $\fS_w(\x;\y)$ and $\bfS_w(\x;\y)$.

\subsection*{Slicing BPDs}
We reduce Problem~\ref{Pb: main}
by finding formulas for 
$a^w_\lambda(\x; \y)$ that are easier to study. 
To achieve this, we study the combinatorics of BPDs.
In particular, we cut BPDs into smaller
pieces and study the generating function 
of each piece. This is an
overview of~\S\ref{S: slice}.

Our first cut is the \definition{horizon cut} on $D\in\bBPD(w)$, 
decomposing $D$ into a BPD of the 
\definition{upper half-plane}
$H^\uparrow := \Z_{\le 0}\times\Z$
and a BPD of the 
\definition{lower half-plane}
$H^\downarrow := \Z_{>0}\times\Z$.
This construction is a natural extension of the ``halfplane crossless pipedreams''
and ``rectangular bumpless pipedreams''
in Section 5.3 and 5.4 of~\cite{LLS}.

Let $S_-$ denote the set of permutations $w\in S_\Z$
satisfying $w(1)<w(2)<\cdots$.
For $\sigma\in S_-$, we define $\LBPD(\sigma,w)$,
a family of BPDs of $H^\downarrow$,
and denote its generating function by $\lfS_{\sigma,w}(\x;\y)$ (see Definition~\ref{D: Lower Schub}).
Finally, 
we show that every coefficient $a^w_\lambda(\x;\y)$
can be expressed in terms of these functions:
\begin{coralowerlower}
For $w\in S_\Z$,
    \begin{equation}
    \label{EQ: intro a}
    a^w_\lambda(\x;\y)
    =\sum_{\sigma\in S_-}
    \omega_1\!\big(\lfS_{w_{\lambda'},\neg\sigma}(\x;\y)\big)\:
    \lfS_{\sigma,w}(\x;\y),
    \textrm{ so }
    j^w_\lambda(\y)
    =\sum_{\sigma\in S_-}
    \omega_1\!\big(\lfS_{w_{\lambda'},\neg\sigma}(\y;\y)\big)\:
    \lfS_{\sigma,w}(\y;\y).
    \end{equation}
\end{coralowerlower}

Here, $\neg$ is the involution on $S_\Z$ defined in Definition~\ref{D: neg}
and $\omega_1$ is the operator sending each $x_i,y_i$ to $x_{1-i},y_{1-i}$, respectively.
By Lemma~\ref{L: Omega},
$\omega_1$ interchanges type~1 and type~2 forms while preserving type~3.
Thus, by~\eqref{EQ: intro a}, 
Problem~\ref{Pb: main} reduces to:
\begin{prob}
\label{Pb: lower}
For $w \in S_\Z$ and $\sigma \in S_-$, 
give a combinatorial formula
for $\lfS_{\sigma,w}(\y;\y)$
that is a sum where each
summand is a distinct product
of type~1 or type~3 factors.
\end{prob}

\medskip
Then we move to study $\lfS_{\sigma,w}(\x;\y)$.
The second cut is the \definition{diagonal cut}, breaking $D \in \LBPD(\sigma, w)$ 
into a BPD of the ``triangular'' 
region $\A^{\tri}:= \{(i, j) : 0 < i < j\}$ 
and a BPD of the ``trapezoid'' region
$\A^{\tra}:= \{(i, j) : 0 < i \textrm { and } j \leq i\}$.
This cut is analogous to
the tableaux division in Section~3.2 of~\cite{GHY}.

\begin{propDiag}
The diagonal cut yields a bijection, leading to
the polynomial identity:
\begin{equation}
\label{EQ: Intro Diag cut}
\lfS_{\sigma, w}(\x;\y)
= \sum_{\gamma \in \Gamma(\sigma, w)}
\left(\sum_{D\in\BPD^{\tra}(\gamma)}\wt(D)\right)
\left(\sum_{D\in\BPD^{\tri}(\gamma)}\wt(D)\right).
\end{equation}    
\end{propDiag}

Here, each $\gamma$ we summed over
is a tuple of maps called ``boundary conditions''
(see Definition~\ref{D: Boudnary condition}) 
determined by $\sigma$ and $w$.
Intuitively, $\gamma$ dictates how each pipe
travels from the right and top boundary
of $\A^{\tra}$ to the bottom boundary.
Then $\BPD^{\tra}(\gamma)$ 
is a set of BPDs of $\A^{\tra}$
satisfying $\gamma$.
We also define $\BPD^{\tri}(\gamma)$ as
a set of BPDs of $\A^{\tri}$
satisfying some boundary condition
determined by $\gamma, \sigma$ and $w$ (see Definition~\ref{D: Gamma}).
Recall that $\A^{\tri}$
consists of the cells $(i,j)$ with $0 < i < j$.
Therefore, 
after $\x \mapsto \y$,
the generating function 
$\sum_{D\in\BPD^{\tri}(\gamma)}\wt(D)$
in~\eqref{EQ: Intro Diag cut} is already a combinatorial formula
in which each summand is a product of distinct type~1 factors.
Problem~\ref{Pb: lower} reduces to:

\begin{prob}
\label{Pb: TBPD}
Take $\sigma \in S_-$ and $w \in S_\Z$.
For $\gamma \in \Gamma(\sigma, w)$, 
give a combinatorial formula
for $\fS^{\tra}_\gamma(\y; \y)$ where
$$\fS^{\tra}_\gamma(\x; \y) := \sum_{D\in\BPD^{\tra}(\gamma)}\wt(D).$$
The formula should be
a sum where each summand is a distinct product of
type~3 factors.
\end{prob}

We first find integers $a \leq 0 < n$ 
and define $\TBPD_{a,n}(\gamma)$.
These are BPDs
of a finite region 
such that 
$\fS^{\tra}_\gamma(\x; \y) = \sum_{D\in\BPD_{a,n}^{\tra}(\gamma)}\wt(D)$.
To understand these BPDs, we use increasing chains. 

\subsection*{Increasing chains.}
Sottile~\cite{So96} introduced \definition{increasing chains} to model the Pieri rule for Schubert polynomials. 
These are certain chains
in the Bruhat order (see Definition~\ref{D: Inc chain}).
For permutations $u$ and $w$, we write $u\xrightarrow{k}w$ if there exists an increasing $k$-chain from $u$ to $w$.  
Given a sequence $\alpha=(\alpha_1,\dots,\alpha_m)$ of integers, let $C(u,w,\alpha)$ denote the set of sequences $(u_1,\dots,u_{m+1})$ such that
$u_1=u$, $u_{m+1}=w$, and $u_i\xrightarrow{\alpha_i}u_{i+1}$ for each $i\in[m]$.  
Each such sequence is assigned a \definition{weight} $\wt^n_\alpha(u_1, \dots, u_{m+1})(x_1, \dots, x_m;\y)$
depending on $n\in\Z$ and the sequence $\alpha$ (see~\eqref{EQ: Define wt}).
We then define the generating function
\begin{equation}
\label{EQ: define fC}
\fC^n_{u, w, \alpha}(x_1, \dots, x_{m}; \y)
:= 
\sum_{(u_1, \dots, u_{m+1}) \in C(u, w, \alpha)}
\wt^n_\alpha(u_1, \dots, u_{m+1})(x_1, \dots, x_m; \y).    
\end{equation}
Using a Pieri identity of Samuel~\cite{Sam}, 
we show that this generating function is symmetric
under permuting the $\alpha$ and $\x$-variables
in Proposition~\ref{P: double symmetry}.
In particular, 
\begin{equation}
\label{EQ: double rev}
\fC^n_{u,w,\alpha}(x_1,\dots,x_m;\y)
=
\fC^n_{u,w,\rev(\alpha)}(x_m,\dots,x_1;\y),
\end{equation}
where $\rev(\alpha)$ denotes sequence reversal.

\medskip
For $u \in S_n$,
Yu~\cite{Yu} established a bijection
between $\BPD_n(w)$ and $C_{u, w_{0,n}, (n-1, \dots, 1)}$
where $w_{0,n} \in S_n$ is the permutation with one-line notation 
$[n, \dots, 1]$.
Analogous to this bijection, 
we establish a bijection:
\[
\TBPD(\gamma) \longleftrightarrow C(U, W, \rev(\alpha)),
\]
where $U$ and $W$ are permutations 
and $\alpha$ is a sequence of $n-1$ integers,
all constructed from $\gamma$ (see Definition~\ref{D: U W alpha}).
If $D \mapsto (u_n, \dots, u_1)$ under this bijection,
we show that $$\wt(D) \prod_{1 \leq i < j \leq n} (x_i - y_j) = \wt_{\rev(\alpha)}^n(x_{n-1}, \dots, x_1; \y).$$
Thus, we obtain the identity
\[
\fS^{\tra}_\gamma(\x; \y) \cdot 
\prod_{1 \leq i < j \leq n} (x_i - y_j)
= \fC_{U, W, \rev(\alpha)}^n(x_{n-1}, \dots, x_1; \y)
= \fC_{U, W, \alpha}^n(x_{1}, \dots, x_{n-1}; \y),
\]
where the second equality follows from~\eqref{EQ: double rev}.
Consequently,
\begin{propTBPD}
\begin{equation}
\label{EQ: intro final TBPD}
\fS^{\tra}_\gamma(\y; \y) = 
\sum_{(u_1, \dots, u_n) \in 
C(U, W, \alpha)}
\frac{
\wt_{\alpha}^n(u_1, \dots, u_n)
(y_1, \dots, y_{n-1}; \y)
}{
\prod_{1 \leq i < j \leq n} (y_i - y_j)
}
\end{equation}
Moreover, for each $(u_1, \dots, u_n) \in C(U, W, \alpha)$,
its contribution to the right-hand side 
is either $0$ or a distinct product
of type~3 factors. 
This resolves Problem~\ref{Pb: TBPD}.
\end{propTBPD}

\section*{Acknowledgements}
We thank Nantel Bergeron, Adam Gregory, Zachary Hamaker, Brendon Rhoades, Linus Setiabrata, Mark Shimozono, 
Hunter Spink, and Vasu Tewari for many helpful and inspiring conversations. 
We thank Zachary Hamaker for pointing out important 
references for Theorem~\ref{T: Sam}.
We thank Zachary Hamaker, Thomas Lam, Seung Jin Lee, 
Brendon Rhoades and Linus Setiabrata for a careful reading of an earlier draft and providing useful feedback.
\section{Background}
\label{S: Background}

\subsection*{Permutations}

For integers $a \leq n$, 
let $[a,n] := \{a, a+1, \dots, n\}$.
We use $S_{[a,n]}$ to denote the set of 
permutations whose non-fixed points
are contained in $[a,n]$.
We represent such $w \in S_{[a,n]}$
via its one-line notation $[w(a), w(a+1), \dots, w(n)]$.
We will often specify we are giving the one-line 
notation for the permutation as an element of which $S_{[a,n]}$.
For instance, 
consider $w \in S_{[-2,1]}$
with $w(-2) = 1$, 
$w(-1) = -1$, $w(0) = -2$, and $w(1) = 0$.
The one-line notation of $w \in S_{[-2,1]}$ is $[1,-1,-2,0]$,
while the one-line notation of $w \in S_{[-2,2]}$
is $[1, -1, -2,0,2]$.

Notice that each $S_{[a,n]}$
behaves essentially the same as $S_{n-a+1}$.
More precisely, 
let $\std_{[a,n]}$ denote the 
\definition{standardization map}.
For $w \in S_{[a,n]}$, 
$\std_{[a,n]}(w)$ is the permutation
in $S_{n - a + 1}$ defined by
\begin{equation}
\label{EQ: define std}
\std_{[a,n]}(w)(i) = w(i + a - 1) - a + 1
\quad \text{for } i \in [n-a+1].
\end{equation}
Clearly, $\std_{[a,n]}$
is a bijection from $S_{[a,n]}$ to $S_{n-a+1}$.

\medskip

The \definition{descent set} of $w \in S_\Z$ is
\[
\Des(w) := \{\, i \in \Z : w(i) > w(i+1)\,\}.
\]
We say that $w$ is \definition{$k$-Grassmannian}
if $\Des(w)\subseteq\{k\}$.
Notice that for $w$ with $\Des(w)\subset \Z_{\leq 0}$,
we can completely determine $w$ 
by its values on $\Z_{\leq 0}$:
once $w(r)$ is specified for each $r\leq 0$,
the values $w(i)$ for $i > 0$
must be the $i\textsuperscript{th}$ 
smallest integer not in $w(\Z_{\leq 0})$.
Dually, if $\Des(w)\subset \Z_{\geq 0}$,
then $w$ is completely determined by its values on nonnegative indices.
Combining these two ideas, we obtain
a bijection between $0$-Grassmannian permutations 
and certain sets:
the map $w \mapsto w(\Zn)$
is a bijection from $0$-Grassmannian permutations
to sets $I \subset \Z$ such that $I$ (resp. $\Z \setminus I$)
is bounded from above (resp. below).

Moreover, 
$0$-Grassmannian permutations 
are in bijection with \definition{partitions},
that is, weakly decreasing sequences
$\lambda=(\lambda_1,\lambda_2,\dots)$ of 
nonnegative integers with $\lambda_i = 0$ for all sufficiently large $i$.

\begin{defn}
\label{D: parititons to 0-Grass}
Given a partition $\lambda$,
let \definition{$w_\lambda$} be the unique
$0$-Grassmannian permutation such that
$w_\lambda(i) = i+\lambda_{1-i}$ for all $i \in \Z_{\leq 0}$.
\end{defn}
It is a classical fact that $\lambda \mapsto w_\lambda$
is a bijection.

\begin{rem}
\label{R: Sets and paritions}
We introduced three objects 
that are in bijection: partitions,
$0$-Grassmannian permutations, 
and sets $I \subset \Z$ such that $I$ (resp. $\Z \setminus I$)
is bounded from above (resp. below).
We have explicitly described the bijection 
from the first to the second, 
and from the second to the third. 
The explicit descriptions of the remaining bijections
are left as exercises.
We will use these other bijections implicitly;
for instance, given a partition $\lambda$,
we may write ``let $I$ be the
set corresponding to $\lambda$''. 
\end{rem}

For $i < j$, let $\tau_{i,j}$ denote the \definition{transposition}
\[
\tau_{i,j}(i)=j,\quad \tau_{i,j}(j)=i,\quad \tau_{i,j}(p)=p \quad (p\notin\{i,j\}).
\]
Let $s_i:= \tau_{i, i+1}$ denote the \definition{adjacent transposition}.
The set $\{s_i: i\in\Z\}$ generates $S_\Z$ as a group.
Thus every $w\in S_\Z$ has an expression $w=s_{i_1}\cdots s_{i_m}$.
When $m$ is minimal, such expression is called \definition{reduced}, and the \definition{length} $\ell(w)$ is defined to be $m$.
Equivalently, 
$\ell(w) := |\{(i,j)\in \Z^2 : i < j,\, w(i) > w(j)\}|$.
We record a classical fact about reduced expressions for later use.

\begin{lem}
\label{L: Des and red}
For $w \in S_\Z$ and $j \in \Z$, 
$j\in \Des(w)$
if and only if $s_j$ appears at the end
of some reduced expression of $w$.  
\end{lem}

Finally, we recall the $\neg$ involution on $S_\Z$.
\begin{defn}~\cite{LLS}
\label{D: neg}
Let $\neg$ be the group automorphism 
on $S_\Z$ by setting $\neg(s_i)=s_{-i}$.
\end{defn}

\begin{lem}\label{L: Negation and Des}
If $w\in S_\Z$ satisfies $\Des(w)\subset \Z_{\leq 0}$,
then $\Des(\neg w)\subset \Z_{\geq 0}$.
\end{lem}
\begin{proof}
By Lemma~\ref{L: Des and red},
since $\Des(w)\subset \Z_{\leq 0}$,
any reduced expression of $w$ ends with some $s_j$
with $j \leq 0$.
Applying $\neg$ yields a reduced expression of $\neg w$
ending with $s_{-j}$ with $-j\geq 0$.
By Lemma~\ref{L: Des and red} again,
$\Des(\neg w)\subset \Z_{\geq 0}$.
\end{proof}

\medskip
For a partition $\lambda$, 
the \definition{double Schur function} $s_\lambda(\x || \y)$
is a function involving $x_i$ with $i \in \Zn$ and $\y$--variables.
Recall from~\S\ref{S: Intro} that 
$s_\lambda(\x || \y) := \bfS_{w_\lambda}(\x;\y)$.
We also recall the involution $\widetilde{\omega}:=\omega_1\otimes\omega_2$ on
$\Q(\x;\y)\otimes \Lambda(\x;\y)$ from \cite{LLS}; where $\omega_1$ is the involution on $\Q(\x;\y)$ defined by
\[
\omega_1(x_i)=x_{1-i},\qquad
\omega_1(y_i)=y_{1-i};
\]
and $\omega_2$ is the involution on $\Lambda(\x||\y)$ defined by $\omega_2(s_\lambda)=s_{\lambda'}$,
where $\lambda'$ is the \definition{conjugation}
of $\lambda$:
$$
\lambda' = (\lambda'_1, \lambda_2', \dots),
\textrm{ where $\lambda_i'$ is the number of 
entries in $\lambda$ that are at least $i$.}
$$
The effect of $\widetilde{\omega}$ on back stable Schubert functions
is characterized by the involution $\neg$
on $S_\Z$.

\begin{pro}[\cite{LLS}*{Proposition 4.19}]
\label{P: Omega and Schub}
For $w\in S_\Z$, $\bfS_w = \widetilde{\omega} (\bfS_{\neg w})$.
\end{pro}

\begin{cor}
\label{C: Omega}
For $w\in S_\Z$, 
$a^w_\lambda = \omega_1(a^{\neg w}_{\lambda'})$
and $j^w_\lambda(\y) = \omega_1(j^{\neg w}_{\lambda'}(\y))$.
\end{cor}
\begin{proof}
We expand the equation in Proposition~\ref{P: Omega and Schub}
into the double Schur basis:
$$
\sum_{\lambda} a^w_{\lambda}(\x;\y) s_\lambda(\x||\y)
=  \sum_{\lambda} \omega_1(a^{\neg w}_\lambda(\x;\y)) s_{\lambda'}(\x||\y).
$$
Extracting the coefficient of $s_\lambda(\x || \y)$ on both sides gives
the first equation in our corollary. 
The second equation follows
from the first since $\omega_1$ commutes
with the operator on $\Q(\x; \y)$ that sets $\x$ to $\y$.
\end{proof}

We end this section
with one observation on $\omega_1$.
\begin{lem}
\label{L: Omega}
We consider the action of $\omega_1$ 
on terms $(y_i-y_j)$ with $i\prec j$.
Then $\omega_1$ maps type~1 terms to type~2,
type~2 terms to type~1,
and type~3 terms to type~3.
\end{lem}
\begin{proof}
Direct computation.
\end{proof}

\subsection*{Bumpless pipedreams}
In the work of Lam, Lee, and Shimozono, 
BPDs were originally defined as certain tilings
of rectangular regions. 
We begin by defining BPDs
on a more general family of regions.

Define \definition{row $r$} and \definition{column $c$} as
\[
\row(r):=\{(r,i):i\in\Z\}, \qquad
\col(c):=\{(i,c):i\in\Z\}.
\]
For $\A \subseteq \Z \times \Z$,
let $R(\A)$ (resp.~$C(\A)$)
be the set of $i$ such that
$\row(i) \cap \A$ (resp.~$\col(i) \cap \A$) is nonempty. 

We say that $\A$ is a \definition{trapezoid}
if the following conditions hold:
\begin{itemize}
\item The set $R(\A)$ is an interval of integers.
\item For every $r \in R(\A)$,
the intersection $\A \cap \row(r)$ is an interval.
\item The intervals $\A \cap \row(r)$ form a chain as $r$
ranges over $R(\A)$ from smallest to largest (or vice versa).
\end{itemize}
Intuitively, a trapezoid is a connected set of cells
whose rows form contiguous intervals
that grow from top to bottom
or bottom to top.
\begin{rem}
    We adopt this definition for regions to include all rectangles, $\A^{\tri}$, $\A^{\tra}$ (see Definition~\ref{D: Tra}), as well as horizontal cuts (see Section~\ref{S: horizontal cut}) of these regions.
\end{rem}

\begin{defn}
\label{D: Tra}
We introduce notations for the following main infinite trapezoids:
\begin{itemize}
\item The whole plane: $\Z \times \Z$;
\item The upper half-plane: $H^\uparrow :=\Z_{\leq 0} \times \Z$;
\item The lower half-plane: 
$H^\downarrow := \Z_{>0} \times \Z$;
\item The ``triangular'' region
in $H^\downarrow$ above the diagonal:
$\A^{\tri}:= \{(i, j) : 0 < i < j\}$;
\item The ``trapezoid'' region
in $H^\downarrow$ weakly below the diagonal:
$\A^{\tra}:= \{(i, j) : 0 < i \textrm{ and }j \leq i\}$. \qedhere
\end{itemize}
\end{defn}
We also study trapezoids of finite size, such as
the square $[n] \times [n]$.

\medskip

A \definition{tiling} of a trapezoid $\A$ is a map 
$D:\A\longrightarrow
\{\btile,\htile,\vtile,\rtile,\jtile,\ptile\}$
where the six tiles 
are decorated by blue pipes.
We say that a tile has a pipe that 
\definition{connects to the right} (resp.~left, up, down) if a pipe is connected to the right (resp.~left, top, bottom) edge of this tile.
We say $D$ is \definition{consistent}
if pipes connect along adjacent edges:
\begin{itemize}
\item If $(r,c),(r,c+1)\in\A$, then
$D(r,c)$ has a pipe connecting to the right
iff $D(r,c+1)$ has a pipe connecting to the left;
\item If $(r,c),(r+1,c)\in\A$, then
$D(r,c)$ has a pipe connecting to the bottom
iff $D(r+1,c)$ has a pipe connecting to the top.
\end{itemize}
In a consistent tiling, pipes are traced from their entries on the right or top edges to their exits on the left or bottom edges.
When two pipes meet in a $\ptile$,
the pipe entering from the right exits to the left,
while the pipe entering from the top exits to the bottom;
in this case we say the two pipes \definition{cross}.

A consistent tiling $D$ is \definition{reduced} if
each pair of pipes crosses each other at most once.
A \definition{labeling} of $D$ is an assignment of distinct integers in $\Z$ to its pipes. 
We say a labeling is \definition{valid} 
if, at each crossing, the vertical pipe has a smaller label than the horizontal one.
Clearly, having a valid labeling implies
the tiling is reduced. 

We impose one last requirement
on tilings:

\begin{defn}
Let $D$ be a consistent tiling of a trapezoid $\A$.
For $r \in \Z$ with $\A\cap\row(r)\neq\varnothing$:
\begin{itemize}
\item If $\A\cap\row(r)$ has a rightmost cell $(r,c)$,
then a pipe \definition{enters from row $r$}
if it connects to the right edge of $(r,c)$.
\item Otherwise, a pipe \definition{enters from row $r$}
if it occupies $(r,c)$ for all sufficiently large $c$.
\end{itemize}
Definitions of \definition{entering from a column},
\definition{exiting from a row}, and
\definition{exiting from a column} are analogous.
We say $D$ is \definition{stabilized}
if every pipe in $D$ enters from a row 
and exits from a column.
\end{defn}

From now on we restrict to consistent, reduced, 
and stabilized tilings,
called \definition{bumpless pipedreams (BPDs)}.

\begin{exa}
The following are two BPDs of $\Z \times \Z$ (left) and  $\A^{\tra}$ (right). 
We only draw a finite region of these two BPDs, so for any $(i,j)$ not depicted, it is $\rtile$ if $i = j$,
$\vtile$ if $i > j$ and $\htile$ if $i < j$.
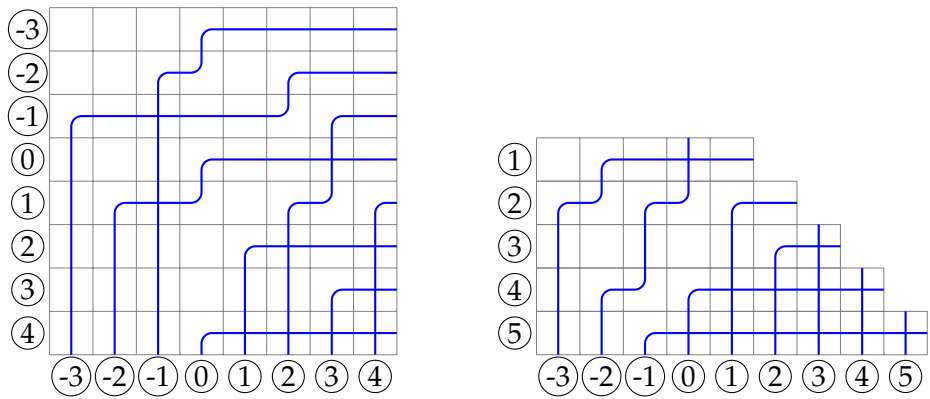
$$
    \begin{tikzpicture}[x=1.5em,y=1.5em,thick,rounded corners,color = blue]
        \draw[step=1,gray,ultra thin] (0,0) grid (8,8);
        \draw[color=blue, thick] (0.5,0)--(0.5,5.5)--(5.5,5.5)--(5.5,6.5)--(8,6.5);
        \draw[color=blue, thick] (1.5,0)--(1.5,3.5)--(3.5,3.5)--(3.5,4.5)--(8,4.5);
        \draw[color=blue, thick] (2.5,0)--(2.5,6.5)--(3.5,6.5)--(3.5,7.5)--(8,7.5);
        \draw[color=blue, thick] (3.5,0)--(3.5,0.5)--(8,0.5);
        \draw[color=blue, thick] (4.5,0)--(4.5,2.5)--(8,2.5);
        \draw[color=blue, thick] (5.5,0)--(5.5,3.5)--(6.5,3.5)--(6.5,5.5)--(8,5.5);
        \draw[color=blue, thick] (6.5,0)--(6.5,1.5)--(8,1.5);
        \draw[color=blue, thick] (7.5,0)--(7.5,3.5)--(8,3.5);
        \node[color=black] at (-0.5,7.5) {$\circled{-3}$};
        \node[color=black] at (-0.5,6.5) {$\circled{-2}$};
        \node[color=black] at (-0.5,5.5) {$\circled{-1}$};
        \node[color=black] at (-0.5,4.5) {$\circled{0}$};
        \node[color=black] at (-0.5,3.5) {$\circled{1}$};
        \node[color=black] at (-0.5,2.5) {$\circled{2}$};
        \node[color=black] at (-0.5,1.5) {$\circled{3}$};
        \node[color=black] at (-0.5,0.5) {$\circled{4}$};
        \node[color=black] at (0.5,-0.5) {$\circled{-3}$};
        \node[color=black] at (1.5,-0.5) {$\circled{-2}$};
        \node[color=black] at (2.5,-0.5) {$\circled{-1}$};
        \node[color=black] at (3.5,-0.5) {$\circled{0}$};
        \node[color=black] at (4.5,-0.5) {$\circled{1}$};
        \node[color=black] at (5.5,-0.5) {$\circled{2}$};
        \node[color=black] at (6.5,-0.5) {$\circled{3}$};
        \node[color=black] at (7.5,-0.5) {$\circled{4}$};
    \end{tikzpicture}
    \quad \quad \quad 
    \begin{tikzpicture}[x=1.5em,y=1.5em,thick,rounded corners,color = blue]
        \draw[step=1,gray,ultra thin] (0,0) grid (9,1);
        \draw[step=1,gray,ultra thin] (0,1) grid (8,2);
        \draw[step=1,gray,ultra thin] (0,2) grid (7,3);
        \draw[step=1,gray,ultra thin] (0,3) grid (6,4);
        \draw[step=1,gray,ultra thin] (0,4) grid (5,5);
        \draw[color=blue, thick] (0.5,0)--(0.5,3.5)--(1.5,3.5)--(1.5,4.5)--(5,4.5);
        \draw[color=blue, thick] (1.5,0)--(1.5,1.5)--(2.5,1.5)--(2.5,3.5)--(3.5,3.5)--(3.5,5);
        \draw[color=blue, thick] (2.5,0)--(2.5,0.5)--(9,0.5);
        \draw[color=blue, thick] (3.5,0)--(3.5,1.5)--(8,1.5);
        \draw[color=blue, thick] (4.5,0)--(4.5,3.5)--(6,3.5);
        \draw[color=blue, thick] (5.5,0)--(5.5,2.5)--(7,2.5);
        \draw[color=blue, thick] (6.5,0)--(6.5,3);
        \draw[color=blue, thick] (7.5,0)--(7.5,2);
        \draw[color=blue, thick] (8.5,0)--(8.5,1);
        \node[color=black] at (-0.5,4.5) {$\circled{1}$};
        \node[color=black] at (-0.5,3.5) {$\circled{2}$};
        \node[color=black] at (-0.5,2.5) {$\circled{3}$};
        \node[color=black] at (-0.5,1.5) {$\circled{4}$};
        \node[color=black] at (-0.5,0.5) {$\circled{5}$};
        \node[color=black] at (0.5,-0.5) {$\circled{-3}$};
        \node[color=black] at (1.5,-0.5) {$\circled{-2}$};
        \node[color=black] at (2.5,-0.5) {$\circled{-1}$};
        \node[color=black] at (3.5,-0.5) {$\circled{0}$};
        \node[color=black] at (4.5,-0.5) {$\circled{1}$};
        \node[color=black] at (5.5,-0.5) {$\circled{2}$};
        \node[color=black] at (6.5,-0.5) {$\circled{3}$};
        \node[color=black] at (7.5,-0.5) {$\circled{4}$};
        \node[color=black] at (8.5,-0.5) {$\circled{5}$};
    \end{tikzpicture}
$$
We put the row labels and column labels
near the rows and columns in a circle. 
\end{exa}

\medskip

\begin{defn}
\label{D: Boudnary condition}
A \definition{boundary condition} $\gamma$ 
of a region $\A$ is a tuple of four maps $(\gamma_N, 
\gamma_E, \gamma_S, \gamma_W)$,
where $\gamma_N, \gamma_S$ are 
from $C(\A)$ to $\Z \sqcup \{\varnothing\}$ and $\gamma_E, \gamma_W$ are 
from $R(\A)$ to $\Z \sqcup \{\varnothing\}$.

Let $D$ be a BPD of $\A$ with a valid labeling.
We say $D$ and this labeling \definition{satisfy} $\gamma$
if: 
\begin{itemize}
\item Row $r$ has a pipe entering (resp.~exiting)
if and only if $\gamma_E(r)$ (resp.~$\gamma_W(r)$) is not $\varnothing$.
In this case, this pipe has label
$\gamma_E(r)$ (resp.~$\gamma_W(r)$).
\item Column $c$ has a pipe entering (resp.~exiting)
if and only if $\gamma_N(c)$ (resp.~$\gamma_S(c)$) is not $\varnothing$.
In this case, this pipe has label
$\gamma_N(c)$ (resp.~$\gamma_S(c)$).
\end{itemize}

Now for a BPD $D$ of the region $\A$,
we say $D$ \definition{satisfies}
$\gamma$ if there exists a valid labeling
such that $D$ and the labeling satisfy $\gamma$.
We say this labeling is \definition{the
labeling of $D$ with respect to $\gamma$}.
Let $\BPD(\A, \gamma)$ 
be the set of all BPDs of the region $\A$
that satisfy $\gamma$.
\end{defn}
\begin{exa}
\label{Ex: boundary conditions}
The following is a BPD $D$ of the square 
$[6] \times [6]$.
It satisfies the boundary condition $\gamma$ and $\delta$.
We represent $\gamma$ by writing numbers on the corresponding side of the first square. 
For instance, by looking at number on top of the first square, 
we know $\gamma_N(2) = -3$, 
$\gamma_N(4) = 1$, $\gamma_N(6) = 0$
and $\gamma_N(i) = \varnothing$ if $i = 1, 3$ or $5$.
We represent $\delta$ by numbers on the side
of the second square.
    $$
    
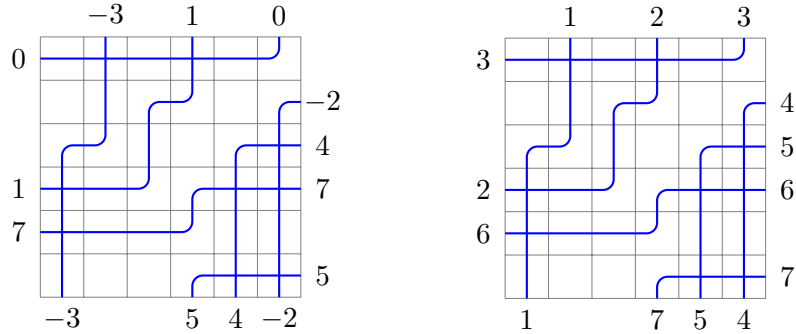
\begin{tikzpicture}[x=1.5em,y=1.5em,thick,rounded corners,color = blue]
        \draw[step=1,gray,ultra thin] (0,0) grid (6,6);
        \draw[color=blue, thick] (0,5.5)--(5.5,5.5)--(5.5,6);
        \draw[color=blue, thick] (0.5,0)--(0.5,3.5)--(1.5,3.5)--(1.5,6);
        \draw[color=blue, thick] (0,2.5)--(2.5,2.5)--(2.5,4.5)--(3.5,4.5)--(3.5,6);
        \draw[color=blue, thick] (0,1.5)--(3.5,1.5)--(3.5,2.5)--(6,2.5);
        \draw[color=blue, thick] (3.5,0)--(3.5,0.5)--(6,0.5);
        \draw[color=blue, thick] (4.5,0)--(4.5,3.5)--(6,3.5);
        \draw[color=blue, thick] (5.5,0)--(5.5,4.5)--(6,4.5);
        \node[color=black] at (6.5,4.5) {$-2$};
        \node[color=black] at (6.5,3.5) {$4$};
        \node[color=black] at (6.5,2.5) {$7$};
        \node[color=black] at (6.5,0.5) {$5$};
        \node[color=black] at (0.5,-0.5) {$-3$};
        \node[color=black] at (3.5,-0.5) {$5$};
        \node[color=black] at (4.5,-0.5) {$4$};
        \node[color=black] at (5.5,-0.5) {$-2$};
        \node[color=black] at (1.5,6.5) {$-3$};
        \node[color=black] at (3.5,6.5) {$1$};
        \node[color=black] at (5.5,6.5) {$0$};
        \node[color=black] at (-0.5,1.5) {$7$};
        \node[color=black] at (-0.5,2.5) {$1$};
        \node[color=black] at (-0.5,5.5) {$0$};
    \end{tikzpicture}\quad\quad\quad\quad
\begin{tikzpicture}[x=1.5em,y=1.5em,thick,rounded corners,color = blue]
        \draw[step=1,gray,ultra thin] (0,0) grid (6,6);
        \draw[color=blue, thick] (0,5.5)--(5.5,5.5)--(5.5,6);
        \draw[color=blue, thick] (0.5,0)--(0.5,3.5)--(1.5,3.5)--(1.5,6);
        \draw[color=blue, thick] (0,2.5)--(2.5,2.5)--(2.5,4.5)--(3.5,4.5)--(3.5,6);
        \draw[color=blue, thick] (0,1.5)--(3.5,1.5)--(3.5,2.5)--(6,2.5);
        \draw[color=blue, thick] (3.5,0)--(3.5,0.5)--(6,0.5);
        \draw[color=blue, thick] (4.5,0)--(4.5,3.5)--(6,3.5);
        \draw[color=blue, thick] (5.5,0)--(5.5,4.5)--(6,4.5);
        \node[color=black] at (6.5,4.5) {$4$};
        \node[color=black] at (6.5,3.5) {$5$};
        \node[color=black] at (6.5,2.5) {$6$};
        \node[color=black] at (6.5,0.5) {$7$};
        \node[color=black] at (0.5,-0.5) {$1$};
        \node[color=black] at (3.5,-0.5) {$7$};
        \node[color=black] at (4.5,-0.5) {$5$};
        \node[color=black] at (5.5,-0.5) {$4$};
        \node[color=black] at (1.5,6.5) {$1$};
        \node[color=black] at (3.5,6.5) {$2$};
        \node[color=black] at (5.5,6.5) {$3$};
        \node[color=black] at (-0.5,1.5) {$6$};
        \node[color=black] at (-0.5,2.5) {$2$};
        \node[color=black] at (-0.5,5.5) {$3$};
    \end{tikzpicture}
    $$
Since $D$ satisfies both $\gamma$ and $\delta$,
we have $D \in \BPD([6] \times[6], \gamma)$
and $D \in \BPD([6] \times[6], \delta)$.
In fact, one can verify $\BPD([6] \times[6], \gamma) = \BPD([6] \times[6], \delta)$.
\end{exa}

\begin{rem}
\label{R: label change}
As shown by Example~\ref{Ex: boundary conditions},
one BPD may satisfy multiple boundary conditions.
Let $D$ be a BPD of $\A$.
In fact, for any valid labeling of $D$,
we can find a unique boundary condition $\gamma$
satisfied by $D$ and this labeling. 
Then $D \in \BPD(\A, \gamma)$.

In other words, 
it is possible for 
different boundary conditions $\gamma, \delta$,
we get the same set of BPDs: $\BPD(\A, \gamma) = \BPD(\A, \gamma')$.
We may take any $D \in \BPD(\A, \gamma)$
and impose an arbitrary valid labeling on $D$.
Let $\delta$ be the boundary condition
satisfied by $D$ and this new labeling. 
We have $\BPD(\A, \gamma) = \BPD(\A, \delta)$.
\end{rem}

The classical setting of Lam-Lee-Shimozono \cite{LLS}
arises as follows.
\begin{defn}[\cite{LLS}]
\label{D: BPD for Schubert}
For $w\in S_\Z$,
let $\gamma^w$ be the following boundary condition:
\[
\gamma^w_N(i) = \gamma^w_W(i) = \varnothing,\quad
\gamma^w_E(i) = i, \quad
\gamma^w_S(i) = w^{-1}(i),
\]
for all $i \in \Z$.
Then define $\bBPD(w)$ as 
$\BPD(\Z \times \Z, \gamma^w)$.
If $u \in S_n$,
we let $\BPD_n(u) := \BPD([n] \times [n], \gamma^{u, n})$, where $\gamma^{u, n}$
is obtained by restricting 
the maps in $\gamma^u$
to $[n]$.
\end{defn}

Then recall from Theorem~\ref{T: LLS BPD} that 
Lam, Lee and Shimozono used 
$\BPD_n(u)$ (resp. $\bBPD(w)$) to give
combinatorial formulas for $\fS_u(\x; \y)$ 
(resp. $\bfS_w(\x; \y)$).

We end this section by listing a few basic properties
of BPDs whose proofs are left as exercises.
First, 
when the region is rectangular, 
the boundary conditions determine
whether certain pipes will cross.

\begin{lem}
\label{L: Boundary and crossing}
Let $\gamma$ be a boundary condition of $\A$
where $\A  = R \times C$ for some
$R, C \subseteq \Z$.
Say $\gamma_E(r) = \gamma_S(c) = r$ and $\gamma_E(r') = \gamma_S(c') = r'$ with $r < r'$.
Then in any $D \in \BPD(\A, \gamma)$,
pipe $r$ and pipe $r'$ cross 
if and only if $c > c'$.
\end{lem}

Next, notice that the descent set $\Des(w)$
gives us some information on $\bBPD(w)$.
\begin{lem}
\label{L: Fixed under Des}
Let $w \in S_\Z$ and set $m = \max(\Des(w))$.
Then all bumpless pipedreams in $\bBPD(w)$
agree under row~$m$,
and they have no blank tiles under this row.
\end{lem}

Finally, we recall the classical notion of 
the \definition{Rothe diagram} of $w \in S_\Z$.
It can be viewed as an element of $\bBPD(w)$.
In words, it is the BPD where the pipe
entering from row $r$ makes a ``left turn''
in column $w(r)$ and exits from that column.
More specifically, $\Rothe_w(i,j)$ 
is determined by the following five cases: 
\begin{itemize}
\item If $j = w(i)$, then $\Rothe_w(i,j) = \rtile$.
\item If $j < w(i)$ and $i< w^{-1}(j)$, 
then $\Rothe_w(i,j) = \btile$.
\item If $j < w(i)$ and $i > w^{-1}(j)$, 
then $\Rothe_w(i,j) = \vtile$.
\item If $j > w(i)$ and $i< w^{-1}(j)$, 
then $\Rothe_w(i,j) = \htile$.
\item If $j > w(i)$ and $i > w^{-1}(j)$, 
then $\Rothe_w(i,j) = \ptile$.
\end{itemize}

\begin{lem}
\label{L: Rothe}
The Rothe diagram $\Rothe_w$ is in
$\bBPD(w)$ for any $w \in S_\Z$.
\end{lem}
\begin{exa}
    The following is the Rothe diagram for $w = [-1, -3,1,2,-2,0] \in S_{[-3, 2]}$.
    We only depict the square $[-3, 2] \times [-3, 2]$.
    $$
    
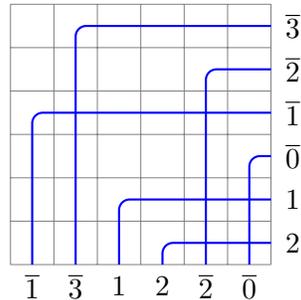
\begin{tikzpicture}[x=1.5em,y=1.5em,thick,rounded corners,color = blue]
        \draw[step=1,gray,ultra thin] (0,0) grid (6,6);
        \draw[color=blue, thick] (0.5,0)--(0.5,3.5)--(6,3.5);
        \draw[color=blue, thick] (1.5,0)--(1.5,5.5)--(6,5.5);
        \draw[color=blue, thick] (2.5,0)--(2.5,1.5)--(6,1.5);
        \draw[color=blue, thick] (3.5,0)--(3.5,0.5)--(6,0.5);
        \draw[color=blue, thick] (4.5,0)--(4.5,4.5)--(6,4.5);
        \draw[color=blue, thick] (5.5,0)--(5.5,2.5)--(6,2.5);
        \node[color=black] at (6.5,5.5) {$\overline{3}$};
        \node[color=black] at (6.5,4.5) {$\overline{2}$};
        \node[color=black] at (6.5,3.5) {$\overline{1}$};
        \node[color=black] at (6.5,2.5) {$\overline{0}$};
        \node[color=black] at (6.5,1.5) {$1$};
        \node[color=black] at (6.5,0.5) {$2$};
        \node[color=black] at (0.5,-0.5) {$\overline{1}$};
        \node[color=black] at (1.5,-0.5) {$\overline{3}$};
        \node[color=black] at (2.5,-0.5) {$1$};
        \node[color=black] at (3.5,-0.5) {$2$};
        \node[color=black] at (4.5,-0.5) {$\overline{2}$};
        \node[color=black] at (5.5,-0.5) {$\overline{0}$};
    \end{tikzpicture}
    $$
Again, we write numbers on the side
to represent the boundary condition
$\gamma^w$.
\end{exa}

\subsection*{Increasing chains}
We let $\leq_k$ denote the \definition{$k$-Bruhat order},
whose cover relation is 
$u \lessdot_k u\tau_{a,b}$ if 
$a \leq k < b$ and 
$\ell(u\tau_{a,b}) = \ell(u) + 1$.
Sottile~\cite{So96} introduced and studied
a distinguished family of chains in this order.

\begin{defn}\cite{So96}
\label{D: Inc chain}
Take a chain 
$u_1 \lessdot_k u_2 \lessdot_k \cdots \lessdot_k u_s$
and suppose $u_{i+1} = u_i\tau_{a_i,b_i}$.
This chain is an 
\definition{increasing $k$--chain} if 
\[
u_1(a_1)<u_2(a_2)<\cdots<u_s(a_s).
\]
In words, this means that the smaller value
being swapped increases along the chain.
We write \definition{$u \xrightarrow{k} w$} 
when such a chain from $u$ to $w$ exists.
\end{defn}

\begin{rem}
\label{R: distinct right points}
By~\cite{So96}, the relation $u \xrightarrow{k} w$ is equivalent to the existence of a chain
\[
u = u_1 \lessdot_k u_2 \lessdot_k \cdots \lessdot_k u_s = w,
\]
such that if we write $u_{i+1} = u_i \tau_{a_i,b_i}$, then the indices
$b_1,\dots,b_{s-1}$ are distinct.
\end{rem}

Consider $a, n \in \Z$ 
and $k \in [a, n)$.
Notice that the map $\std_{[a,n]}$
defined in~\eqref{EQ: define std}
is a poset isomorphism between $S_{[a,n]}$ under 
the $k$-Bruhat order
and $S_{n-a+1}$ under 
the $(k - a + 1)$-Bruhat order.
Moreover, for $u, w \in S_{[a,n]}$,
\begin{equation}
\label{EQ: k inc correspondence}
u \xrightarrow{k} w 
\quad \Longleftrightarrow \quad
\std_{[a,n]}(u) \xrightarrow{k - a  +1} \std_{[a,n]}(w).   
\end{equation}

Therefore, results on the $k$-Bruhat order
in $S_n$ can be extended to $S_{[a,n]}$.
For instance, we have the following:

\begin{lem}
\label{L: finiteness of inc chain}
For $a, n \in \Z$,
take $w \in S_{[a,n]}$. 
If $u \xrightarrow{k} w$, then $u \in S_{[a,n]}$.
In particular, for a fixed $w \in S_{[a,n]}$,
the number of increasing $k$-chains to $w$ is finite. 
\end{lem}

\begin{proof}
It suffices to assume $u \lessdot_k w$,
so $u = w \tau_{i,j}$. 
When $a = 1$, it is clear that 
$u \in S_n = S_{[a,n]}$.
The general case follows from
applying the $\std_{a,n}$ map. 
\end{proof}
From this lemma, we know that
for a fixed $w \in S_{[a,n]}$ and $k$,
the number of $u$ with $u \xrightarrow{k} w$
is finite. 
We also record several basic facts
for later use.

\begin{lem}
\label{L: Insert large number}
Consider $u, w \in S_{[a,n]}$.
Take $i, k \in [a,n)$.
Construct $u', w' \in S_{[a, n+1]}$
where the one-line notation of $u'$ 
(resp. $w'$) is obtained from that of $u$
(resp. $w$) by inserting $n+1$ after 
$u(i)$ (resp. $w(i)$).
If $i \leq k$,
then 
\[
u \xrightarrow{k} w\quad \Longleftrightarrow \quad u' \xrightarrow{k  +1} w', \qquad \textrm{ and }\fix_{(k, n]}(u,w) = \fix_{(k+1,n+1]}(u',w').
\]
If $i \geq k$,
then 
\[
u \xrightarrow{k} w\quad \Longleftrightarrow \quad u' \xrightarrow{k} w', \qquad \textrm{ and }\fix_{(k, n]}(u,w) \sqcup \{n+1\} = \fix_{(k,n+1]}(u',w').
\]
\end{lem}
\begin{proof}
Follows from Definition~\ref{D: Inc chain}.
\end{proof}

\begin{lem}
\label{L: LIRD}
If $u \leq_k w$,
then $u(i) \leq w(i)$ 
for $i \leq k$
and $u(j) \geq w(j)$
for $j > k$
\end{lem}
\begin{proof}
It is enough to show the lemma with
$u \lessdot_k w$.
We have $w = u\tau_{a,b}$ such that
$a \leq k < b$ and $u(a) < u(b)$.
Our statement follows.
\end{proof}

\begin{lem}
\label{L: not fixed point}
Say $u, w \in S_{[a,n]}$ and
$u \xrightarrow{k} w$.
If $i \in u((k, n])$ but $i \notin \fix_{(k, n]}(u,w)$, then there exists $t \in [a,k]$
such that $w(t) \geq i$.
\end{lem}
\begin{proof}
Let $u = u_1 \lessdot_k \cdots \lessdot_k u_s = w$ be the increasing $k$-chain 
from $u$ to $w$.
Each $u_j$ is obtained from $u_{j-1}$
by swapping two numbers. 
Since $i \notin \fix_{(k, n]}(u,w)$, 
the value $i$ is swapped in the chain. 
Since $i \in u((k, n])$, 
$i$ will be swapped to position $t$
with $t \in [a,k]$.
Thus, $u_j(t) = i$ for some $j \in [2,n]$.
Then $w(t) \geq u_j(t) = i$ by
Lemma~\ref{L: LIRD}.
\end{proof}

Finally, for $I \subset \Z$, 
we say that $w \in S_\Z$ is \definition{decreasing 
on $I$} if $w(i) > w(i')$
whenever $i < i'$ and $i, i' \in I$.
Then the following fact is proved in~\cite{Yu}.
\begin{lem}\cite{Yu}*{Lemma 3.7}
\label{L: decreasing preserved}
Consider $w \in S_{[a,n]}$
that is decreasing on $(k, n]$.
If $u \xrightarrow{k} w$,
then $u$ is also
decreasing on $(k,n]$.
\end{lem}

\subsection*{Generating function of increasing chains
and its symmetry}
Let $\alpha = (\alpha_1, \dots, \alpha_{m+1})$
be a tuple of integers.
For $u, w \in S_{[a,n]}$,
recall \definition{$C(u, w, \alpha)$} denotes the set of
sequences of permutations 
$(u_1, u_2, \dots, u_{m+1})$
such that $u_1 = u$, $u_{m+1} = w$,
and $u_i \xrightarrow{\alpha_i} u_{i+1}$ for each $i \in [m]$.
Each such $(u_1, \dots, u_{m+1})$ is assigned a \definition{weight}
\begin{equation}
\label{EQ: Define wt}
\wt^n_\alpha(u_1, \dots, u_{m+1})(x_1, \dots, x_m;\y)
:= 
\prod_{i\in[m]} 
\;\prod_{j\in \fix_{(\alpha_i, n]}(u_i,u_{i+1})} (x_i - y_j),
\end{equation}
where $\fix_I(u,w):=\{u(t): t\in I,\ u(t)=w(t)\}$ for $I\subset\Z$.
Recall that we define the generating function 
of $C(u,w,\alpha)$ as
$$
\fC^n_{u, w, \alpha}(x_1, \dots, x_{m}; \y)
:= 
\sum_{(u_1, \dots, u_{m+1}) \in C(u, w, \alpha)}
\wt^n_\alpha(u_1, \dots, u_{m+1})(x_1, \dots, x_m; \y).$$

\begin{exa}
We often represent 
an element $(u_1, \dots, u_m) \in C(u,w,\alpha)$
as a table. 
We assume $u_1, \dots, u_m \in S_{[a,n]}$ for some 
$a, n \in S_\Z$.
Then we write the one-line
notations of $u_1, \dots, u_m$ from bottom to top. 
To include the information of 
$\alpha$, 
we put a red bar after
$u_{i}(\alpha_i)$ and $u_{i+1}(\alpha_i)$.
Then $\fix_{(\alpha_i, n]}(u_i, u_{i+1})$ is the set
of numbers that match in the 
rows of $u_i$ and $u_{i+1}$
on the right side of the red bar. 
We circle these numbers.
For instance, the following table
represents $(u_1, u_2, u_3) \in C(u,w, \alpha)$:
\[
\begin{tikzpicture}[scale=0.7]
  \draw (0,0) rectangle (7,3);
  \foreach \x in {1,...,6} \draw (\x,0) -- (\x,3);
  \foreach \y in {1,2} \draw (0,\y) -- (7,\y);

  \node at (0.5,0.5) {-1};
  \node at (1.5,0.5) {0};
  \node at (2.5,0.5) {4};
  \node at (3.5,0.5) {1};
  \node at (4.5,0.5) {5};
  \node at (5.5,0.5) {3};
  \node at (6.5,0.5) {2};

  \node at (0.5,1.5) {3};
  \node at (1.5,1.5) {-1};
  \node at (2.5,1.5) {4};
  \node at (3.5,1.5) {0};
  \node at (4.5,1.5) {5};
  \node at (5.5,1.5) {1};
  \node at (6.5,1.5) {2};

  \node at (0.5,2.5) {3};
  \node at (1.5,2.5) {0};
  \node at (2.5,2.5) {5};
  \node at (3.5,2.5) {-1};
  \node at (4.5,2.5) {4};
  \node at (5.5,2.5) {1};
  \node at (6.5,2.5) {2};

  \draw[red,ultra thick] (1,0) -- (1,2); 
  \draw[red,ultra thick] (3,1) -- (3,3);
  \draw[green,thick] (2.5,1) circle [x radius=0.35, y radius=0.8];
  \draw[green,thick] (4.5,1) circle [x radius=0.35, y radius=0.8];
  \draw[green,thick] (6.5,1) circle [x radius=0.35, y radius=0.8];
  \draw[green,thick] (5.5,2) circle [x radius=0.35, y radius=0.8];
  \draw[green,thick] (6.5,2) circle [x radius=0.35, y radius=0.8];
\end{tikzpicture}
\]
We know $u = u_1 = [-1,0,4,1,5,3,2]$,
$u_2 = [3,-1,4,0,5,1,2]$ 
and $u_3 = [3,0,5,-1,4,1,2]$.
From the table, we can also read
$\alpha = (-1, 1)$ and 
$$
\wt_\alpha^5(u_1, u_2, u_3)(x_1, x_2; \y)
= (x_1 - y_5) (x_1 - y_4) (x_1 - y_2) (x_2 - y_1)(x_2 - y_2). \qedhere
$$
\end{exa}

The generating function $\fC_{u, w, \alpha}^n(x_1, \dots, x_m;\y)$ has a remarkable symmetry 
that will play a major role in \S\ref{S: TBPD}.
For $\sigma \in S_m$ 
and a sequence 
$\alpha = (\alpha_1, \dots, \alpha_m)$,
let $\sigma\alpha$ denote
$(\alpha_{\sigma(1)}, \dots, \alpha_{\sigma(m)})$.

\begin{prop}
\label{P: double symmetry} 
For $u, w \in S_{[a,n]}$ and 
$\alpha=(\alpha_1,\dots,\alpha_m)$,
we have
\begin{equation}
\label{EQ: double symmetry prop}
\fC^n_{u,w,\alpha}(x_1,\dots,x_m;\y)
=
\fC^n_{u,w,\sigma\alpha}(x_{\sigma(1)},\dots,x_{\sigma(m)};\y)
\end{equation}
In particular, when $\sigma=[m,\dots,2,1]$,
\begin{equation}
\label{EQ: double rev 2} 
\fC^n_{u,w,\alpha}(x_1,\dots,x_m;\y)
=
\fC^n_{u,w,\rev(\alpha)}(x_m,\dots,x_1;\y),
\end{equation}
where $\rev(\alpha)$ denotes sequence reversal.
\end{prop}

\begin{exa}
\label{E: compute 2143 with chains}
Let $u = [2,1,4,3], w = [4,3,2,1]$ and $\alpha = (1,2,3)$. There are $3$ chains in $C(u, w, \alpha)$, represented by the following three tables:
    $$
    \begin{tikzpicture}[scale=0.7]
      \draw (0,0) rectangle (4,4);
      \foreach \x in {1,...,4} \draw (\x,0) -- (\x,4);
      \foreach \y in {1,...,4} \draw (0,\y) -- (4,\y);
      \node at (0.5,0.5) {2};
      \node at (1.5,0.5) {1};
      \node at (2.5,0.5) {4};
      \node at (3.5,0.5) {3};
      \node at (0.5,1.5) {4};
      \node at (1.5,1.5) {1};
      \node at (2.5,1.5) {3};
      \node at (3.5,1.5) {2};
      \node at (0.5,2.5) {4};
      \node at (1.5,2.5) {3};
      \node at (2.5,2.5) {1};
      \node at (3.5,2.5) {2};
      \node at (0.5,3.5) {4};
      \node at (1.5,3.5) {3};
      \node at (2.5,3.5) {2};
      \node at (3.5,3.5) {1};
      \draw[red,ultra thick] (2,1) -- (2,3); 
      \draw[red,ultra thick] (3,2) -- (3,4);
      \draw[red,ultra thick] (1,0) -- (1,2);
      \draw[green,thick] (1.5,1) circle [x radius=0.35, y radius=0.8];
      \draw[green,thick] (3.5,2) circle [x radius=0.35, y radius=0.8];
    \end{tikzpicture}
\quad \quad
    \begin{tikzpicture}[scale=0.7]
      \draw (0,0) rectangle (4,4);
      \foreach \x in {1,...,4} \draw (\x,0) -- (\x,4);
      \foreach \y in {1,...,4} \draw (0,\y) -- (4,\y);
      \node at (0.5,0.5) {2};
      \node at (1.5,0.5) {1};
      \node at (2.5,0.5) {4};
      \node at (3.5,0.5) {3};
      \node at (0.5,1.5) {4};
      \node at (1.5,1.5) {1};
      \node at (2.5,1.5) {2};
      \node at (3.5,1.5) {3};
      \node at (0.5,2.5) {4};
      \node at (1.5,2.5) {3};
      \node at (2.5,2.5) {1};
      \node at (3.5,2.5) {2};
      \node at (0.5,3.5) {4};
      \node at (1.5,3.5) {3};
      \node at (2.5,3.5) {2};
      \node at (3.5,3.5) {1};
      \draw[red,ultra thick] (2,1) -- (2,3); 
      \draw[red,ultra thick] (3,2) -- (3,4);
      \draw[red,ultra thick] (1,0) -- (1,2);
      \draw[green,thick] (1.5,1) circle [x radius=0.35, y radius=0.8];
      \draw[green,thick] (3.5,1) circle [x radius=0.35, y radius=0.8];
    \end{tikzpicture}
    \quad \quad
    \begin{tikzpicture}[scale=0.7]
      \draw (0,0) rectangle (4,4);
      \foreach \x in {1,...,4} \draw (\x,0) -- (\x,4);
      \foreach \y in {1,...,4} \draw (0,\y) -- (4,\y);
      \node at (0.5,0.5) {2};
      \node at (1.5,0.5) {1};
      \node at (2.5,0.5) {4};
      \node at (3.5,0.5) {3};
      \node at (0.5,1.5) {4};
      \node at (1.5,1.5) {1};
      \node at (2.5,1.5) {3};
      \node at (3.5,1.5) {2};
      \node at (0.5,2.5) {4};
      \node at (1.5,2.5) {3};
      \node at (2.5,2.5) {2};
      \node at (3.5,2.5) {1};
      \node at (0.5,3.5) {4};
      \node at (1.5,3.5) {3};
      \node at (2.5,3.5) {2};
      \node at (3.5,3.5) {1};
      \draw[red,ultra thick] (2,1) -- (2,3); 
      \draw[red,ultra thick] (3,2) -- (3,4);
      \draw[red,ultra thick] (1,0) -- (1,2);
      \draw[green,thick] (1.5,1) circle [x radius=0.35, y radius=0.8];
      \draw[green,thick] (3.5,3) circle [x radius=0.35, y radius=0.8];
    \end{tikzpicture}
    $$
    Thus, the generating function is
    \begin{align}
    \label{EQ: Double sym Example 1}
        \fC^4_{u, w, (1,2,3)}(x_1, x_2, x_3; \y) = (x_1 - y_1)(x_2 - y_2) + (x_1 - y_1)(x_1 - y_3) + (x_1 - y_1)(x_3 - y_1).
    \end{align}

    Notice that $\rev(\alpha) = (3,2,1)$. There are also $3$ chains in $C(u, w, (3,2,1))$, represented by the following three tables:
    $$
    \begin{tikzpicture}[scale=0.7]
      \draw (0,0) rectangle (4,4);
      \foreach \x in {1,...,4} \draw (\x,0) -- (\x,4);
      \foreach \y in {1,...,4} \draw (0,\y) -- (4,\y);
      \node at (0.5,0.5) {2};
      \node at (1.5,0.5) {1};
      \node at (2.5,0.5) {4};
      \node at (3.5,0.5) {3};
      \node at (0.5,1.5) {2};
      \node at (1.5,1.5) {1};
      \node at (2.5,1.5) {4};
      \node at (3.5,1.5) {3};
      \node at (0.5,2.5) {2};
      \node at (1.5,2.5) {4};
      \node at (2.5,2.5) {3};
      \node at (3.5,2.5) {1};
      \node at (0.5,3.5) {4};
      \node at (1.5,3.5) {3};
      \node at (2.5,3.5) {2};
      \node at (3.5,3.5) {1};
      \draw[red,ultra thick] (3,0) -- (3,2); 
      \draw[red,ultra thick] (2,1) -- (2,3);
      \draw[red,ultra thick] (1,2) -- (1,4);
      \draw[green,thick] (3.5,1) circle [x radius=0.35, y radius=0.8];
      \draw[green,thick] (3.5,3) circle [x radius=0.35, y radius=0.8];
    \end{tikzpicture}
\quad \quad
    \begin{tikzpicture}[scale=0.7]
      \draw (0,0) rectangle (4,4);
      \foreach \x in {1,...,4} \draw (\x,0) -- (\x,4);
      \foreach \y in {1,...,4} \draw (0,\y) -- (4,\y);
      \node at (0.5,0.5) {2};
      \node at (1.5,0.5) {1};
      \node at (2.5,0.5) {4};
      \node at (3.5,0.5) {3};
      \node at (0.5,1.5) {3};
      \node at (1.5,1.5) {1};
      \node at (2.5,1.5) {4};
      \node at (3.5,1.5) {2};
      \node at (0.5,2.5) {3};
      \node at (1.5,2.5) {4};
      \node at (2.5,2.5) {2};
      \node at (3.5,2.5) {1};
      \node at (0.5,3.5) {4};
      \node at (1.5,3.5) {3};
      \node at (2.5,3.5) {2};
      \node at (3.5,3.5) {1};
      \draw[red,ultra thick] (3,0) -- (3,2); 
      \draw[red,ultra thick] (2,1) -- (2,3);
      \draw[red,ultra thick] (1,2) -- (1,4);
      \draw[green,thick] (2.5,3) circle [x radius=0.35, y radius=0.8];
      \draw[green,thick] (3.5,3) circle [x radius=0.35, y radius=0.8];
    \end{tikzpicture}
    \quad \quad
    \begin{tikzpicture}[scale=0.7]
      \draw (0,0) rectangle (4,4);
      \foreach \x in {1,...,4} \draw (\x,0) -- (\x,4);
      \foreach \y in {1,...,4} \draw (0,\y) -- (4,\y);
      \node at (0.5,0.5) {2};
      \node at (1.5,0.5) {1};
      \node at (2.5,0.5) {4};
      \node at (3.5,0.5) {3};
      \node at (0.5,1.5) {2};
      \node at (1.5,1.5) {3};
      \node at (2.5,1.5) {4};
      \node at (3.5,1.5) {1};
      \node at (0.5,2.5) {2};
      \node at (1.5,2.5) {4};
      \node at (2.5,2.5) {3};
      \node at (3.5,2.5) {1};
      \node at (0.5,3.5) {4};
      \node at (1.5,3.5) {3};
      \node at (2.5,3.5) {2};
      \node at (3.5,3.5) {1};
      \draw[red,ultra thick] (3,0) -- (3,2); 
      \draw[red,ultra thick] (2,1) -- (2,3);
      \draw[red,ultra thick] (1,2) -- (1,4);
      \draw[green,thick] (3.5,2) circle [x radius=0.35, y radius=0.8];
      \draw[green,thick] (3.5,3) circle [x radius=0.35, y radius=0.8];
    \end{tikzpicture}
    $$
    Therefore,the generating function is
    \begin{align}
    \label{EQ: Double sym Example 2}
        \fC^4_{u,w,\rev(\alpha)}(x_3, x_2, x_1; \y) = (x_1 - y_1)(x_3 - y_3) + (x_1 - y_1)(x_1 - y_2) + (x_1 - y_1)(x_2 - y_1).
    \end{align}
    One can compare~\eqref{EQ: Double sym Example 1} 
    and~\eqref{EQ: Double sym Example 2} to
    confirm that $\fC^4_{u,w,\alpha}(x_1, x_2, x_3; \y) = \fC^4_{u,w,\rev(\alpha)}(x_3, x_2, x_1; \y)$,
    which is expected by Proposition~\ref{P: double symmetry}.
    Notice that even though the generating functions agree,
    there does not exist a weight-preserving bijection
    between $C(u,w,(1,2,3))$ and $C(u,w,(3,2,1))$.
\end{exa}

This symmetry can be implied by the Pieri identity
of Samuel~\cite{Sam}.
We recall the following analogue of the notion 
$u\xrightarrow{k}w$ that was introduced by Sottile~\cite{So96}
and played an important role in~\cite{Sam}.
\begin{defn}
\label{D: E chains}
For permutations $u, w \in S_n$,
write\footnote{This is denoted as $u \xrightarrow{k} w$
in~\cite{Sam}.} $u\xrightarrow[k]{}w$
if there is a chain 
\[
u = u_1 \lessdot_k u_2 \lessdot_k \cdots \lessdot_k u_s = w
\]
such that if we write $u_{i+1} = u_i \tau_{a_i,b_i}$ for each $i$, then the indices
$a_1,\dots,a_{s-1}$ are pairwise distinct.
\end{defn}

This notion is related to the relation $u \xrightarrow{k} w$ as follows.
Recall that $w_{0,n} = [n,n-1,\dots,2,1]$.

\begin{lem}
\label{L: equivalence of two increasing defs}   
For $u,w \in S_n$, we have
\[
u \xrightarrow{k} w
\quad\Longleftrightarrow\quad
w\, w_{0,n} \xrightarrow{\,n-k\,} u\, w_{0,n}.
\]
\end{lem}

\begin{proof}
Suppose $u \xrightarrow{k} w$.
Let
\[
u = u_1 \lessdot_k u_2 \lessdot_k \cdots \lessdot_k u_s = w
\]
be a chain satisfying Definition~\ref{D: E chains}.
Write $u_{i+1} = u_i \tau_{a_i,b_i}$ for each $i \in [s-1]$,
so that the indices $a_1,\dots,a_{s-1}$ are distinct and
$a_i \le k < b_i$.
Then we obtain a chain from $w w_{0,n}$ to $u w_{0,n}$:
\[
u_s w_{0,n} \lessdot_{n-k} \cdots \lessdot_{n-k} u_1 w_{0,n},
\]
where
\[
u_i w_{0,n}
= (u_{i+1} w_{0,n}) \tau_{\,n+1-b_i,\;n+1-a_i}.
\]
Since
\[
n+1-b_i \le n-k < n+1-a_i,
\]
and the indices $n+1-a_1,\dots,n+1-a_{s-1}$ are distinct,
Remark~\ref{R: distinct right points} implies
\[
w w_{0,n} \xrightarrow{\,n-k\,} u w_{0,n}.
\]

The converse direction is proved by the same argument applied in reverse.
\end{proof}

Now we introduce the Pieri rule of Samuel. 

\begin{thm}\cite{Sam}*{Theorem~7.1}
\label{T: Sam}
For $u \in S_n$,
$$
\fS_u(\x; \y) \times (x_1 - z)(x_2 - z) \cdots (x_k - z)
= \sum_{u \xrightarrow[k]{}w} \fS_w(\x; \y) \prod_{j \in \fix_{[k]}(u,w)} (y_{u(j)} - z).
$$   
\end{thm}

\begin{cor}
\label{C: chain-sym-two-steps}
For $u,w \in S_n$ and $k_1,k_2 \in [n-1]$, we have
\begin{align*}
& \sum_{u \xrightarrow[k_1]{} v \xrightarrow[k_2]{} w}
\left(\prod_{j \in \fix_{[k_1]}(u,v)} (y_{j} - z)\right)
\left(\prod_{j \in \fix_{[k_2]}(v,w)} (y_{j} - z')\right) \\
= \;&
\sum_{u \xrightarrow[k_2]{} v \xrightarrow[k_1]{} w}
\left(\prod_{j \in \fix_{[k_2]}(u,v)} (y_{j} - z')\right)
\left(\prod_{j \in \fix_{[k_1]}(v,w)} (y_{j} - z)\right).
\end{align*}
\end{cor}

\begin{proof}
Recall that the double Schubert polynomials $\fS_w(\x;\y)$ form a basis of $\ZZ[\y][\x]$ as polynomials in the $\x$-variables.
By Theorem~\ref{T: Sam}, the left-hand side is the coefficient of $\fS_w(\x;\y)$ in
\[
\fS_u(\x;\y)\,
\bigl((x_1-z)\cdots(x_{k_1}-z)\bigr)\,
\bigl((x_1-z')\cdots(x_{k_2}-z')\bigr),
\]
while the right-hand side is the coefficient of $\fS_w(\x;\y)$ in
\[
\fS_u(\x;\y)\,
\bigl((x_1-z')\cdots(x_{k_2}-z')\bigr)\,
\bigl((x_1-z)\cdots(x_{k_1}-z)\bigr).
\]
These two products are equal, so the coefficients of $\fS_w(\x;\y)$ agree.
\end{proof}

\begin{proof}[Proof of Proposition~\ref{P: double symmetry}]
To prove~\eqref{EQ: double symmetry prop}, it suffices to assume $\alpha=(k_1,k_2)$.
Moreover, by~\eqref{EQ: k inc correspondence}, we may assume $u,w\in S_n$ and $k_1,k_2\in[n-1]$.
After a substitution of variables, our goal is to show
\[
\fC^n_{u,w,(k_1,k_2)}(z,z';\y)
=
\fC^n_{u,w,(k_2,k_1)}(z',z;\y).
\]
By Lemma~\ref{L: equivalence of two increasing defs}, this is equivalent to
\begin{align*}
& \sum_{w w_{0,n} \xrightarrow[n-k_2]{} v \xrightarrow[n-k_1]{} u w_{0,n}}
\left(\prod_{j \in \fix_{[n-k_2]}(w w_{0,n},v)} (z' - y_j)\right)
\left(\prod_{j \in \fix_{[n-k_1]}(v,u w_{0,n})} (z - y_j)\right) \\
= \;&
\sum_{w w_{0,n} \xrightarrow[n-k_1]{} v \xrightarrow[n-k_2]{} u w_{0,n}}
\left(\prod_{j \in \fix_{[n-k_1]}(w w_{0,n},v)} (z - y_j)\right)
\left(\prod_{j \in \fix_{[n-k_2]}(v,u w_{0,n})} (z' - y_j)\right).
\end{align*}
This follows from Corollary~\ref{C: chain-sym-two-steps}.
\end{proof}

\subsection*{Bumpless pipedreams and increasing chains}

We have introduced two combinatorial models:
BPDs and increasing chains.
Yu~\cite{Yu} connected these
two models by giving a bijection
\[
\chain_n: \BPD_n(u) \longrightarrow C\bigl(u, w_{0,n}, (n-1, \dots, 2, 1)\bigr).
\]
We now describe this bijection in our setup. 

\begin{defn}
\label{D: cross section}   
Let $D$ be a BPD with a valid labeling.
Take $r \in \Z$
and let $I_r \subset \Z$ consist of 
all $c$ such that $(r,c)$ or $(r+1, c)$
is a tile in $D$.
The \definition{cross section} of $D$
below row $r$ (or above row $r + 1$)
with respect to this labeling 
is a map $\varphi: I_r \rightarrow \Z \cup \{\varnothing\}$.
If there is a pipe with label $p$ connecting
to the bottom edge of $(r,c)$ or the top
edge of $(r+1, c)$, $\varphi(c) = p$;
otherwise, $\varphi(c) = \varnothing$.

For a cross section $\varphi$, let $\varphi^{-1}$ denote its ``inverse'', and let $\img(\varphi) = \{\varphi(c)\,:\,\varphi(c) \in \Z, \, c \in I_r\}$. Then $\varphi^{-1}$ would be a map from $\img(\varphi)$ to $\Z$.
We define $\varphi^{-1}(p) = c$ if $\varphi(c) = p$.
In words, 
$\varphi^{-1}(p)$ is the column index of the pipe
$p$ at the bottom of row $r$.
\end{defn}

\begin{rem}
\label{R: Top row bottom row cs}
Take $D \in \BPD(\A, \gamma)$ where $R(\A) = [a,b]$ and consider the labeling of $D$
with respect to $\gamma$.
Then the cross section of $D$ above row $a$ 
(resp. below row $b$) must be
$\gamma_N$ (resp. $\gamma_S$). 
In particular, they are the same
for all $D \in \BPD(\A, \gamma)$.
\end{rem}

\begin{exa}
Let $D$ be the BPD in 
Example~\ref{Ex: boundary conditions} with labeling as
in the first square. 
Let $\varphi_r$ be the 
cross section of $D$ above row
$r$ for $r \in [7]$. 
We list $\varphi_r(c)$ in the 
following table where 
rows correspond to $r$
and columns correspond to $c$.
\begin{center}
\begin{tabular}{c|cccccc}
  $r \backslash c$ & $1$ & $2$ & $3$ & $4$ & $5$ & $6$ \\ \hline
  $1$ & $\varnothing$ & -3 & $\varnothing$ & 1 & $\varnothing$ & 0 \\
  $2$ & $\varnothing$ & -3 & $\varnothing$ & 1 & $\varnothing$ & $\varnothing$ \\
  $3$ & $\varnothing$ & -3 & 1 & $\varnothing$ & $\varnothing$ & -2 \\
  $4$ & -3 & $\varnothing$ & 1 & $\varnothing$ & 4 & -2 \\
  $5$ & -3 & $\varnothing$ & $\varnothing$ & 7 & 4 & -2 \\
  $6$ & -3 & $\varnothing$ & $\varnothing$ & $\varnothing$ & 4 & -2 \\
  $7$ & -3 & $\varnothing$ & $\varnothing$ & 5 & 4 & -2 
\end{tabular}
\end{center}
One can also read $\varphi_r^{-1}$.
For instance, 
$\varphi_3^{-1}$ is defined
on \{-3, 1, 2\} with
$\varphi_3^{-1}(-3) = 2$, $\varphi_3^{-1}(1) = 3$
and $\varphi_3^{-1}(-2) = 6$.

\end{exa}

\begin{defn}\cite{Yu}*{Definition~3.4}
\label{D: BPD chain}
Take $D \in \BPD_n(u)$ with the labeling with respect to $\gamma^{w,n}$ (see Definition~\ref{D: BPD for Schubert}).
For each $r \in [n+1]$,
let $\varphi_r$ be the cross section of $D$
above row $r$,
so $\img(\varphi_r) = [r-1]$.
Define $u_r \in S_n$ as a permutation 
that agrees with $\varphi_r^{-1}$ on $[r-1]$
and decreases on $[r, n]$.
Finally, define \definition{$\chain_n(D)$} as $(u_n, \dots, u_1)$. 
\end{defn}

\begin{thm}\cite{Yu}*{Theorem~3.6}
\label{T: BPD chain}
The map $\chain_n$ 
is a bijection from $\BPD_n(u)$ to $C\bigl(u, w_{0,n}, (n-1, \dots, 2, 1)\bigr)$.
If $\chain_n(D) = (u_{n}, \dots, u_1)$,
then $\wt(D) = \wt_{\alpha}^n(\x_{n-1}, \dots, x_1; \y)$.
\end{thm}

\begin{exa}\cite{Yu}*{Example 3.5}
\label{E: BPD to chain}
Take $n = 6$ and $D \in \BPD_n(w)$
with $w = [2,1,6,5,3,4]$.
We depict $D$ below on the left
together with the boundary condition $\gamma^{w,n}$.
We represent $\chain_n(D) \in C(u, w_{0, n}, (5,4,3,2,1))$ on 
the right.
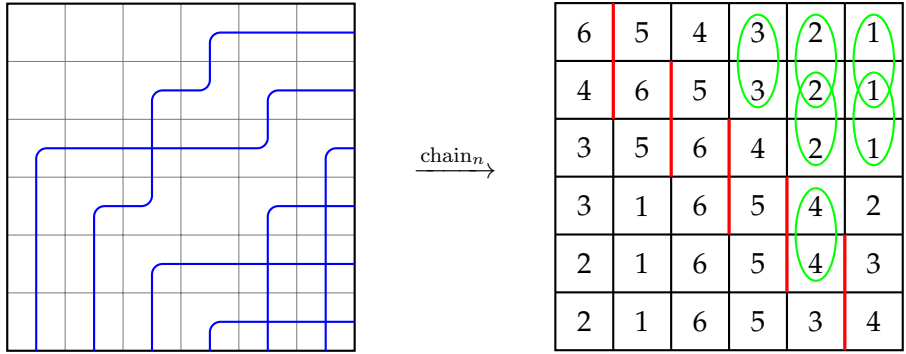
\[
\begin{tikzpicture}[x=2em,y=2em,thick,rounded corners, color = blue]
\draw[step=1,gray,thin] (0,0) grid (6,6);
\draw[color=black, thick, sharp corners] (0,0) rectangle (6,6);
\draw(0.5, 0)--(0.5,3.5)--(4.5,3.5)--(4.5,4.5)--(6,4.5);
\draw(1.5, 0)--(1.5,2.5)--(2.5,2.5)--(2.5,4.5)--(3.5,4.5)--(3.5,5.5)--(6,5.5);
\draw(2.5, 0)--(2.5,1.5)--(6,1.5);
\draw(3.5, 0)--(3.5,0.5)--(6,0.5);
\draw(4.5, 0)--(4.5,2.5)--(6,2.5);
\draw(5.5, 0)--(5.5,3.5)--(6,3.5);
\end{tikzpicture}
\raisebox{6em}{$\qquad \xrightarrow{\chain_n}\qquad$}
\begin{tikzpicture}
[x=2em,y=2em,thick,sharp corners, color = black]
  \draw (0,0) rectangle (6,6);
  \foreach \x in {1,...,5} \draw (\x,0) -- (\x,6);
  \foreach \y in {1,...,5} \draw (0,\y) -- (6,\y);

  \node at (0.5,0.5) {2};
  \node at (1.5,0.5) {1};
  \node at (2.5,0.5) {6};
  \node at (3.5,0.5) {5};
  \node at (4.5,0.5) {3};
  \node at (5.5,0.5) {4};

  \node at (0.5,1.5) {2};
  \node at (1.5,1.5) {1};
  \node at (2.5,1.5) {6};
  \node at (3.5,1.5) {5};
  \node at (4.5,1.5) {4};
  \node at (5.5,1.5) {3};

  \node at (0.5,2.5) {3};
  \node at (1.5,2.5) {1};
  \node at (2.5,2.5) {6};
  \node at (3.5,2.5) {5};
  \node at (4.5,2.5) {4};
  \node at (5.5,2.5) {2};

  \node at (0.5,3.5) {3};
  \node at (1.5,3.5) {5};
  \node at (2.5,3.5) {6};
  \node at (3.5,3.5) {4};
  \node at (4.5,3.5) {2};
  \node at (5.5,3.5) {1};

  \node at (0.5,4.5) {4};
  \node at (1.5,4.5) {6};
  \node at (2.5,4.5) {5};
  \node at (3.5,4.5) {3};
  \node at (4.5,4.5) {2};
  \node at (5.5,4.5) {1};

  \node at (0.5,5.5) {6};
  \node at (1.5,5.5) {5};
  \node at (2.5,5.5) {4};
  \node at (3.5,5.5) {3};
  \node at (4.5,5.5) {2};
  \node at (5.5,5.5) {1};

\draw[red,very thick] (5,0) -- (5,2);   
\draw[red,very thick] (4,1) -- (4,3);
\draw[red,very thick] (3,2) -- (3,4);
\draw[red,very thick] (2,3) -- (2,5);
\draw[red,very thick] (1,4) -- (1,6);
\draw[green,thick] (4.5,2) circle [x radius=0.35, y radius=0.8];
\draw[green,thick] (4.5,4) circle [x radius=0.35, y radius=0.8];
\draw[green,thick] (5.5,4) circle [x radius=0.35, y radius=0.8];
\draw[green,thick] (3.5,5) circle [x radius=0.35, y radius=0.8];
\draw[green,thick] (4.5,5) circle [x radius=0.35, y radius=0.8];
\draw[green,thick] (5.5,5) circle [x radius=0.35, y radius=0.8];
\end{tikzpicture}
\]
Notice that
$$
\wt(D) = \wt_{\alpha}^n(\x_{n-1}, \dots, x_1; \y) = (x_4 - y_4)(x_2 - y_2)(x_2 - y_1)(x_1 - y_3)(x_1 - y_2)(x_1 - y_1).
$$
We explicitly show the computation 
of $u_4$ in Definition~\ref{D: BPD chain}.
First, we compute $\varphi_4$.
The cross section of $D$ above
row $4$ is:
$$\varphi_4(1) = 2, \quad \varphi_4(2) = \varnothing, \quad \varphi_4(3) = 1, \quad
\varphi_4(4) = \varnothing, \quad
\varphi_4(5) = \varnothing, \quad
\varphi_4(6) = 3.$$
Then $\varphi_4^{-1}$ is defined
on $[3]$.
We have $\varphi_4^{-1}(1) = 3$, 
$\varphi_4^{-1}(2) = 1$ and
$\varphi_4^{-1}(3) = 6$.
Since Definition~\ref{D: BPD chain}
requires that $u_4$ agrees with $\varphi_4^{-1}$ in $[3]$,
we have $u_4(1) = 3$, $u_4(2) = 1$
and $u_4(3) = 6$.
Finally, since $u_4 \in S_6$
and is decreasing on $[4, 6]$,
we have $u_4 = [3,1,6,5,4,2]$.
\end{exa}

We extract a key lemma from the proof of this theorem in~\cite{Yu}.
It relates BPD of one row with increasing chains. 

\begin{lem}
\label{L: one row BPD, right entry}
Let $\gamma$ be a boundary condition
of the one-row trapezoid $[1] \times [a, n]$.
Assume 
\[
\img(\gamma_N) = [a, k-1],\quad
\gamma_E(1) = k, \quad
\img(\gamma_S) = [a, k], \quad
\gamma_W(1)= \varnothing.
\]
Construct $w \in S_{a,n}$ (resp. $u \in S_{[a,n]}$)
as the permutation that
agrees with $\gamma_N^{-1}$ (resp. $\gamma_S^{-1}$)
on $[a, k-1]$ (resp. $[a, k]$) and decreasing 
on $(k-1, n]$ (resp. $(k, n]$).
Then $u \xrightarrow{k} w$ if and only if
there exists a BPD of $[1] \times [a, n]$
satisfying $\gamma$.
In this case, there is exactly one such BPD $D$,
and
\begin{equation}
\label{EQ: fix points and BPD}
\{c: D(1, c) = \btile\} = \fix_{(k, n]}(u, w).    
\end{equation}
\end{lem}

In words, the assumptions on $\gamma$
mean that a labeled BPD satisfying $\gamma$
has $k$ pipes with labels 
from $a$ to $k$,
all exiting from the bottom of row~$1$.
Pipe $k+1$ enters from the right,
while the other pipes enter from the top.

\begin{proof}
Set $a = 1$.
Suppose such a BPD exists. 
See~\cite{Yu}*{Lemma~3.10} for an explicit
construction of the increasing $k$-chain 
from $u$ to $w$.
Conversely, see Definition~3.15, Lemma~3.17, 
and Lemma~3.18 of~\cite{Yu} for an explicit 
construction of the BPD $D$ satisfying $\gamma$.
Uniqueness is immediate. 
Finally,~\eqref{EQ: fix points and BPD} 
follows from this construction, and the theorem for general $a$ follows 
from~\eqref{EQ: k inc correspondence}.
\end{proof}
\begin{exa}
    Consider the one row BPD 
    of $[1] \times [-1, 6]$
    where pipe $-1, \dots, 2$ 
    enter from the top and pipe $3$ 
    enters from row $1$.
    Let $\gamma$ be the boundary
    condition depicted below.
    $$
    
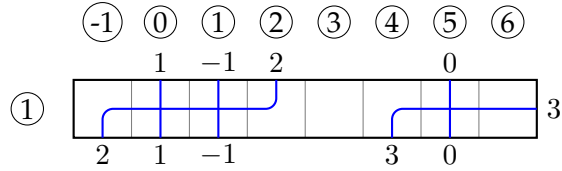
\begin{tikzpicture}[x=2em,y=2em,thick,rounded corners, color = blue]
        \draw[step=1,gray,thin] (0,0) grid (8,1);
        \draw[color=black, thick, sharp corners] (0,0) rectangle (8,1);
        \draw(0.5, 0)--(0.5,0.5)--(3.5,0.5)--(3.5,1);
        \draw(1.5, 0)--(1.5,1);
        \draw(2.5, 0)--(2.5,1);
        \draw(6.5, 0)--(6.5,1);
        \draw(5.5, 0)--(5.5,0.5)--(8,0.5);
        \node[color=black] at (-0.8,0.5) {$\circled{1}$};
        \node[color=black] at (0.5,2) {$\circled{-1}$};
        \node[color=black] at (1.5,2) {$\circled{0}$};
        \node[color=black] at (2.5,2) {$\circled{1}$};
        \node[color=black] at (3.5,2) {$\circled{2}$};
        \node[color=black] at (4.5,2) {$\circled{3}$};
        \node[color=black] at (5.5,2) {$\circled{4}$};
        \node[color=black] at (6.5,2) {$\circled{5}$};
        \node[color=black] at (7.5,2) {$\circled{6}$};
        \node[color=black] at (1.5,1.3) {$1$};
        \node[color=black] at (2.5,1.3) {$-1$};
        \node[color=black] at (3.5,1.3) {$2$};
        \node[color=black] at (6.5,1.3) {$0$};
        \node[color=black] at (0.5,-0.3) {$2$};
        \node[color=black] at (1.5,-0.3) {$1$};
        \node[color=black] at (2.5,-0.3) {$-1$};
        \node[color=black] at (5.5,-0.3) {$3$};
        \node[color=black] at (6.5,-0.3) {$0$};
        \node[color=black] at (8.3,0.5) {$3$};
    \end{tikzpicture}
    $$
    We construct $u = [1,5,0,-1,4,6,3,2]$ and $w = [1,5,0,2,6,4,3,-1]$ from $\gamma$
    as in Lemma~\ref{L: one row BPD, right entry}.
    Since there is a BPD satisfying $\gamma$,
    we have $u \xrightarrow{3} w$:
    The increasing chain from $u$ to $w$ 
    is $(u,[1,5,0,2,4,6,3,-1] ,w)$.
    We may depict $u \xrightarrow{3} w$ 
    as the table:
    $$
    \begin{tikzpicture}[scale=0.7]
      \draw (0,0) rectangle (8,2);
      \foreach \x in {1,...,8} \draw (\x,0) -- (\x,2);
      \foreach \y in {1,2} \draw (0,\y) -- (8,\y);
    
      \node at (0.5,0.5) {1};
      \node at (1.5,0.5) {5};
      \node at (2.5,0.5) {0};
      \node at (3.5,0.5) {-1};
      \node at (4.5,0.5) {4};
      \node at (5.5,0.5) {6};
      \node at (6.5,0.5) {3};
      \node at (7.5,0.5) {2};
    
      \node at (0.5,1.5) {1};
      \node at (1.5,1.5) {5};
      \node at (2.5,1.5) {0};
      \node at (3.5,1.5) {2};
      \node at (4.5,1.5) {6};
      \node at (5.5,1.5) {4};
      \node at (6.5,1.5) {3};
      \node at (7.5,1.5) {-1};
    
      \draw[red,ultra thick] (5,0) -- (5,2);
      \draw[green,thick] (6.5,1) circle [x radius=0.35, y radius=0.8];
    \end{tikzpicture}
    $$
    The table has fixed points $\fix_{(3, 6]}(u,w) = \{3\}$, corresponding to the single $\btile$ in column $3$ of the BPD. Conversely, given such $u\xrightarrow{k}w$, we can reconstruct the unique BPD satisfying $\gamma$.
\end{exa}

We now extend this lemma
to the case where there is no pipe
entering from the right edge of the row. 

\begin{lem}
\label{L: one row BPD, no right entry}
Let $\gamma$ be a boundary condition
of the one-row trapezoid $[1] \times [a, n]$.
Assume 
\[
\img(\gamma_N) = \img(\gamma_S) = [a, k],\quad
\gamma_E(1) = \gamma_W(1)=\varnothing.
\]
Construct $w \in S_{[a,n]}$ (resp. $u \in S_{[a,n]}$)
as the permutation that
agrees with $\gamma_N^{-1}$ (resp. $\gamma_S^{-1}$)
on $[a, k]$ and decreasing on $(k, n]$.
Then $u \xrightarrow{k} w$ if and only if
there exists a BPD of $\{1\} \times [a, n]$
satisfying $\gamma$.
In this case, there is exactly one such BPD $D$,
and~\eqref{EQ: fix points and BPD} holds.
\end{lem}

\begin{proof}
Define $\gamma'$ as a boundary condition on
$[1] \times [a, n+1]$ by
\[
\gamma'_N(i) = \gamma_N(i), \qquad 
\gamma'_S(i) = \gamma_S(i) \quad \text{for } i \in [n],
\]
and
\[
\gamma'_S(n+1) = \gamma'_E(1) = k+1, \quad 
\gamma'_N(n+1) = \gamma'_W(1) = \varnothing.
\]
Then there is a bijection from 
$\BPD([1] \times [a, n], \gamma)$ 
to $\BPD([1] \times [a, n+1], \gamma')$
obtained by appending a $\rtile$ at the cell $(1,n+1)$.
By Lemma~\ref{L: one row BPD, right entry}, 
$\BPD([1] \times [a, n+1], \gamma')$ is nonempty
if and only if $u' \xrightarrow{k+1} w'$, where 
\[
u' = [u(1), \dots, u(k), n+1, u(k+1), \dots, u(n)],
\]
and $w'$ is obtained similarly from $w$.
By Lemma~\ref{L: Insert large number},
$u' \xrightarrow{k+1} w'$ 
if and only if $u \xrightarrow{k} w$.
In this case, the unique BPD in 
$\BPD([1] \times [a, n], \gamma)$ 
has $\btile$ in the same positions 
as the unique BPD in 
$\BPD([1] \times [a, n+1], \gamma')$.
Also, 
by Lemma~\ref{L: Insert large number},
$\fix_{(k,n]}(u,w) = \fix_{(k+1, n+1]}(u',w')$.
Hence the result follows from 
Lemma~\ref{L: one row BPD, right entry}.
\end{proof}

\begin{exa}
    Consider the one row BPD 
    of $[1] \times [-1, 6]$
    where pipe $-1, \dots, 2$ 
    enter from the top and no pipe 
    enters from row $1$.
    Let $\gamma$ be the boundary
    condition depicted below.
    $$
    
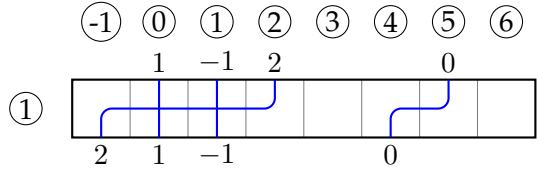
\begin{tikzpicture}[x=2em,y=2em,thick,rounded corners, color = blue]
        \draw[step=1,gray,thin] (0,0) grid (8,1);
        \draw[color=black, thick, sharp corners] (0,0) rectangle (8,1);
        \draw(0.5, 0)--(0.5,0.5)--(3.5,0.5)--(3.5,1);
        \draw(1.5, 0)--(1.5,1);
        \draw(2.5, 0)--(2.5,1);
        \draw(5.5, 0)--(5.5,0.5)--(6.5,0.5)--(6.5,1);
        \node[color=black] at (-0.8,0.5) {$\circled{1}$};
        \node[color=black] at (0.5,2) {$\circled{-1}$};
        \node[color=black] at (1.5,2) {$\circled{0}$};
        \node[color=black] at (2.5,2) {$\circled{1}$};
        \node[color=black] at (3.5,2) {$\circled{2}$};
        \node[color=black] at (4.5,2) {$\circled{3}$};
        \node[color=black] at (5.5,2) {$\circled{4}$};
        \node[color=black] at (6.5,2) {$\circled{5}$};
        \node[color=black] at (7.5,2) {$\circled{6}$};
        \node[color=black] at (1.5,1.3) {$1$};
        \node[color=black] at (2.5,1.3) {$-1$};
        \node[color=black] at (3.5,1.3) {$2$};
        \node[color=black] at (6.5,1.3) {$0$};
        \node[color=black] at (0.5,-0.3) {$2$};
        \node[color=black] at (1.5,-0.3) {$1$};
        \node[color=black] at (2.5,-0.3) {$-1$};
        \node[color=black] at (5.5,-0.3) {$0$};
    \end{tikzpicture}
    $$
    We construct $u = [1,4,0,-1,6,5,3,2]$ and $w = [1,5,0,2,6,4,3,-1]$ from $\gamma$
    as in Lemma~\ref{L: one row BPD, no right entry}.
    Since there is a BPD satisfying $\gamma$,
    we have $u \xrightarrow{2} w$:
    The increasing chain from $u$ to $w$ 
    is $$([1,4,0,-1,6,5,3,2],[1,4,0,2,6,5,3,-1] ,[1,5,0,2,6,4,3,-1]).$$
    We may depict $u \xrightarrow{2} w$ 
    as the table:
    $$
    \begin{tikzpicture}[scale=0.7]
      \draw (0,0) rectangle (8,2);
      \foreach \x in {1,...,8} \draw (\x,0) -- (\x,2);
      \foreach \y in {1,2} \draw (0,\y) -- (8,\y);
    
      \node at (0.5,0.5) {1};
      \node at (1.5,0.5) {4};
      \node at (2.5,0.5) {0};
      \node at (3.5,0.5) {-1};
      \node at (4.5,0.5) {6};
      \node at (5.5,0.5) {5};
      \node at (6.5,0.5) {3};
      \node at (7.5,0.5) {2};
    
      \node at (0.5,1.5) {1};
      \node at (1.5,1.5) {5};
      \node at (2.5,1.5) {0};
      \node at (3.5,1.5) {2};
      \node at (4.5,1.5) {6};
      \node at (5.5,1.5) {4};
      \node at (6.5,1.5) {3};
      \node at (7.5,1.5) {-1};
    
      \draw[red,ultra thick] (4,0) -- (4,2);
      \draw[green,thick] (6.5,1) circle [x radius=0.35, y radius=0.8];
      \draw[green,thick] (4.5,1) circle [x radius=0.35, y radius=0.8];
    \end{tikzpicture}
    $$

    The table has fixed points $\fix_{(3, 6]}(u,w) = \{3,6\}$, corresponding to the two $\btile$ in columns $3$ and $6$ of the BPD. Conversely, given such $u\xrightarrow{k}w$, we can reconstruct the unique BPD satisfying $\gamma$.
\end{exa}

\section{Slicing BPDs}
\label{S: slice}
As outlined in the introduction, 
we reduce our main problem (Problem~\ref{Pb: main}) to Problem~\ref{Pb: lower},
and then to Problem~\ref{Pb: TBPD}.
These two reductions are achieved
by slicing a BPD of $\bBPD(w)$:
first under row $0$ and then along the 
main diagonal. 
Our process is summarized by Figure~\ref{F: cuts}. 
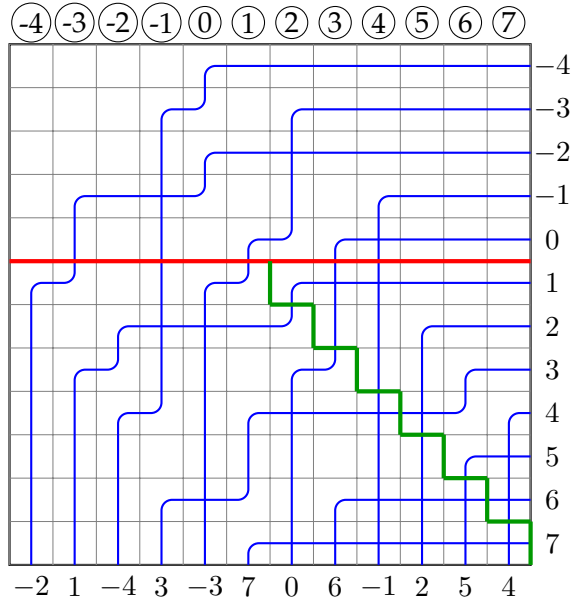
\begin{figure}
    \centering
    \begin{tikzpicture}[x=1.5em,y=1.5em,thick,rounded corners,color = blue]
    \draw[color=black, thick, sharp corners] (0,0) rectangle (12,12);
    \draw[step=1,gray,ultra thin] (0,0) grid (12,12);
    \draw[color=blue, thick] (0.5,0)--(0.5, 6.5)--(1.5, 6.5)--(1.5, 8.5)--(4.5,8.5)--(4.5,9.5)--(12,9.5);
    \draw[blue,thick] (1.5,0) -- (1.5,4.5) -- (2.5,4.5) -- (2.5,5.5) -- (6.5,5.5) -- (6.5,6.5)--(12,6.5);
    \draw[blue,thick] (2.5,0) -- (2.5,3.5) -- (3.5,3.5) -- (3.5,10.5)--(4.5,10.5)--(4.5,11.5)--(12,11.5);
    \draw[blue,thick] (3.5,0) -- (3.5,1.5) -- (5.5,1.5) -- (5.5,3.5) -- (10.5,3.5)--(10.5,4.5)--(12,4.5);
    \draw[blue,thick] (4.5,0) -- (4.5,6.5) -- (5.5,6.5) -- (5.5,7.5)--(6.5,7.5)--(6.5,10.5)--(12,10.5);
    \draw[blue,thick] (5.5,0) -- (5.5,0.5) -- (12,0.5);
    \draw[blue,thick] (6.5,0) -- (6.5,4.5) -- (7.5,4.5) -- (7.5,7.5)--(12,7.5);
    \draw[blue,thick] (7.5,0) -- (7.5,1.5) -- (12,1.5);
    \draw[blue,thick] (8.5,0) -- (8.5,8.5)--(12,8.5);
    \draw[blue,thick] (9.5,0) -- (9.5,5.5)--(12,5.5);
    \draw[blue,thick] (10.5,0) -- (10.5,2.5)--(12,2.5);
    \draw[blue,thick] (11.5,0) -- (11.5,3.5)--(12,3.5);
    \draw[red,ultra thick] (0,7) -- (12,7);
    \draw[darkgreen,ultra thick] (6,7) -- (6,6);
    \draw[darkgreen,ultra thick] (7,6) -- (7,5);
    \draw[darkgreen,ultra thick] (8,5) -- (8,4);
    \draw[darkgreen,ultra thick] (9,4) -- (9,3);
    \draw[darkgreen,ultra thick] (10,3) -- (10,2);
    \draw[darkgreen,ultra thick] (11,2) -- (11,1);
    \draw[darkgreen,ultra thick] (12,1) -- (12,0);
    \draw[darkgreen,ultra thick] (6,6) -- (7,6);
    \draw[darkgreen,ultra thick] (7,5) -- (8,5);
    \draw[darkgreen,ultra thick] (8,4) -- (9,4);
    \draw[darkgreen,ultra thick] (9,3) -- (10,3);
    \draw[darkgreen,ultra thick] (10,2) -- (11,2);
    \draw[darkgreen,ultra thick] (11,1) -- (12,1);
    \node[color=black] at (0.5,12.5) {$\circled{-4}$};
    \node[color=black] at (1.5,12.5) {$\circled{-3}$};
    \node[color=black] at (2.5,12.5) {$\circled{-2}$};
    \node[color=black] at (3.5,12.5) {$\circled{-1}$};
    \node[color=black] at (4.5,12.5) {$\circled{0}$};
    \node[color=black] at (5.5,12.5) {$\circled{1}$};
    \node[color=black] at (6.5,12.5) {$\circled{2}$};
    \node[color=black] at (7.5,12.5) {$\circled{3}$};
    \node[color=black] at (8.5,12.5) {$\circled{4}$};
    \node[color=black] at (9.5,12.5) {$\circled{5}$};
    \node[color=black] at (10.5,12.5) {$\circled{6}$};
    \node[color=black] at (11.5,12.5) {$\circled{7}$};
    \node[color=black] at (12.5,0.5) {$7$};
    \node[color=black] at (12.5,1.5) {$6$};
    \node[color=black] at (12.5,2.5) {$5$};
    \node[color=black] at (12.5,3.5) {$4$};
    \node[color=black] at (12.5,4.5) {$3$};
    \node[color=black] at (12.5,5.5) {$2$};
    \node[color=black] at (12.5,6.5) {$1$};
    \node[color=black] at (12.5,7.5) {$0$};
    \node[color=black] at (12.5,8.5) {$-1$};
    \node[color=black] at (12.5,9.5) {$-2$};
    \node[color=black] at (12.5,10.5) {$-3$};
    \node[color=black] at (12.5,11.5) {$-4$};
    \node[color=black] at (0.5,-0.5) {$-2$};
    \node[color=black] at (1.5,-0.5) {$1$};
    \node[color=black] at (2.5,-0.5) {$-4$};
    \node[color=black] at (3.5,-0.5) {$3$};
    \node[color=black] at (4.5,-0.5) {$-3$};
    \node[color=black] at (5.5,-0.5) {$7$};
    \node[color=black] at (6.5,-0.5) {$0$};
    \node[color=black] at (7.5,-0.5) {$6$};
    \node[color=black] at (8.5,-0.5) {$-1$};
    \node[color=black] at (9.5,-0.5) {$2$};
    \node[color=black] at (10.5,-0.5) {$5$};
    \node[color=black] at (11.5,-0.5) {$4$};
\end{tikzpicture}
\caption{Slicing
a BPD of $\bBPD(w)$. We only depict the square $[-4, 7] \times [-4, 7]$.}
\label{F: cuts}
\end{figure}

\subsection*{Slicing back stable BPDs into upper and lower parts}

\label{S: horizontal cut}
We reduce Problem~\ref{Pb: main}
to Problem~\ref{Pb: lower}
using the \definition{horizontal cut} (represented by the red line in Figure~\ref{F: cuts}).
When applied to $D \in \bBPD(w)$,
the horizontal cut yields 
$D^\uparrow$ and $D^\downarrow$, obtained by restricting $D$ to the upper and lower half-planes
$H^\uparrow$ and $H^\downarrow$, respectively (see Definition~\ref{D: Tra}).
Then $D^\uparrow \in \BPD(H^\uparrow, \Upsilon)$ and
$D^\downarrow \in \BPD(H^\downarrow, \Delta)$
for certain boundary conditions $\Upsilon$ and $\Delta$.
We now introduce convenient notation to describe these boundary conditions.

\begin{defn}
\label{D: varphi and sigma}
A map $\varphi: \Z \rightarrow \Zn \sqcup \{\varnothing\}$ is called a \definition{horizon section}
if it is the cross section under row $r$
of some $D \in \bBPD(w)$ with the labeling with respect to $\gamma^w$.
In other words, each $i\in\Zn$ has a unique preimage
under $\varphi$
and $\varphi(i)=i$ for all sufficiently negative~$i$.

Recall that $S_- := \{ w \in S_\Z: \Des(w)\subset \Zn\}$.
We describe bijections between horizon sections and $S_-$.
For a horizon section $\varphi$, define $\sigma_\varphi$ as the permutaiton in $S_-$
that agrees with $\varphi^{-1}$ on $\Zn$.
Explicitly, 
$\sigma_\varphi(i)=j$ if $i\in\Zn$ and $\varphi(j)=i$.
Conversely, given $\sigma \in S_-$,
define the horizon section $\varphi_\sigma$ as the 
cross section of $\Rothe_w$ under row $0$.
Explicitly, 
\[
\varphi_\sigma(j)=
\begin{cases}
i &\text{if } j\in\sigma(\Zn)\text{ and }\sigma(i)=j,\\
\varnothing &\text{otherwise.}
\end{cases}
\]
Clearly, the maps $\varphi \mapsto \sigma_\varphi$ and $\sigma \mapsto \varphi_\sigma$
are mutually inverse.
\end{defn}

\begin{defn}
Given $w\in S_\Z$ and $\sigma\in S_-$,
define two boundary conditions $\Upsilon^{\sigma}$ and $\Delta^{\sigma,w}$
for $H^\uparrow$ and $H^\downarrow$, respectively, by
\begin{itemize}
\item $\Upsilon^{\sigma}_S = \Delta_N^{\sigma,w} = \varphi_\sigma$;
\item all other components $\Upsilon^{\sigma}_X(i)$ and $\Delta^{\sigma,w}_X(i)$,
for $X\in\{N,E,S,W\}$, agree with $\gamma^w_X(i)$.
\end{itemize}
Notice that $\Upsilon^{\sigma}$ does not depend on $w$.
\end{defn}

\begin{pro}
\label{P: Horizon cut}
The horizon cut defines a bijection
\begin{equation}
\label{EQ: horizon cut bijection}
\bBPD(w)\;\longrightarrow\;
\bigsqcup_{\sigma\in S_-}
\BPD(H^\uparrow,\Upsilon^{\sigma})\times
\BPD(H^\downarrow,\Delta^{\sigma,w}).
\end{equation}
Consequently,
\begin{equation}
\label{EQ: Horizon cut}
\bfS_w(\x;\y)
= \sum_{\sigma\in S_-}
\left(\sum_{D\in\BPD(H^\uparrow,\Upsilon^{\sigma})}\wt(D)\right)
\left(\sum_{D\in\BPD(H^\downarrow,\Delta^{\sigma,w})}\wt(D)\right).
\end{equation}
Moreover, in both~\eqref{EQ: horizon cut bijection} and~\eqref{EQ: Horizon cut},
one may restrict to $\sigma$
such that $\varphi_\sigma$ is the 
cross section of some $D \in \bBPD(w)$ 
under row $0$.
\end{pro}

\begin{proof}
Given $D\in\bBPD(w)$,
label $D$ so that it satisfies $\gamma^w$.
Then $D^\uparrow$ and $D^\downarrow$ inherit this labeling,
and let $\gamma^\uparrow$ and $\gamma^\downarrow$
denote their respective boundary conditions.
We have $\gamma^\uparrow_S=\gamma^\downarrow_N=\varphi$
where $\varphi$ is the cross section 
of $D$ under row $0$.
It is then routine to verify that
$\gamma^\uparrow=\Upsilon^{\sigma_\varphi}$ and
$\gamma^\downarrow=\Delta^{\sigma_\varphi,w}$

Conversely, fix $\sigma\in S_-$.
Given
$D_1\in\BPD(H^\uparrow,\Upsilon^{\sigma})$
and $D_2\in\BPD(H^\downarrow,\Delta^{\sigma,w})$,
concatenate them along the horizon to obtain a tiling $D$.
Since $\Upsilon^{\sigma}_S=\Delta^{\sigma,w}_N$,
we can label pipes in $D$ so that the labeling
agrees with that of $D_1$ (resp.~$D_2$) on the upper (resp.~lower) half-plane.
We check that $D$ is a BPD: 
$D$ is clearly 
consistent and stabilized. 
To see $D$ is reduced, 
we show $D$ has a valid labeling: 
Because labelings of both $D_1$ and $D_2$ are valid,
so is the labeling inherited by $D$.
It is immediate that $D$ satisfies $\gamma^w$, hence $D\in\bBPD(w)$. Finally, $\wt(D)=\wt(D^\uparrow)\wt(D^\downarrow)$ gives~\eqref{EQ: Horizon cut}.
\end{proof}

Thus, to study $\bfS_w(\x;\y)$ it suffices to understand the generating functions of
$\BPD(H^\uparrow,\Upsilon^{\sigma})$ and $\BPD(H^\downarrow,\Delta^{\sigma,w})$.
We begin with the upper half-plane.

\begin{lem}
\label{L: Upper Schub}
Given $w\in S_-$,
we have
\[
\sum_{D\in\BPD(H^\uparrow,\Upsilon^{w})}\wt(D)=\bfS_w(\x;\y).
\]
\end{lem}

\begin{proof}
By Lemma~\ref{L: Fixed under Des},
since $\Des(w)\subset\Zn$,
all BPDs in $\bBPD(\sigma)$
agree below row~$0$.
Hence there is exactly one $\sigma\in S_-$
that contributes to the right hand side of~\eqref{EQ: Horizon cut}:
the $\sigma$ such that $\varphi_\sigma$ 
is the cross section of 
$\Rothe_w$ under row $0$, so $\sigma = w$.
By Lemma~\ref{L: Fixed under Des} again,
$\BPD(H^\downarrow,\Delta^{w,w})$ consists of a single BPD without any $\btile$.
Applying~\eqref{EQ: Horizon cut} to $\bfS_u$
yields
\[
\bfS_w(\x;\y)
=\Big(\sum_{D\in\BPD(H^\uparrow,\Upsilon^{w})}\wt(D)\Big)\cdot1,
\]
as desired.
\end{proof}

\begin{rem}
\label{R: Upper = Schur}
Given a $0$-Grassmannian $\sigma=w_\lambda$,
we observe that its corresponding 
$\varphi_\sigma$ satisfies:
$\varphi_\sigma(c)<\varphi_\sigma(c')$ whenever
$\varphi_\sigma(c)\neq\varnothing$, $\varphi_\sigma(c')\neq\varnothing$, and $c<c'$.
By Lemma~\ref{L: Boundary and crossing},
BPDs in $\BPD(H^\uparrow,\Upsilon^{\sigma})$
contain no $\ptile$.
These are exactly the half-plane crossless pipedreams studied by
Lam, Lee, and Shimozono~\cite{LLS}*{\S5.3},
whose generating function is $s_\lambda(\x\|\y)$.
Our Lemma~\ref{L: Upper Schub} would 
say that their generating function is
$\bfS_{w_\lambda}(\x; \y) = s_\lambda(\x\|\y)$.
Thus, Lemma~\ref{L: Upper Schub} extends their result.
\end{rem}

We now reformulate~\eqref{EQ: Horizon cut} by introducing a new polynomial.

\begin{defn}
\label{D: Lower Schub}
For $w\in S_\Z$ and $\sigma\in S_-$,
let $\LBPD(\sigma,w):=\BPD(H^\downarrow,\Delta^{\sigma,w})$
and define the \definition{lower Schubert polynomial} as
\[
\lfS_{\sigma,w}(\x;\y)
:=\sum_{D\in\LBPD(\sigma,w)}\wt(D).
\]
\end{defn}

\begin{cor}
For $w\in S_\Z$,
\begin{equation}
\label{EQ: Horizon cut 2}
\bfS_w(\x;\y)
=\sum_{\sigma\in S_-}
\bfS_\sigma(\x;\y)\:\lfS_{\sigma,w}(\x;\y).
\end{equation}
Moreover, one may restrict to $\sigma$ such that $\varphi_\sigma$
is the cross section of some $D \in \bBPD(w)$
under row $0$.
\end{cor}

\begin{proof}
This follows directly from substituting Lemma~\ref{L: Upper Schub}
and Definition~\ref{D: Lower Schub} into~\eqref{EQ: Horizon cut}.
\end{proof}

We next identify certain coefficients $a^w_\lambda(\x;\y)$
as lower Schubert polynomials.

\begin{lem}
\label{L: Lower Schub}
Let $w\in S_\Z$ with $\Des(w)\subset\Z_{\ge0}$
and let $\lambda$ be a partition.
Then $a^w_\lambda(\x;\y)=\lfS_{w_\lambda,w}(\x;\y)$.
\end{lem}

\begin{proof}
Consider the right-hand side of~\eqref{EQ: Horizon cut 2}.
We know the contribution of $\sigma$
is $0$ unless $\varphi_\sigma$ is the 
cross section of some $D \in \bBPD(w)$.
We claim such $\sigma$ are $0$-Grassmannian. 
If not, 
then $\sigma(i-1)>\sigma(i)$ for some $i\in\Zn$
and say $D \in \bBPD(w)$ has $\varphi_\sigma$
as the cross section under row $0$.
By Lemma~\ref{L: Boundary and crossing} and $\Des(w) \subset \Z_{\geq 0}$,
pipe $i$ and pipe $i-1$ in $D$ cannot cross.
But Lemma~\ref{L: Boundary and crossing} and 
$\sigma(i-1)>\sigma(i)$ implies that 
these two pipes cross in the upper-half plane,
so we reach a contradiction.

Thus, the sum in~\eqref{EQ: Horizon cut 2}
can be restricted to $\sigma$ that are $0$-Grassmannian.
Since $0$-Grassmannian permutations correspond bijectively to partitions,
we may rewrite~\eqref{EQ: Horizon cut 2} as
\[
\bfS_w(\x;\y)
=\sum_{\lambda}\bfS_{w_\lambda}(\x;\y)\:\lfS_{w_\lambda,w}(\x;\y),
\]
and the claim follows because $\bfS_{w_\lambda}(\x;\y)=s_\lambda(\x\|\y)$.
\end{proof}

Finally, we express all $a^w_\lambda(\x;\y)$ in terms of lower Schubert polynomials.

\begin{cor}
\label{C: a = lower * lower}
For $w\in S_\Z$,
\begin{equation}
a^w_\lambda(\x;\y)
=\sum_{\sigma\in S_-}
\omega_1(\lfS_{w_{\lambda'},\neg\sigma}(\x;\y))\:
\lfS_{\sigma,w}(\x;\y).
\end{equation}
\end{cor}

\begin{proof}
Extracting the coefficient of $s_\lambda(\x\|\y)$
from~\eqref{EQ: Horizon cut 2} gives
\[
a^w_\lambda(\x;\y)
=\sum_{\sigma\in S_-}
a^\sigma_\lambda(\x;\y)\:\lfS_{\sigma,w}(\x;\y).
\]
By Corollary~\ref{C: Omega},
$a^\sigma_\lambda(\x;\y)=\omega_1(a_{\lambda'}^{\neg\sigma}(\x;\y))$.
Then Lemma~\ref{L: Negation and Des} implies
$\Des(\neg\sigma)\subset\Z_{\ge0}$,
and by Lemma~\ref{L: Lower Schub},
$a_{\lambda'}^{\neg\sigma}(\x;\y)=\lfS_{w_{\lambda'},\neg\sigma}(\x;\y)$,
completing the proof.
\end{proof}

Thus, Problem~\ref{Pb: main}
is reduced to Problem~\ref{Pb: lower}.

\subsection*{Slicing BPDs of the lower half-plane 
into a trapezoid and a triangle}
\label{S: diagonal cut}

We reduce Problem~\ref{Pb: lower}
to Problem~\ref{Pb: TBPD} by further dissembling the polynomial 
$\lfS_{\sigma, w}(\x; \y)$
into two smaller pieces.
One piece will automatically have a satisfying
combinatorial formula
while the other piece is left to study in the 
next section.

We apply the \definition{diagonal cut} (represented by the
green line in Figure~\ref{F: cuts}).
Analogous to the horizon cut in the previous section,
the diagonal cut on
$D \in \LBPD(\sigma, w)$
cuts $D$ into $D^{\tra}$
and $D^{\tri}$ by restricting
$D$ to $\A^{\tra}$ and $\A^{\tri}$, respectively (see Definition~\ref{D: Tra}).
Label $D$ with respect to 
$\Delta^{\sigma, w}$;
then $D^{\tra}$ and $D^{\tri}$
inherit this labeling. 
We now record the boundary condition of $D^{\tra}$ that can arise.

\begin{defn}
\label{D: Gamma}
Take $\sigma \in S_-$ and $w \in S_\Z$.
Let $\Gamma(\sigma, w)$ be the set of 
boundary condition $\gamma$ of $\A^{\tra}$
such that 
\[
\gamma_N = \varphi_\sigma(i)\mid_{(-\infty, 1]},\qquad 
\gamma_S = \gamma_S^w,\qquad
\text{and for sufficiently large $r$: }
\gamma_E(r) = r,\ \ \gamma_N(r) = \varnothing.
\]

Together with $\sigma$ and $w$,
each $\gamma \in \Gamma(\sigma, w)$
determines a boundary condition
\definition{$\theta$} on $\A^{\tri}$ by
\[
\theta_N = \varphi_\sigma\!\mid_{[2,\,\infty)},\qquad
\theta_E(i) = i,\qquad
\theta_S = \gamma_N,\qquad
\theta_W = \gamma_E.
\]
Then define $\BPD^{\tra}(\gamma) := \BPD(\A^{\tra}, \gamma)$
and $\BPD^{\tri}(\gamma, \sigma, w) = \BPD(\A^{\tri}, \theta^{\gamma, \sigma, w})$.
\end{defn}

\begin{pro}
\label{P: Diag cut}
The diagonal cut yields a bijection
\[
\LBPD(\sigma, w) \;\longrightarrow\;
\bigsqcup_{\gamma \in \Gamma(\sigma, w)}
\BPD^{\tra}(\gamma) \times \BPD^{\tri}(\gamma, \sigma, w).
\]
Consequently,
\begin{equation}
\label{EQ: Diag cut}
\lfS_{\sigma, w}(\x;\y)
= \sum_{\gamma \in \Gamma(\sigma, w)}
\left(\sum_{D\in\BPD^{\tra}(\gamma)}\wt(D)\right)
\left(\sum_{D\in\BPD^{\tri}( \gamma, \sigma, w)}\wt(D)\right).
\end{equation}
\end{pro}
\begin{proof}
Analogous to the proof of Proposition~\ref{P: Horizon cut}.
\end{proof}

Note that the region $\A^{\tri}$
consists of the cells $(i,j)$ with $0 < i < j$.
Therefore the generating function 
$\sum_{D\in\BPD^{\tri}( \gamma, \sigma, w)}\wt(D)$
in~\eqref{EQ: Diag cut} is already a combinatorial formula
in which each summand is a product of distinct factors $(x_i - y_j)$
with $0 < i < j$.
After specializing $\x \mapsto \y$,
each summand is a product of distinct type~$1$ factors.

Hence our focus is the other factor in~\eqref{EQ: Diag cut}.
For $\gamma \in \Gamma(\sigma, w)$, write 
$$\fS^{\tra}_\gamma(\x; \y) := \sum_{D\in\BPD^{\tra}(\gamma)}\wt(D).$$
Our goal becomes Problem~\ref{Pb: TBPD}:
give a combinatorial formula for
$\fS^{\tra}_\gamma(\y; \y)$
that is manifestly a sum of distinct products
of type $3$ factors.

\section{BPDs in the Trapezoid region
and Increasing chains}
\label{S: TBPD}

\subsection*{Preparation}

Fix $\sigma \in S_-$, $w \in S_\Z$, 
and $\gamma \in \Gamma(\sigma, w)$.
There exist integers $a \leq 0 < n$ such that 
\begin{equation}
\label{EQ: finding a n}   
\text{$\gamma_N(c) = \gamma_S(c) = c$ for all $c < a$, 
and $\gamma_E(r) = \gamma_S(r) = r$ for all $r > n$.}
\end{equation}
It follows that every BPD in $\TBPD(\gamma)$
agrees at all positions $(r,c)$ with $r > n$ or $c < a$: 
these cells must be $\rtile$ if $r = c$, 
and $\vtile$ otherwise.
In particular, such cells cannot be $\btile$.

Hence, we fix $a \leq 0 < n$ 
and restrict attention to boundary conditions $\gamma$ 
satisfying~\eqref{EQ: finding a n}.
We consider the finite trapezoidal region
\[
\A^{\tra}_{a,n} :=
\{(r,c) : a \le c, \, r \le n\} \cap \A^{\tra}.
\]
Each $D \in \TBPD(\gamma)$ can be restricted 
to $\A^{\tra}_{a,n}$, and we first characterize
the boundary conditions that can arise.

\begin{defn}
A boundary condition $\delta$ 
of $\A^{\tra}_{a,n}$ is \definition{regular}
if $\delta_S(c) \neq \varnothing$ for all 
$c \in [a,n]$ and 
$\delta_W(r) = \varnothing$ 
for all $r \in [n]$.
We write $\TBPD_{a,n}(\delta) := \BPD(\A^{\tra}_{a,n}, \delta)$.
\end{defn}
In words, these requirements on 
$\delta$ make sure that
a BPD satisfying $\delta$
has pipes exiting from each
column but no pipes exiting from
rows.

\begin{lem}
\label{L: regular boundary condition}
The restriction map $D \mapsto D\mid_{\A^{\tra}_{a,n}}$
is a weight-preserving bijection
\[
\TBPD(\gamma) \longrightarrow \TBPD_{a,n}(\delta)
\]
for some regular $\delta$.
Specifically, $\delta$ is the boundary condition
obtained by restricting the maps in $\gamma$:
\[
\delta_N = \gamma_N \mid_{[a,n]}, \qquad
\delta_E = \gamma_E \mid_{[n]}, \qquad
\delta_S = \gamma_S \mid_{[a,n]}, \qquad
\delta_W = \gamma_W \mid_{[n]}.
\]
Consequently,
\[
\fS^{\tra}_\gamma(\x; \y) = 
\fS^{\tra}_{a,n,\delta}(\x ;\y)
\textrm{ where }
\fS^{\tra}_{a,n,\delta}(\x ;\y):=
\sum_{D \in \TBPD_{a,n}(\delta)} \wt(D).
\]
\end{lem}

\begin{proof}
Take $D \in \TBPD(\gamma)$.
By assumption, if $(r,c) \notin \A^{\tra}_{a,n}$,
then $D(r,c) = \rtile$ if $r = c$,
and $D(r,c) = \vtile$ otherwise. 
Hence $D\mid_{\A^{\tra}_{a,n}}$ 
satisfies $\delta$.
Since $D$ has no $\btile$ outside $\A^{\tra}_{a,n}$,
the map preserves weights. 
Bijectivity is immediate. 
To check regularity, note that 
$(n+1,a), \dots, (n+1,n)$
are all $\vtile$ for any $D \in \TBPD(\gamma)$.
Thus, in every $D \in \TBPD_{a,n}(\gamma)$,
each tile in the bottom row ($r = n$) 
connects to the boundary below,
so $\delta_S(c) \neq \varnothing$ 
for all $c \in [a,n]$.
Similarly, each tile in the leftmost
column ($c = a$) does not connect to the west boundary,
so $\delta_W(r) = \varnothing$ 
for all $r \in [n]$.
\end{proof}

We may restate Problem~\ref{Pb: TBPD}
as follows. 

\begin{prob}
\label{Pb: TBPD 2}
Fix $a \leq 0 < n$.
Let $\gamma$ be a regular boundary condition
of $\A^{\tra}_{a,n}$.
Give a combinatorial formula
for $\fS^{\tra}_{a,n,\gamma}(\y; \y)$
The formula should be
a sum where each summand is a distinct product of
type~3 factors.
\end{prob}

We first handle a vanishing case.
If there exists $r \in [n]$
with $\gamma_N(r) = \gamma_E(r) = \varnothing$,
then the cell $(r,r)$ 
must be a $\btile$ for every $D \in \TBPD_{a,n}(\gamma)$,
so $\wt(D)$
vanishes after substituting $\x \mapsto \y$.
Henceforth we assume that
for all $r \in [n]$, 
either $\gamma_N(r) \neq \varnothing$ or $\gamma_E(r)\neq \varnothing$.

Next, suppose $\gamma_N(r) = \varnothing$
for some $r \in [n]$.
Define $\widetilde{\gamma}$ to be the boundary condition
of $\A^{\tra}$ obtained by setting 
$\widetilde{\gamma}_N(r) = \gamma_E(r)$ and 
$\widetilde{\gamma}_E(r) = \varnothing$, 
agreeing with $\gamma$ elsewhere. 
There is a weight-preserving bijection 
\begin{equation}
\label{EQ: TBPD bijection reduction}
\TBPD_{a,n}(\gamma) \longrightarrow \TBPD_{a,n}(\widetilde{\gamma})    
\end{equation}
given by replacing the tile in $(r,r)$:
if it is $\rtile$ (resp. $\htile$), 
replace it by $\vtile$ (resp. $\jtile$).
It therefore suffices to resolve 
Problem~\ref{Pb: TBPD 2}
for $\TBPD_{a,n}(\widetilde{\gamma})$ instead of $\TBPD_{a,n}(\gamma)$.

\begin{exa}
\label{Ex: twist j tile}
Take $a = -1$ and $n = 4$.
Consider a regular $\gamma$ 
where $D \in \TBPD_{a,n}(\gamma)$ is
the BPD on the left:
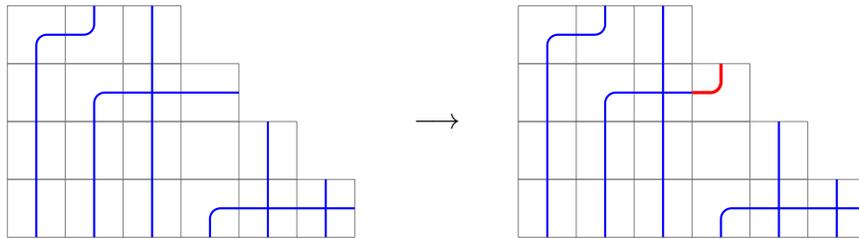
\[
\begin{tikzpicture}[x=2em,y=2em,thick,rounded corners,color = blue]
\draw[step=1,gray,ultra thin] (0,0) grid (6,1);
\draw[step=1,gray,ultra thin] (0,1) grid (5,2);
\draw[step=1,gray,ultra thin] (0,2) grid (4,3);
\draw[step=1,gray,ultra thin] (0,3) grid (3,4);
\draw[color=blue, thick] (.5,0)--(.5, 3.5)--(1.5,3.5)--(1.5,4);
\draw[color=blue, thick](1.5,0)--(1.5,2.5)--(4,2.5);
\draw[color=blue, thick](2.5,0)--(2.5,4);
\draw[color=blue, thick](3.5,0)--(3.5,.5)--(6, .5);
\draw[color=blue, thick](4.5,0)--(4.5,2);
\draw[color=blue, thick](5.5,0)--(5.5,1);
\end{tikzpicture}
\qquad \raisebox{3.8em}{$\longrightarrow$} \qquad
\begin{tikzpicture}[x=2em,y=2em,thick,rounded corners,color = blue]
\draw[step=1,gray,ultra thin] (0,0) grid (6,1);
\draw[step=1,gray,ultra thin] (0,1) grid (5,2);
\draw[step=1,gray,ultra thin] (0,2) grid (4,3);
\draw[step=1,gray,ultra thin] (0,3) grid (3,4);
\draw[color=blue, thick] (.5,0)--(.5, 3.5)--(1.5,3.5)--(1.5,4);
\draw[color=blue, thick](1.5,0)--(1.5,2.5)--(3,2.5);
\draw[color=red, very thick]
(3,2.5)--(3.5,2.5)--(3.5,3);
\draw[color=blue, thick](2.5,0)--(2.5,4);
\draw[color=blue, thick](3.5,0)--(3.5,.5)--(6, .5);
\draw[color=blue, thick](4.5,0)--(4.5,2);
\draw[color=blue, thick](5.5,0)--(5.5,1);
\end{tikzpicture}
\] 
From the picture of $D$,
we have 
$\gamma_N(2) = \varnothing$. 
Let $\widetilde{\gamma}$
be the regular boundary 
condition constructed in the arguments above.
Then $\TBPD_{a,n}$
contains the BPD
on the right.
It is obtained from $D$
by applying the bijection
in~\eqref{EQ: TBPD bijection reduction}: changing 
$\htile$ in $(2,2)$ into
$\jtile$.
\end{exa} 

To summarize, we have fixed $a \leq 0 < n$
and reduced Problem~\ref{Pb: TBPD 2}
to regular $\gamma$ 
with $\gamma_N(r) \neq \varnothing$
for all $r \in [n]$.
We henceforth work only with such $\gamma$.

\subsection*{Example of the argument}
Before turning to the general construction, we give a small example illustrating how our method produces a combinatorial formula for $\fS^{\tra}_\gamma(\y;\y)$, thereby resolving Problem~\ref{Pb: TBPD 2}.
Set $a=-1$, $n=4$, and let $\gamma$ be the regular boundary condition satisfied by the BPD on the right of Example~\ref{Ex: twist j tile}. 
We enumerate $\TBPD_{a,n}(\gamma)$:
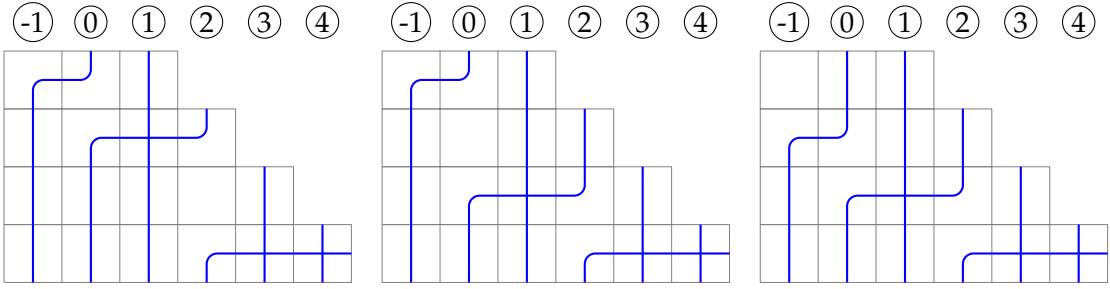
$$
\begin{tikzpicture}[x=2em,y=2em,thick,rounded corners,color = blue]
    \draw[step=1,gray,ultra thin] (0,0) grid (6,1);
    \draw[step=1,gray,ultra thin] (0,1) grid (5,2);
    \draw[step=1,gray,ultra thin] (0,2) grid (4,3);
    \draw[step=1,gray,ultra thin] (0,3) grid (3,4);
    \draw[color=blue, thick] (.5,0)--(.5, 3.5)--(1.5,3.5)--(1.5,4);
    \draw[color=blue, thick](1.5,0)--(1.5,2.5)--(3.5,2.5)--(3.5,3);
    \draw[color=blue, thick](2.5,0)--(2.5,4);
    \draw[color=blue, thick](3.5,0)--(3.5,.5)--(6, .5);
    \draw[color=blue, thick](4.5,0)--(4.5,2);
    \draw[color=blue, thick](5.5,0)--(5.5,1);
    \node[color=black] at (0.5,4.5) {$\circled{-1}$};
    \node[color=black] at (1.5,4.5) {$\circled{0}$};
    \node[color=black] at (2.5,4.5) {$\circled{1}$};
    \node[color=black] at (3.5,4.5) {$\circled{2}$};
    \node[color=black] at (4.5,4.5) {$\circled{3}$};
    \node[color=black] at (5.5,4.5) {$\circled{4}$};
\end{tikzpicture}
\quad
\begin{tikzpicture}[x=2em,y=2em,thick,rounded corners,color = blue]
    \draw[step=1,gray,ultra thin] (0,0) grid (6,1);
    \draw[step=1,gray,ultra thin] (0,1) grid (5,2);
    \draw[step=1,gray,ultra thin] (0,2) grid (4,3);
    \draw[step=1,gray,ultra thin] (0,3) grid (3,4);
    \draw[color=blue, thick] (0.5,0)--(0.5, 3.5)--(1.5,3.5)--(1.5,4);
    \draw[color=blue, thick](1.5,0)--(1.5,1.5)--(3.5,1.5)--(3.5,3);
    \draw[color=blue, thick](2.5,0)--(2.5,4);
    \draw[color=blue, thick](3.5,0)--(3.5,0.5)--(6, 0.5);
    \draw[color=blue, thick](4.5,0)--(4.5,2);
    \draw[color=blue, thick](5.5,0)--(5.5,1);
    \node[color=black] at (0.5,4.5) {$\circled{-1}$};
    \node[color=black] at (1.5,4.5) {$\circled{0}$};
    \node[color=black] at (2.5,4.5) {$\circled{1}$};
    \node[color=black] at (3.5,4.5) {$\circled{2}$};
    \node[color=black] at (4.5,4.5) {$\circled{3}$};
    \node[color=black] at (5.5,4.5) {$\circled{4}$};
\end{tikzpicture}
\quad
\begin{tikzpicture}[x=2em,y=2em,thick,rounded corners,color = blue]
    \draw[step=1,gray,ultra thin] (0,0) grid (6,1);
    \draw[step=1,gray,ultra thin] (0,1) grid (5,2);
    \draw[step=1,gray,ultra thin] (0,2) grid (4,3);
    \draw[step=1,gray,ultra thin] (0,3) grid (3,4);
    \draw[color=blue, thick] (0.5,0)--(0.5, 2.5)--(1.5,2.5)--(1.5,4);
    \draw[color=blue, thick](1.5,0)--(1.5,1.5)--(3.5,1.5)--(3.5,3);
    \draw[color=blue, thick](2.5,0)--(2.5,4);
    \draw[color=blue, thick](3.5,0)--(3.5,0.5)--(6, 0.5);
    \draw[color=blue, thick](4.5,0)--(4.5,2);
    \draw[color=blue, thick](5.5,0)--(5.5,1);
    \node[color=black] at (0.5,4.5) {$\circled{-1}$};
    \node[color=black] at (1.5,4.5) {$\circled{0}$};
    \node[color=black] at (2.5,4.5) {$\circled{1}$};
    \node[color=black] at (3.5,4.5) {$\circled{2}$};
    \node[color=black] at (4.5,4.5) {$\circled{3}$};
    \node[color=black] at (5.5,4.5) {$\circled{4}$};
\end{tikzpicture}
$$
Hence
\[
\fS^{\tra}_{a,n,\gamma}(\x;\y)=(x_3-y_2)+(x_2-y_0)+(x_1-y_{-1}).
\]
After specializing $\x\mapsto\y$, this does not yet resolve Problem~\ref{Pb: TBPD 2}, since the term $y_3-y_2$ is not type~3.

We now define a bijection
\[
\TBPD_{a,n}(\gamma)\;\xrightarrow{\;\;\sim\;\;}\;C(U,W,\rev(\alpha)),
\]
where $U=[-1,1,0,3,4,2]\in S_{[a,n]}$, $W=[0,1,2,3,4,-1]\in S_{[a,n]}$, and $\alpha=(0,1,2)$.
Under this bijection, the three BPDs above correspond, respectively, to the following elements of $C(U,W,\rev(\alpha))$ 
$$
\begin{tikzpicture}[x=2em,y=2em,thick,sharp corners, color = black]
      \draw (0,0) rectangle (6,4);
      \foreach \x in {1,...,5} \draw (\x,0) -- (\x,4);
      \foreach \y in {1,...,3} \draw (0,\y) -- (6,\y);
    
      \node at (0.5,3.5) {0};
      \node at (1.5,3.5) {1};
      \node at (2.5,3.5) {2};
      \node at (3.5,3.5) {3};
      \node at (4.5,3.5) {4};
      \node at (5.5,3.5) {-1};
    
      \node at (0.5,2.5) {-1};
      \node at (1.5,2.5) {1};
      \node at (2.5,2.5) {2};
      \node at (3.5,2.5) {3};
      \node at (4.5,2.5) {4};
      \node at (5.5,2.5) {0};
    
      \node at (0.5,1.5) {-1};
      \node at (1.5,1.5) {1};
      \node at (2.5,1.5) {0};
      \node at (3.5,1.5) {3};
      \node at (4.5,1.5) {4};
      \node at (5.5,1.5) {2};

      \node at (0.5,0.5) {-1};
      \node at (1.5,0.5) {1};
      \node at (2.5,0.5) {0};
      \node at (3.5,0.5) {3};
      \node at (4.5,0.5) {4};
      \node at (5.5,0.5) {2};

    \draw[red,very thick] (4,0) -- (4,2);   
    \draw[red,very thick] (3,1) -- (3,3);
    \draw[red,very thick] (2,2) -- (2,4);
    \draw[green,thick] (5.5,1) circle [x radius=0.35, y radius=0.8];
        \draw[blue,thick] (2.5,3) circle [x radius=0.35, y radius=0.8];
    \draw[blue,thick] (3.5,3) circle [x radius=0.35, y radius=0.8];
    \draw[blue,thick] (4.5,3) circle [x radius=0.35, y radius=0.8];
    \draw[blue,thick] (3.5,2) circle [x radius=0.35, y radius=0.8];
    \draw[blue,thick] (4.5,2) circle [x radius=0.35, y radius=0.8];
    \draw[blue,thick] (4.5,1) circle [x radius=0.35, y radius=0.8];
\end{tikzpicture}
\quad \quad 
\begin{tikzpicture}[x=2em,y=2em,thick,sharp corners, color = black]
      \draw (0,0) rectangle (6,4);
      \foreach \x in {1,...,5} \draw (\x,0) -- (\x,4);
      \foreach \y in {1,...,3} \draw (0,\y) -- (6,\y);
    
      \node at (0.5,3.5) {0};
      \node at (1.5,3.5) {1};
      \node at (2.5,3.5) {2};
      \node at (3.5,3.5) {3};
      \node at (4.5,3.5) {4};
      \node at (5.5,3.5) {-1};
    
      \node at (0.5,2.5) {-1};
      \node at (1.5,2.5) {1};
      \node at (2.5,2.5) {2};
      \node at (3.5,2.5) {3};
      \node at (4.5,2.5) {4};
      \node at (5.5,2.5) {0};
    
      \node at (0.5,1.5) {-1};
      \node at (1.5,1.5) {1};
      \node at (2.5,1.5) {2};
      \node at (3.5,1.5) {3};
      \node at (4.5,1.5) {4};
      \node at (5.5,1.5) {0};

      \node at (0.5,0.5) {-1};
      \node at (1.5,0.5) {1};
      \node at (2.5,0.5) {0};
      \node at (3.5,0.5) {3};
      \node at (4.5,0.5) {4};
      \node at (5.5,0.5) {2};

    \draw[red,very thick] (4,0) -- (4,2);   
    \draw[red,very thick] (3,1) -- (3,3);
    \draw[red,very thick] (2,2) -- (2,4);
    \draw[green,thick] (5.5,2) circle [x radius=0.35, y radius=0.8];
    \draw[blue,thick] (2.5,3) circle [x radius=0.35, y radius=0.8];
    \draw[blue,thick] (3.5,3) circle [x radius=0.35, y radius=0.8];
    \draw[blue,thick] (4.5,3) circle [x radius=0.35, y radius=0.8];
    \draw[blue,thick] (3.5,2) circle [x radius=0.35, y radius=0.8];
    \draw[blue,thick] (4.5,2) circle [x radius=0.35, y radius=0.8];
    \draw[blue,thick] (4.5,1) circle [x radius=0.35, y radius=0.8];
\end{tikzpicture}
\quad \quad 
\begin{tikzpicture}[x=2em,y=2em,thick,sharp corners, color = black]
      \draw (0,0) rectangle (6,4);
      \foreach \x in {1,...,5} \draw (\x,0) -- (\x,4);
      \foreach \y in {1,...,3} \draw (0,\y) -- (6,\y);
    
      \node at (0.5,3.5) {0};
      \node at (1.5,3.5) {1};
      \node at (2.5,3.5) {2};
      \node at (3.5,3.5) {3};
      \node at (4.5,3.5) {4};
      \node at (5.5,3.5) {-1};
    
      \node at (0.5,2.5) {0};
      \node at (1.5,2.5) {1};
      \node at (2.5,2.5) {2};
      \node at (3.5,2.5) {3};
      \node at (4.5,2.5) {4};
      \node at (5.5,2.5) {-1};
    
      \node at (0.5,1.5) {-1};
      \node at (1.5,1.5) {1};
      \node at (2.5,1.5) {2};
      \node at (3.5,1.5) {3};
      \node at (4.5,1.5) {4};
      \node at (5.5,1.5) {0};

      \node at (0.5,0.5) {-1};
      \node at (1.5,0.5) {1};
      \node at (2.5,0.5) {0};
      \node at (3.5,0.5) {3};
      \node at (4.5,0.5) {4};
      \node at (5.5,0.5) {2};

    \draw[red,very thick] (4,0) -- (4,2);   
    \draw[red,very thick] (3,1) -- (3,3);
    \draw[red,very thick] (2,2) -- (2,4);
    \draw[green,thick] (5.5,3) circle [x radius=0.35, y radius=0.8];
        \draw[blue,thick] (2.5,3) circle [x radius=0.35, y radius=0.8];
    \draw[blue,thick] (3.5,3) circle [x radius=0.35, y radius=0.8];
    \draw[blue,thick] (4.5,3) circle [x radius=0.35, y radius=0.8];
    \draw[blue,thick] (3.5,2) circle [x radius=0.35, y radius=0.8];
    \draw[blue,thick] (4.5,2) circle [x radius=0.35, y radius=0.8];
    \draw[blue,thick] (4.5,1) circle [x radius=0.35, y radius=0.8];
\end{tikzpicture}
$$

Notice that for $(u_n, \dots, u_1) \in C(U, W, \rev(\alpha))$,
we always have $[i+1, n] \subset \fix_{(\alpha_i, n]}(u_{i+1}, u_{i})$ for $i \in [n-1]$.
These fixed points are circled 
in blue while other fixed points are circled in green. 
Under our bijection, if $D$
corresponds to $(u_n, \dots, u_1)$, they satisfy:
\begin{equation}
\label{EQ: TBPD chain wt}
\wt(D)\cdot\prod_{1\le i<j\le n}(x_i-y_j)
=\wt_{\rev(\alpha)}^{n}(u_n,\dots,u_1)(x_{n-1},\dots,x_1;\y).
\end{equation}
Consequently,
\[
\fS^{\tra}_{a,n,\gamma}(\x;\y)
=\frac{\fC_{U,W,\rev(\alpha)}^{\,n}(x_{n-1},\dots,x_1;\y)}{\prod_{1\le i<j\le n}(x_i-y_j)}.
\]
By Proposition~\ref{P: double symmetry},
\[
\fC_{U,W,\rev(\alpha)}^{\,n}(x_{n-1},\dots,x_1;\y)
=\fC_{U,W,\alpha}^{\,n}(x_{1},\dots,x_{n-1};\y),
\]
and a direct enumeration of $C(U,W,\alpha)$ gives
\[
\begin{tikzpicture}[x=2em,y=2em,thick,sharp corners, color = black]
      \draw (0,0) rectangle (6,4);
      \foreach \x in {1,...,5} \draw (\x,0) -- (\x,4);
      \foreach \y in {1,...,3} \draw (0,\y) -- (6,\y);
    
      \node at (0.5,3.5) {0};
      \node at (1.5,3.5) {1};
      \node at (2.5,3.5) {2};
      \node at (3.5,3.5) {3};
      \node at (4.5,3.5) {4};
      \node at (5.5,3.5) {-1};
    
      \node at (0.5,2.5) {0};
      \node at (1.5,2.5) {1};
      \node at (2.5,2.5) {2};
      \node at (3.5,2.5) {3};
      \node at (4.5,2.5) {4};
      \node at (5.5,2.5) {-1};
    
      \node at (0.5,1.5) {0};
      \node at (1.5,1.5) {1};
      \node at (2.5,1.5) {-1};
      \node at (3.5,1.5) {3};
      \node at (4.5,1.5) {4};
      \node at (5.5,1.5) {2};

      \node at (0.5,0.5) {-1};
      \node at (1.5,0.5) {1};
      \node at (2.5,0.5) {0};
      \node at (3.5,0.5) {3};
      \node at (4.5,0.5) {4};
      \node at (5.5,0.5) {2};

    \draw[red,very thick] (4,2) -- (4,4);
    \draw[red,very thick] (3,1) -- (3,3);
    \draw[red,very thick] (2,0) -- (2,2);
    \draw[blue,thick] (3.5,1) circle [x radius=0.35, y radius=0.8];
    \draw[blue,thick] (4.5,1) circle [x radius=0.35, y radius=0.8];
    \draw[blue,thick] (5.5,1) circle [x radius=0.35, y radius=0.8];
    \draw[blue,thick] (3.5,2) circle [x radius=0.35, y radius=0.8];
    \draw[blue,thick] (4.5,2) circle [x radius=0.35, y radius=0.8];
    \draw[blue,thick] (4.5,3) circle [x radius=0.35, y radius=0.8];
    
    \draw[green,thick] (5.5,3) circle [x radius=0.35, y radius=0.8];
\end{tikzpicture}
\quad \quad 
\begin{tikzpicture}[x=2em,y=2em,thick,sharp corners, color = black]
      \draw (0,0) rectangle (6,4);
      \foreach \x in {1,...,5} \draw (\x,0) -- (\x,4);
      \foreach \y in {1,...,3} \draw (0,\y) -- (6,\y);
    
      \node at (0.5,3.5) {0};
      \node at (1.5,3.5) {1};
      \node at (2.5,3.5) {2};
      \node at (3.5,3.5) {3};
      \node at (4.5,3.5) {4};
      \node at (5.5,3.5) {-1};
    
      \node at (0.5,2.5) {0};
      \node at (1.5,2.5) {1};
      \node at (2.5,2.5) {-1};
      \node at (3.5,2.5) {3};
      \node at (4.5,2.5) {4};
      \node at (5.5,2.5) {2};
    
      \node at (0.5,1.5) {0};
      \node at (1.5,1.5) {1};
      \node at (2.5,1.5) {-1};
      \node at (3.5,1.5) {3};
      \node at (4.5,1.5) {4};
      \node at (5.5,1.5) {2};

      \node at (0.5,0.5) {-1};
      \node at (1.5,0.5) {1};
      \node at (2.5,0.5) {0};
      \node at (3.5,0.5) {3};
      \node at (4.5,0.5) {4};
      \node at (5.5,0.5) {2};

    \draw[red,very thick] (4,2) -- (4,4);
    \draw[red,very thick] (3,1) -- (3,3);
    \draw[red,very thick] (2,0) -- (2,2);
    \draw[blue,thick] (3.5,1) circle [x radius=0.35, y radius=0.8];
    \draw[blue,thick] (4.5,1) circle [x radius=0.35, y radius=0.8];
    \draw[blue,thick] (5.5,1) circle [x radius=0.35, y radius=0.8];
    \draw[blue,thick] (3.5,2) circle [x radius=0.35, y radius=0.8];
    \draw[blue,thick] (4.5,2) circle [x radius=0.35, y radius=0.8];
    \draw[blue,thick] (4.5,3) circle [x radius=0.35, y radius=0.8];
    \draw[green,thick] (5.5,2) circle [x radius=0.35, y radius=0.8];
\end{tikzpicture}
\quad \quad 
\begin{tikzpicture}[x=2em,y=2em,thick,sharp corners, color = black]
      \draw (0,0) rectangle (6,4);
      \foreach \x in {1,...,5} \draw (\x,0) -- (\x,4);
      \foreach \y in {1,...,3} \draw (0,\y) -- (6,\y);
    
      \node at (0.5,3.5) {0};
      \node at (1.5,3.5) {1};
      \node at (2.5,3.5) {2};
      \node at (3.5,3.5) {3};
      \node at (4.5,3.5) {4};
      \node at (5.5,3.5) {-1};
    
      \node at (0.5,2.5) {-1};
      \node at (1.5,2.5) {1};
      \node at (2.5,2.5) {2};
      \node at (3.5,2.5) {3};
      \node at (4.5,2.5) {4};
      \node at (5.5,2.5) {0};
    
      \node at (0.5,1.5) {-1};
      \node at (1.5,1.5) {1};
      \node at (2.5,1.5) {0};
      \node at (3.5,1.5) {3};
      \node at (4.5,1.5) {4};
      \node at (5.5,1.5) {2};

      \node at (0.5,0.5) {-1};
      \node at (1.5,0.5) {1};
      \node at (2.5,0.5) {0};
      \node at (3.5,0.5) {3};
      \node at (4.5,0.5) {4};
      \node at (5.5,0.5) {2};

    \draw[red,very thick] (4,2) -- (4,4);
    \draw[red,very thick] (3,1) -- (3,3);
    \draw[red,very thick] (2,0) -- (2,2);
    \draw[blue,thick] (3.5,1) circle [x radius=0.35, y radius=0.8];
    \draw[blue,thick] (4.5,1) circle [x radius=0.35, y radius=0.8];
    \draw[blue,thick] (5.5,1) circle [x radius=0.35, y radius=0.8];
    \draw[blue,thick] (3.5,2) circle [x radius=0.35, y radius=0.8];
    \draw[blue,thick] (4.5,2) circle [x radius=0.35, y radius=0.8];
    \draw[blue,thick] (4.5,3) circle [x radius=0.35, y radius=0.8];
    \draw[green,thick] (2.5,1) circle [x radius=0.35, y radius=0.8];
\end{tikzpicture}
\]
Observe that 
for all $(u_1, \dots, u_n) \in C(U, W, \alpha)$,
we still have $[i+1, n] \subset \fix_{(\alpha_i, n]}(u_{i}, u_{i+1})$ for $i \in [n-1]$.
These fixed points are circled 
in blue while other fixed points are circled in green.
Thus, we have
\[
\fS^{\tra}_{a,n,\gamma}(\x;\y)
=\frac{\fC_{U,W,\alpha}^{\,n}(x_{1},\dots,x_{n-1};\y)}{\prod_{1\le i<j\le n}(x_i-y_j)}
=(x_3-y_{-1})+(x_2-y_{2})+(x_1-y_{0}).
\]

Finally, upon specializing $\x\mapsto\y$, each summand either vanishes or is a distinct product of type~3 factors, exactly as required in Problem~\ref{Pb: TBPD 2}.

\subsection*{Main arguments}
By regularity, there are exactly \((n-a+1)\) pipes in \(D\in \TBPD_{a,n}(\gamma)\).
We label these pipes by \(a,a+1,\dots,n\) as follows:
\begin{itemize}
\item If a pipe enters from column \(c\) and another from column \(c'\) with \(c<c'\), then the former receives the smaller label.
\item The pipe entering from row \(r\) receives a larger label than the pipe entering from column \(r\), but a smaller label than the pipe entering from column \(r+1\).
\end{itemize}

This labeling is valid: along the top of each row, if two pipes have not yet crossed, then the left pipe has the smaller label. Unless stated otherwise, we use this labeling throughout. In addition, by Remark~\ref{R: label change}, we may replace \(\gamma\) by the boundary condition satisfied by \(D\) with this labeling; this does not change \(\TBPD_{a,n}(\gamma)\)
or $\fS^{\tra}_{a,n,\gamma}(\x; \y)$.

\begin{exa}
\label{Ex: Label TBPD}
Consider the following \(D\in \TBPD_{a,n}(\gamma)\) with \(a=-4\) and \(n=7\).
We place each pipe’s label near its entry point.
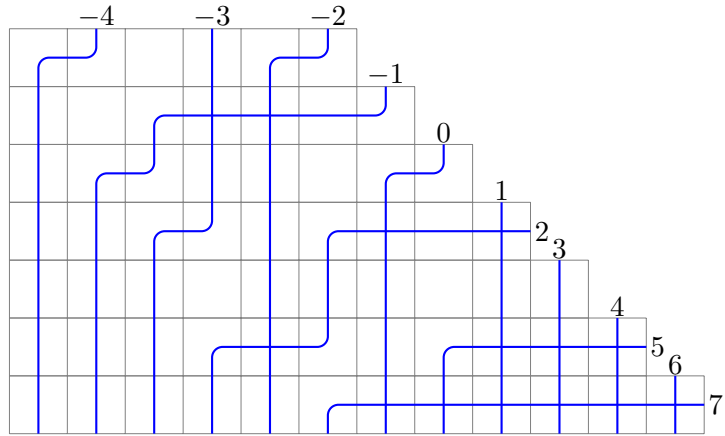
\[
\begin{tikzpicture}[x=2em,y=2em,thick,rounded corners,color = blue]
    \draw[step=1,gray,ultra thin] (0,0) grid (12,1);
    \draw[step=1,gray,ultra thin] (0,1) grid (11,2);
    \draw[step=1,gray,ultra thin] (0,2) grid (10,3);
    \draw[step=1,gray,ultra thin] (0,3) grid (9,4);
    \draw[step=1,gray,ultra thin] (0,4) grid (8,5);
    \draw[step=1,gray,ultra thin] (0,5) grid (7,6);
    \draw[step=1,gray,ultra thin] (0,6) grid (6,7);
    \draw[color=blue, thick] (0.5,0)--(0.5, 6.5)--(1.5, 6.5)--(1.5, 7);
    \draw[blue,thick] (1.5,0) -- (1.5,4.5) -- (2.5,4.5) -- (2.5,5.5) -- (6.5,5.5) -- (6.5,6);
    \draw[blue,thick] (2.5,0) -- (2.5,3.5) -- (3.5,3.5) -- (3.5,7);
    \draw[blue,thick] (3.5,0) -- (3.5,1.5) -- (5.5,1.5) -- (5.5,3.5) -- (9,3.5);
    \draw[blue,thick] (4.5,0) -- (4.5,6.5) -- (5.5,6.5) -- (5.5,7);
    \draw[blue,thick] (5.5,0) -- (5.5,0.5) -- (12,0.5);
    \draw[blue,thick] (6.5,0) -- (6.5,4.5) -- (7.5,4.5) -- (7.5,5);
    \draw[blue,thick] (7.5,0) -- (7.5,1.5) -- (11,1.5);
    \draw[blue,thick] (8.5,0) -- (8.5,4);
    \draw[blue,thick] (9.5,0) -- (9.5,3);
    \draw[blue,thick] (10.5,0) -- (10.5,2);
    \draw[blue,thick] (11.5,0) -- (11.5,1);
    \node[color=black] at (1.5,7.2) {$-4$};
    \node[color=black] at (3.5,7.2) {$-3$};
    \node[color=black] at (5.5,7.2) {$-2$};
    \node[color=black] at (6.5,6.2) {$-1$};
    \node[color=black] at (7.5,5.2) {$0$};
    \node[color=black] at (8.5,4.2) {$1$};
    \node[color=black] at (9.2,3.5) {$2$};
    \node[color=black] at (9.5,3.2) {$3$};
    \node[color=black] at (10.5,2.2) {$4$};
    \node[color=black] at (11.2,1.5) {$5$};
    \node[color=black] at (11.5,1.2) {$6$};
    \node[color=black] at (12.2,0.5) {$7$};
\end{tikzpicture}
\]
\end{exa}

In the remainder of the paper we fix such a \(\gamma\).
By assumption there is a pipe entering from column \(c\) for each \(c=1,\dots,n\).
Let \(p_c\) be the label of the pipe entering from column \(c\) (i.e., \(p_c=\gamma_N(c)\)).
Note that$ p_{r+1}\in\{p_r+1,\;p_r+2\}$,
depending on whether there is a pipe entering from row \(r\).

\begin{exa}
Continuing Example~\ref{Ex: Label TBPD}, we have
\[
p_1=-2,\quad
p_2=-1,\quad
p_3=0,\quad
p_4=1,\quad
p_5=3,\quad
p_6=4,\quad
p_7=6.\qedhere
\]
\end{exa}

\begin{defn}
\label{D: P_r}
For \(r\in[0,n]\), let \(P_r\) be the set of permutations \(w\in S_{[a,n]}\) such that \(w(p_j)=j\) for all \(j\in(r,n]\).

There is a bijection \(\iota_r:P_r\to S_{[a,r]}\): the one-line notation of \(\iota_r(u)\in S_{[a,r]}\) is obtained from that of \(u\in S_{[a,n]}\) by deleting the entries \(r+1,\dots,n\).
\end{defn}

We next define a sequence of permutations \(u_1,\dots,u_n\) for \(D\in \TBPD_{a,n}(\gamma)\), analogous to Definition~\ref{D: BPD chain}.

\begin{defn}\label{D: u from TBPD}
Fix \(D\in \TBPD_{a,n}(\gamma)\) and \(r\in[n]\).
Let \(\varphi_r:[a,r]\to [a,p_r]\sqcup\{\varnothing\}\) be the cross section of \(D\) above row \(r\).
Define \(u_r\) by
\begin{equation}
\label{EQ: define u_r}
\text{\(u_r\in P_r\) such that  \(\iota_r(u_r)\) agrees with \(\varphi_r^{-1}\) on \([a,p_r]\) and is decreasing on \((p_r,r]\).}
\end{equation}
Finally, set \(\chain(D)=(u_n,\dots,u_1)\).
\end{defn}

We notice that \(u_1\) and \(u_n\) can be described without mentioning \(D\).

\begin{defn}
\label{D: U W alpha}
Define \(W\in P_1\) to be the permutation such that \(\iota_1(W)\) agrees with \(\gamma_N^{-1}\) on \([a,p_1]\) and is decreasing on \((p_1,1]\).
Let \(U\in S_{[a,n]}\) be the permutation agreeing with \(\gamma_S^{-1}\) on \([a,n]\).
For \(i\in[0, n-1]\) set \(\alpha_i:=p_{i+1}-1\) and write \(\alpha=(\alpha_1,\dots,\alpha_{n-1})\).
\end{defn}

Our main result in this section is as follows.
\begin{pro}
\label{P: TBPD bijection}
The map \(\chain\) is a bijection
\[
\chain:\ \TBPD_{a,n}(\gamma)\longrightarrow C(U,W,\rev(\alpha)).
\]
Moreover, if \(\chain(D)=(u_n,\dots,u_1)\), then they satisfy~\eqref{EQ: TBPD chain wt}.
\end{pro}

\begin{exa}
\label{Ex: TBPD chain}
Consider \(D\in \TBPD_{a,n}(\gamma)\) in Example~\ref{Ex: Label TBPD}.
We compute \(\chain_{-4,7}(D)=(u_7,\dots,u_1)\in C(U,W,\rev(\alpha))\) with \(\alpha=(-2,-1,0,2,3,5)\).
\[
\begin{tikzpicture}[scale=0.7]
  \draw (0,0) rectangle (12,7);
  \foreach \x in {1,...,11} \draw (\x,0) -- (\x,7);
  \foreach \y in {1,...,6}  \draw (0,\y) -- (12,\y);

  \node at (0.5,6.5) {-3}; \node at (1.5,6.5) {-1}; \node at (2.5,6.5) {\textcolor{blue}{1}};
  \node at (3.5,6.5) {\textcolor{blue}{2}};  \node at (4.5,6.5) {\textcolor{blue}{3}};  \node at (5.5,6.5) {\textcolor{blue}{4}};
  \node at (6.5,6.5) {0};  \node at (7.5,6.5) {\textcolor{blue}{5}};  \node at (8.5,6.5) {\textcolor{blue}{6}};
  \node at (9.5,6.5) {-2}; \node at (10.5,6.5) {\textcolor{blue}{7}}; \node at (11.5,6.5) {-4};

  \node at (0.5,5.5) {-4}; \node at (1.5,5.5) {-1}; \node at (2.5,5.5) {0};
  \node at (3.5,5.5) {\textcolor{blue}{2}};  \node at (4.5,5.5) {\textcolor{blue}{3}};  \node at (5.5,5.5) {\textcolor{blue}{4}};
  \node at (6.5,5.5) {1};  \node at (7.5,5.5) {\textcolor{blue}{5}};  \node at (8.5,5.5) {\textcolor{blue}{6}};
  \node at (9.5,5.5) {-2}; \node at (10.5,5.5) {\textcolor{blue}{7}}; \node at (11.5,5.5) {-3};

  \node at (0.5,4.5) {-4}; \node at (1.5,4.5) {-1}; \node at (2.5,4.5) {0};
  \node at (3.5,4.5) {-2}; \node at (4.5,4.5) {\textcolor{blue}{3}};  \node at (5.5,4.5) {\textcolor{blue}{4}};
  \node at (6.5,4.5) {2};  \node at (7.5,4.5) {\textcolor{blue}{5}};  \node at (8.5,4.5) {\textcolor{blue}{6}};
  \node at (9.5,4.5) {1};  \node at (10.5,4.5) {\textcolor{blue}{7}}; \node at (11.5,4.5) {-3};

  \node at (0.5,3.5) {-4}; \node at (1.5,3.5) {-1}; \node at (2.5,3.5) {0};
  \node at (3.5,3.5) {-3}; \node at (4.5,3.5) {2};  \node at (5.5,3.5) {\textcolor{blue}{4}};
  \node at (6.5,3.5) {3};  \node at (7.5,3.5) {\textcolor{blue}{5}};  \node at (8.5,3.5) {\textcolor{blue}{6}};
  \node at (9.5,3.5) {1};  \node at (10.5,3.5) {\textcolor{blue}{7}}; \node at (11.5,3.5) {-2};

  \node at (0.5,2.5) {-4}; \node at (1.5,2.5) {-2}; \node at (2.5,2.5) {0};
  \node at (3.5,2.5) {-3}; \node at (4.5,2.5) {2};  \node at (5.5,2.5) {4};
  \node at (6.5,2.5) {1};  \node at (7.5,2.5) {\textcolor{blue}{5}};  \node at (8.5,2.5) {\textcolor{blue}{6}};
  \node at (9.5,2.5) {3};  \node at (10.5,2.5) {\textcolor{blue}{7}}; \node at (11.5,2.5) {-1};

  \node at (0.5,1.5) {-4}; \node at (1.5,1.5) {-2}; \node at (2.5,1.5) {0};
  \node at (3.5,1.5) {-3}; \node at (4.5,1.5) {2};  \node at (5.5,1.5) {4};
  \node at (6.5,1.5) {1};  \node at (7.5,1.5) {5};  \node at (8.5,1.5) {\textcolor{blue}{6}};
  \node at (9.5,1.5) {3};  \node at (10.5,1.5) {\textcolor{blue}{7}}; \node at (11.5,1.5) {-1};

  \node at (0.5,0.5) {-4}; \node at (1.5,0.5) {-2}; \node at (2.5,0.5) {0};
  \node at (3.5,0.5) {-3}; \node at (4.5,0.5) {2};  \node at (5.5,0.5) {4};
  \node at (6.5,0.5) {-1}; \node at (7.5,0.5) {5}; \node at (8.5,0.5) {6};
  \node at (9.5,0.5) {3};  \node at (10.5,0.5) {\textcolor{blue}{7}}; \node at (11.5,0.5) {1};

  \draw[red,very thick] (3,5) -- (3,7);
  \draw[red,very thick] (4,4) -- (4,6);
  \draw[red,very thick] (5,3) -- (5,5);
  \draw[red,very thick] (7,2) -- (7,4);
  \draw[red,very thick] (8,1) -- (8,3);
  \draw[red,very thick] (10,0) -- (10,2);
  \draw[green,thick] (9.5,2) circle [x radius=0.35, y radius=0.8];
  \draw[green,thick] (11.5,2) circle [x radius=0.35, y radius=0.8];
  \draw[green,thick] (9.5,4) circle [x radius=0.35, y radius=0.8];
  \draw[green,thick] (11.5,5) circle [x radius=0.35, y radius=0.8];
  \draw[green,thick] (9.5,6) circle [x radius=0.35, y radius=0.8];
\end{tikzpicture}
\]
We record an observation, proved in general later in Lemma~\ref{L: Chain fix points}: for \(r\le j\), the value \(u_r(p_j)\) is forced to be \(j\).
We color these entries blue.
Since \(p_j>\alpha_i\) for all \(i\in [j-1]\), we have \(j\in \fix_{(\alpha_i,n]}(u_{i+1},u_i)\).
These entries contribute \(\prod_{1\le i<j\le n}(x_i-y_j)\) to the right-hand side of~\eqref{EQ: TBPD chain wt}.
The remaining contributions are circled above, yielding
\[
(x_5-y_3)(x_5-y_{-1})(x_3-y_1)(x_2-y_{-3})(x_1-y_{-2})=\wt(D). \qedhere
\]
\end{exa}

The proof is deferred to the next section.
We now translate Proposition~\ref{P: TBPD bijection} into an identity of generating functions.

\begin{cor}
\label{C: TBPD chain}
We have
\begin{equation}
\label{EQ: TBPD chain}
\fS^{\tra}_{a,n,\gamma}(\x;\y)
=
\sum_{D\in \TBPD_{a,n}(\gamma)} \wt(D)
=
\frac{\fC^{n}_{U,W,\rev(\alpha)}(x_{n-1},\dots,x_1;\y)}{\displaystyle\prod_{1\le i<j\le n}(x_i-y_j)}
=
\frac{\fC^{n}_{U,W,\alpha}(x_{1},\dots,x_{n-1};\y)}{\displaystyle\prod_{1\le i<j\le n}(x_i-y_j)}.
\end{equation}
\end{cor}
\begin{proof}
The first equality is the definition.
The second follows from Proposition~\ref{P: TBPD bijection}.
The last follows from Proposition~\ref{P: double symmetry}.
\end{proof}

We now study \(C(U,W,\alpha)\).

\begin{exa}
\label{Ex: TBPD reversed chains}
We present two elements \((u_1,\dots,u_n)\) and \((v_1,\dots,v_n)\) in \(C(U,W,\alpha)\) in our running example (Examples~\ref{Ex: Label TBPD} and~\ref{Ex: TBPD chain}).
Here is \((u_1,\dots,u_n)\):
\[
\begin{tikzpicture}[scale=0.7]
  \draw (0,0) rectangle (12,7);
  \foreach \x in {1,...,11} \draw (\x,0) -- (\x,7);
  \foreach \y in {1,...,6}  \draw (0,\y) -- (12,\y);

  \node at (0.5,6.5) {-3}; \node at (1.5,6.5) {-1}; \node at (2.5,6.5) {1};
  \node at (3.5,6.5) {2};  \node at (4.5,6.5) {3};  \node at (5.5,6.5) {4};
  \node at (6.5,6.5) {0};  \node at (7.5,6.5) {5};  \node at (8.5,6.5) {6};
  \node at (9.5,6.5) {-2}; \node at (10.5,6.5) {\textcolor{blue}{7}}; \node at (11.5,6.5) {-4};

  \node at (0.5,5.5) {-3}; \node at (1.5,5.5) {-1}; \node at (2.5,5.5) {1};
  \node at (3.5,5.5) {2};  \node at (4.5,5.5) {3};  \node at (5.5,5.5) {4};
  \node at (6.5,5.5) {0};  \node at (7.5,5.5) {5};  \node at (8.5,5.5) {\textcolor{blue}{6}};
  \node at (9.5,5.5) {-2}; \node at (10.5,5.5) {\textcolor{blue}{7}}; \node at (11.5,5.5) {-4};

  \node at (0.5,4.5) {-3}; \node at (1.5,4.5) {-1}; \node at (2.5,4.5) {1};
  \node at (3.5,4.5) {2};  \node at (4.5,4.5) {3};  \node at (5.5,4.5) {4};
  \node at (6.5,4.5) {-2}; \node at (7.5,4.5) {\textcolor{blue}{5}}; \node at (8.5,4.5) {\textcolor{blue}{6}};
  \node at (9.5,4.5) {0};  \node at (10.5,4.5) {\textcolor{blue}{7}}; \node at (11.5,4.5) {-4};

  \node at (0.5,3.5) {-3}; \node at (1.5,3.5) {-1}; \node at (2.5,3.5) {1};
  \node at (3.5,3.5) {2};  \node at (4.5,3.5) {0};  \node at (5.5,3.5) {\textcolor{blue}{4}};
  \node at (6.5,3.5) {-4}; \node at (7.5,3.5) {\textcolor{blue}{5}}; \node at (8.5,3.5) {\textcolor{blue}{6}};
  \node at (9.5,3.5) {3};  \node at (10.5,3.5) {\textcolor{blue}{7}}; \node at (11.5,3.5) {-2};

  \node at (0.5,2.5) {-3}; \node at (1.5,2.5) {-1}; \node at (2.5,2.5) {1};
  \node at (3.5,2.5) {2};  \node at (4.5,2.5) {-4}; \node at (5.5,2.5) {\textcolor{blue}{4}};
  \node at (6.5,2.5) {-2}; \node at (7.5,2.5) {\textcolor{blue}{5}}; \node at (8.5,2.5) {\textcolor{blue}{6}};
  \node at (9.5,2.5) {\textcolor{blue}{3}};  \node at (10.5,2.5) {\textcolor{blue}{7}}; \node at (11.5,2.5) {0};

  \node at (0.5,1.5) {-3}; \node at (1.5,1.5) {-2}; \node at (2.5,1.5) {1};
  \node at (3.5,1.5) {-4}; \node at (4.5,1.5) {\textcolor{blue}{2}}; \node at (5.5,1.5) {\textcolor{blue}{4}};
  \node at (6.5,1.5) {-1}; \node at (7.5,1.5) {\textcolor{blue}{5}}; \node at (8.5,1.5) {\textcolor{blue}{6}};
  \node at (9.5,1.5) {\textcolor{blue}{3}}; \node at (10.5,1.5) {\textcolor{blue}{7}}; \node at (11.5,1.5) {0};

  \node at (0.5,0.5) {-4}; \node at (1.5,0.5) {-2}; \node at (2.5,0.5) {0};
  \node at (3.5,0.5) {-3}; \node at (4.5,0.5) {\textcolor{blue}{2}}; \node at (5.5,0.5) {\textcolor{blue}{4}};
  \node at (6.5,0.5) {-1}; \node at (7.5,0.5) {\textcolor{blue}{5}}; \node at (8.5,0.5) {\textcolor{blue}{6}};
  \node at (9.5,0.5) {\textcolor{blue}{3}}; \node at (10.5,0.5) {\textcolor{blue}{7}}; \node at (11.5,0.5) {\textcolor{blue}{1}};

  \draw[red,very thick] (3,0) -- (3,2);   
  \draw[red,very thick] (4,1) -- (4,3);   
  \draw[red,very thick] (5,2) -- (5,4);   
  \draw[red,very thick] (7,3) -- (7,5);
  \draw[red,very thick] (8,4) -- (8,6);
  \draw[red,very thick] (10,5) -- (10,7); 

  \draw[green,thick] (6.5,1) circle [x radius=0.35, y radius=0.8];
  \draw[green,thick] (11.5,2) circle [x radius=0.35, y radius=0.8];
  \draw[green,thick] (9.5,3) circle [x radius=0.35, y radius=0.8];
  \draw[green,thick] (11.5,5) circle [x radius=0.35, y radius=0.8];
  \draw[green,thick] (11.5,6) circle [x radius=0.35, y radius=0.8];
\end{tikzpicture}
\]

We color the value \(j\) in blue across all \(u_i\) with \(i\in [j]\).
Notice that \(u_i^{-1}(j)\) is the same for all \(i\in [j]\), and this position is larger than \(\alpha_1,\dots,\alpha_{j-1}\).
Thus \(j\in \fix_{(\alpha_i,n]}(u_i,u_{i+1})\) for all \(i\in [j-1]\).
They contribute \(\prod_{1\le i<j\le n}(x_i-y_j)\) to \(\wt_{\alpha}^n(u_1,\dots,u_n)(x_1,\dots,x_n;\y)\).
The remaining contributors are circled in green, so
\[
\wt_\alpha^n(u_1,\dots,u_n)(x_1,\dots,x_{n-1};\y)
=
\Bigl(\prod_{1\le i<j\le n}(x_i-y_j)\Bigr)
(x_1-y_{-1})(x_2-y_0)(x_3-y_3)(x_5-y_{-4})(x_6-y_{-4}).
\]

Here is \((v_1,\dots,v_n)\):
\[
\begin{tikzpicture}[scale=0.7]
  \draw (0,0) rectangle (12,7);
  \foreach \x in {1,...,11} \draw (\x,0) -- (\x,7);
  \foreach \y in {1,...,6}  \draw (0,\y) -- (12,\y);

  \node at (0.5,6.5) {-3}; \node at (1.5,6.5) {-1}; \node at (2.5,6.5) {\textcolor{orange}{1}};
  \node at (3.5,6.5) {\textcolor{orange}{2}};  \node at (4.5,6.5) {\textcolor{orange}{3}};  \node at (5.5,6.5) {\textcolor{orange}{4}};
  \node at (6.5,6.5) {0};  \node at (7.5,6.5) {\textcolor{orange}{5}};  \node at (8.5,6.5) {\textcolor{orange}{6}};
  \node at (9.5,6.5) {-2}; \node at (10.5,6.5) {\textcolor{blue}{7}}; \node at (11.5,6.5) {-4};

  \node at (0.5,5.5) {-3}; \node at (1.5,5.5) {-1}; \node at (2.5,5.5) {\textcolor{orange}{1}};
  \node at (3.5,5.5) {\textcolor{orange}{2}};  \node at (4.5,5.5) {\textcolor{orange}{3}};  \node at (5.5,5.5) {\textcolor{orange}{4}};
  \node at (6.5,5.5) {0};  \node at (7.5,5.5) {\textcolor{orange}{5}};  \node at (8.5,5.5) {\textcolor{blue}{6}};
  \node at (9.5,5.5) {-4}; \node at (10.5,5.5) {\textcolor{blue}{7}}; \node at (11.5,5.5) {-2};

  \node at (0.5,4.5) {-4}; \node at (1.5,4.5) {-1}; \node at (2.5,4.5) {\textcolor{orange}{1}};
  \node at (3.5,4.5) {\textcolor{orange}{2}};  \node at (4.5,4.5) {\textcolor{orange}{3}};  \node at (5.5,4.5) {\textcolor{orange}{4}};
  \node at (6.5,4.5) {0};  \node at (7.5,4.5) {\textcolor{blue}{5}};  \node at (8.5,4.5) {\textcolor{blue}{6}};
  \node at (9.5,4.5) {-3}; \node at (10.5,4.5) {\textcolor{blue}{7}}; \node at (11.5,4.5) {-2};

  \node at (0.5,3.5) {-4}; \node at (1.5,3.5) {-1}; \node at (2.5,3.5) {\textcolor{orange}{1}};
  \node at (3.5,3.5) {\textcolor{orange}{2}};  \node at (4.5,3.5) {\textcolor{orange}{3}};  \node at (5.5,3.5) {\textcolor{blue}{4}};
  \node at (6.5,3.5) {-3}; \node at (7.5,3.5) {\textcolor{blue}{5}}; \node at (8.5,3.5) {\textcolor{blue}{6}};
  \node at (9.5,3.5) {-2}; \node at (10.5,3.5) {\textcolor{blue}{7}}; \node at (11.5,3.5) {0};

  \node at (0.5,2.5) {-4}; \node at (1.5,2.5) {-1}; \node at (2.5,2.5) {\textcolor{orange}{1}};
  \node at (3.5,2.5) {\textcolor{orange}{2}};  \node at (4.5,2.5) {-3}; \node at (5.5,2.5) {\textcolor{blue}{4}};
  \node at (6.5,2.5) {-2}; \node at (7.5,2.5) {\textcolor{blue}{5}}; \node at (8.5,2.5) {\textcolor{blue}{6}};
  \node at (9.5,2.5) {\textcolor{blue}{3}};  \node at (10.5,2.5) {\textcolor{blue}{7}}; \node at (11.5,2.5) {0};

  \node at (0.5,1.5) {-4}; \node at (1.5,1.5) {-1}; \node at (2.5,1.5) {\textcolor{orange}{1}};
  \node at (3.5,1.5) {-3}; \node at (4.5,1.5) {\textcolor{blue}{2}}; \node at (5.5,1.5) {\textcolor{blue}{4}};
  \node at (6.5,1.5) {-2}; \node at (7.5,1.5) {\textcolor{blue}{5}}; \node at (8.5,1.5) {\textcolor{blue}{6}};
  \node at (9.5,1.5) {\textcolor{blue}{3}}; \node at (10.5,1.5) {\textcolor{blue}{7}}; \node at (11.5,1.5) {0};

  \node at (0.5,0.5) {-4}; \node at (1.5,0.5) {-2}; \node at (2.5,0.5) {0};
  \node at (3.5,0.5) {-3}; \node at (4.5,0.5) {\textcolor{blue}{2}}; \node at (5.5,0.5) {\textcolor{blue}{4}};
  \node at (6.5,0.5) {-1}; \node at (7.5,0.5) {\textcolor{blue}{5}}; \node at (8.5,0.5) {\textcolor{blue}{6}};
  \node at (9.5,0.5) {\textcolor{blue}{3}}; \node at (10.5,0.5) {\textcolor{blue}{7}}; \node at (11.5,0.5) {\textcolor{blue}{1}};

  \draw[red,very thick] (3,0) -- (3,2);   
  \draw[red,very thick] (4,1) -- (4,3);   
  \draw[red,very thick] (5,2) -- (5,4);   
  \draw[red,very thick] (7,3) -- (7,5);
  \draw[red,very thick] (8,4) -- (8,6);
  \draw[red,very thick] (10,5) -- (10,7); 

  \draw[green,thick] (3.5,1) circle [x radius=0.35, y radius=0.8];
  \draw[green,thick] (6.5,2) circle [x radius=0.35, y radius=0.8];
  \draw[green,thick] (11.5,2) circle [x radius=0.35, y radius=0.8];
  \draw[green,thick] (11.5,3) circle [x radius=0.35, y radius=0.8];
  \draw[green,thick] (11.5,5) circle [x radius=0.35, y radius=0.8];
\end{tikzpicture}
\]
Again, \(j\in \fix_{(\alpha_i,n]}(u_i,u_{i+1})\) for all \(i\in [j-1]\).
Thus
\[
\wt_\alpha^n(v_1,\dots,v_n)(x_1,\dots,x_{n-1};\y)
=
\Bigl(\prod_{1\le i<j\le n}(x_i-y_j)\Bigr)
(x_1-y_{-3})(x_2-y_{-2})(x_2-y_0)(x_3-y_{0})(x_5-y_{-2}).
\]
We also note that \(v_i(p_j)=j\) for all \(1\le j<i\le n\); these entries are colored orange. This property fails for \((u_1,\dots,u_n)\) since \(u_4(p_3)=0\neq 3\).
\end{exa}

We summarize the preceding observation
regarding the blue numbers.
The proof is delayed to the next section.

\begin{lem}
\label{L: in fixed for increasing}
For \((u_1,\dots,u_n)\in C(U,W,\alpha)\) and each \(i\in[n-1]\), we have
\(
[i+1,n]\subseteq \fix_{(\alpha_i,n]}(u_i,u_{i+1}).
\)
\end{lem}

Consequently,
\(\prod_{1\le i<j\le n}(x_i-y_j)\) divides \(\wt_\alpha^n(u_1,\dots,u_n)(x_1,\dots,x_{n-1};\y)\) for every \((u_1,\dots,u_n)\in C(U,W,\alpha)\).
Hence we may rewrite~\eqref{EQ: TBPD chain} as
\begin{equation}
\label{EQ: tilde wt}
\fS^{\tra}_{a,n,\gamma}(\x;\y)
=
\sum_{(u_1,\dots,u_n)\in C(U,W,\alpha)}
\widetilde{\wt_\alpha^n}(u_1,\dots,u_n)(x_1,\dots,x_{n-1};\y),    
\end{equation}
where
\begin{equation}
\label{EQ: normalized wt}
\widetilde{\wt_\alpha^n}(x_1,\dots,x_{n-1};\y)
:=
\frac{\wt_\alpha^n(x_1,\dots,x_{n-1};\y)}{\displaystyle\prod_{1\le i<j\le n}(x_i-y_j)}
=
\prod_{i\in[n-1]}\ \prod_{j\in \fix_{(\alpha_i,n]}(u_i,u_{i+1})\cap [a,i]} (x_i-y_j).
\end{equation}

\begin{exa}
\label{Ex: tilde wt}
Continuing Example~\ref{Ex: TBPD reversed chains}, we obtain
\[
\begin{aligned}
\widetilde{\wt_\alpha^n}(u_1,\dots,u_n)(x_1,\dots,x_{n-1};\y)
&=(x_1-y_{-1})(x_2-y_0)\textcolor{red}{(x_3-y_3)}(x_5-y_{-4})(x_6-y_{-4}),\\
\widetilde{\wt_\alpha^n}(v_1,\dots,v_n)(x_1,\dots,x_{n-1};\y)
&=(x_1-y_{-3})(x_2-y_{-2})(x_2-y_0)(x_3-y_{0})(x_5-y_{-2}).\qedhere
\end{aligned}
\]
\end{exa}

Next, consider the specialization \(\x\mapsto \y\).
For \((u_1,\dots,u_n)\in C(U,W,\alpha)\), its contribution \eqref{EQ: normalized wt} vanishes unless
\begin{equation}
\label{EQ: Non-vanishing condition}
i\notin \fix_{(\alpha_i,n]}(u_i,u_{i+1})\quad\text{for all }i\in[n-1].
\end{equation}
Let \(\widetilde{C}(U,W,\alpha)\subseteq C(U,W,\alpha)\) be the subset satisfying \eqref{EQ: Non-vanishing condition}. Then
\[
\fS^{\tra}_{a,n,\gamma}(\x;\y)
=
\sum_{(u_1,\dots,u_n)\in \widetilde{C}(U,W,\alpha)}
\widetilde{\wt_\alpha^n}(u_1,\dots,u_n)(x_1,\dots,x_{n-1};\y).
\]
In Examples~\ref{Ex: TBPD reversed chains} and~\ref{Ex: tilde wt},  since \(3\in \fix_{(\alpha_3,n]}(u_3,u_4)\), we have \(\widetilde{\wt_\alpha^n}(u_1,\dots,u_n)(y_1,\dots,y_{n-1};\y)=0\). Hence \((u_1,\dots,u_n)\notin \widetilde{C}(U,W,\alpha)\), whereas \((v_1,\dots,v_n)\in \widetilde{C}(U,W,\alpha)\).

Finally, we establish 
the observation in Example~\ref{Ex: TBPD reversed chains} regarding
the orange numbers:
\begin{lem}
\label{L: fixed numbers in the final}
For \((u_1,\dots,u_n)\in \widetilde{C}(U,W,\alpha)\), we have \(u_i(p_j)=j\) for all \(i>j\).
In particular, 
$$\fix_{(\alpha_i,n]}(u_i,u_{i+1})\cap [a,i]\subseteq [a,0].$$
Thus, 
$\widetilde{\wt_\alpha^n}(u_1,\dots,u_n)(y_1,\dots,y_{n-1};\y)$
is a distinct product
of $(y_i - y_j)$
where $i \in [n-1]$
and $j \in [a,0]$.
\end{lem}

Collecting everything, we conclude:

\begin{pro}
\label{P: TBPD Final}
We have
\[
\fS^{\tra}_{a,n,\gamma}(\y;\y)
=
\sum_{(u_1,\dots,u_n)\in \widetilde{C}(U,W,\alpha)}
\widetilde{\wt_\alpha^n}(u_1,\dots,u_n)(y_1,\dots,y_{n-1};\y),
\]
where each summand on the right-hand side is a distinct product of type~3 terms.
This resolves Problem~\ref{Pb: TBPD 2}.
\end{pro}
\begin{proof}
This follows from \eqref{EQ: tilde wt} and Lemma~\ref{L: fixed numbers in the final}.
\end{proof}

\subsection*{Proofs}
This section presents the proofs deferred from the previous section:
Proposition~\ref{P: TBPD bijection}, Lemma~\ref{L: in fixed for increasing}, and
Lemma~\ref{L: fixed numbers in the final}.

To prove Proposition~\ref{P: TBPD bijection}, we begin by analyzing the last row of
\(D\in\TBPD_{a,n}(\gamma)\).

\begin{rem}\label{R: Last row}
All \(D \in \TBPD_{a,n}(\gamma)\) agree on row \(n\).
We describe this row by considering two cases.
\begin{itemize}
\item If \(\gamma_E(n)=\varnothing\), there is no pipe entering from row \(n\).
Then the pipe entering from column \(n\) has label \(n\) (i.e.\ \(p_n=n\)).
Moreover, in row \(n\) of \(D\), there are \(n-a+1\) pipes entering from the top and
\(n-a+1\) pipes exiting from the bottom, so this row contains only \(\vtile\).

\item Otherwise, pipe \(n-1\) enters from column \(n\) and pipe \(n\) enters from row \(n\).
In row \(n\), pipe \(n\) travels to the unique column \(c\) with \(\gamma_S(c)=n\) and exits there.
Thus, the tile \((n,c)\) is an \(\rtile\), and to its right (resp.\ left) all tiles are \(\ptile\) (resp.\ \(\vtile\)). \qedhere
\end{itemize}
\end{rem}

We first identify the endpoints of \(\chain(D)\).

\begin{lem}\label{L: U W}
Let \((u_n,\dots,u_1)=\chain(D)\).
Then \(u_n=U\) and \(u_1=W\).
\end{lem}
\begin{proof}
Let \(\varphi_r\) denote the cross section of \(D\) above row \(r\) for \(r\in[n+1]\).
By Remark~\ref{R: Top row bottom row cs}, \(\varphi_1\) agrees with \(\gamma_N\) on \([a,1]\),
so \(\varphi_1^{-1}\) agrees with \(\gamma_N^{-1}\) on \([a,p_1]\).
Comparing Definition~\ref{D: u from TBPD} and Definition~\ref{D: U W alpha}, we obtain \(u_1=W\).

Similarly, Remark~\ref{R: Top row bottom row cs} gives \(\varphi_{n+1}=\gamma_S\) on \([a,n]\).
Since \(P_n=S_{[a,n]}\) and \(\iota_n=\mathrm{id}\),
we translate~\eqref{EQ: define u_r} in 
Definition~\ref{D: u from TBPD} as
\begin{equation}\label{EQ: describing u_n}
\text{\(u_n \in S_{[a,n]}\) agrees with \(\varphi_n^{-1}\) on \([a,p_n]\) and is decreasing on \((p_n,n]\).}
\end{equation}
Now consider the two cases from Remark~\ref{R: Last row}.
If \(\gamma_E(n)=\varnothing\), then \(p_n=n\) and row \(n\) is all \(\vtile\), so \(\varphi_n=\varphi_{n+1}\).
Plugging \(p_n=n\) into~\eqref{EQ: describing u_n} shows $u_n$ agrees with$ \varphi_n^{-1}=\gamma_S^{-1}$ on $[a, n]$,
so $u_n = U$.
Otherwise, \(\gamma_E(n)=n\) and \(p_n=n-1\).
If \(\gamma_S(c)=n\), then \(\varphi_n(c)=\varnothing\) and \(\varphi_n(i)=\varphi_{n+1}(i)=\gamma_S(i)\) for \(i\in [a,n] \setminus \{c\}\).
Equation~\eqref{EQ: describing u_n} 
says $u_n$ agrees with $\varphi_n^{-1}$,
or equivalently $\gamma_{S}^{-1}$,
on $[a, n-1]$,
forcing \(u_n(n)=c=\gamma_S^{-1}(n)\). 
Thus, $u_n$ agrees with $ \gamma_S^{-1}$ on $[a, n]$, so $u_n = U$.
\end{proof}

Next, we show the observation made in Example~\ref{Ex: TBPD chain}.
\begin{lem}\label{L: In P_r-1}
Let \(\chain_{a,n}(D)=(u_n,\dots,u_1)\).
Then \(u_r(p_j)=j\) for  \(1 \le r\le j\le n\). Thus, \(u_r\in P_{r-1}\).
\end{lem}
\begin{proof}
By \eqref{EQ: define u_r}, \(u_r\in P_r\), so \(u_r(p_j)=j\) for \(j\in(r,n]\).
Let $\varphi_r$ be the cross section of $D$ above row $r$.
We have \(\varphi_r(r)=p_r\)
since pipe \(p_r\) connects to the top of \((r,r)\)).
Thus, \(u_r(p_r)=\varphi_r^{-1}(p_r)=r\).
\end{proof}

\begin{cor}\label{C: W in P0}
We have \(W\in P_0\) and \(\iota_0(W)\) is decreasing on \([p_1,0]\).
\end{cor}
\begin{proof}
By Lemma~\ref{L: U W}, 
the \(u_1\) in 
Lemma~\ref{L: In P_r-1} is $W$.
\end{proof}

We shall see Lemma~\ref{L: In P_r-1} leads to an alternative description of $u_r$, which is~\eqref{EQ: define u_r 2}. We start with some simple facts regarding $P_r$ and $\iota_r$.

\begin{lem}\label{L: iota left agree}
When \(u\in P_r\),  \(\iota_r(u)\) and \(u\) agree on  \([a,p_{r+1})\).
\end{lem}
\begin{proof}
The values \(r+1,\dots,n\) occur among \(u(p_{r+1}),\dots,u(n)\).
Thus the first \(p_{r+1}-a\) entries of the one-line notation of \(\iota_r(u)\) match those of \(u\).
\end{proof}

\begin{lem}\label{L: r-1 vs r}
If \(u\in P_{r-1}\subset P_r\), then
\[
\iota_{r-1}(u)(i)=
\begin{cases}
\iota_{r}(u)(i), & i\in[a,p_r),\\[2pt]
\iota_{r}(u)(i+1), & i\in[p_r,r-1].
\end{cases}
\]
\end{lem}
\begin{proof}
Because \(u\in P_{r-1}\), we have \(u(p_r)=r\).
By Lemma~\ref{L: iota left agree}, \(\iota_r(u)(p_r)=r\).
Removing this entry from the one-line notation of \(\iota_r(u) \in S_{[a,r]}\) yields that of \(\iota_{r-1}(u) \in S_{[a, r-1]}\).
\end{proof}

Combining these, we obtain an equivalent description of \(u_r\).

\begin{cor}\label{C: Alternative description}
If \(u_r\in P_{r-1}\) and \(\varphi_r\) satisfy \eqref{EQ: define u_r}, then
\begin{equation}\label{EQ: define u_r 2}
\text{\(u_r\in P_{r-1}\) with \(\iota_{r-1}(u_r)=\varphi_r^{-1}\) on \([a,p_r)\) and decreasing on \([p_r,r-1]\).}
\end{equation}
In particular, by Lemma~\ref{L: In P_r-1}, one may replace \eqref{EQ: define u_r} by \eqref{EQ: define u_r 2} in Definition~\ref{D: u from TBPD}.
\end{cor}
\begin{proof}
Immediate from Lemma~\ref{L: r-1 vs r}.
\end{proof}

Lemma~\ref{L: In P_r-1} shows many fixed points in the $\chain(D)$. We show elements of $C(U, W,\rev(\alpha))$ satisfy a similar property. We need the following lemma which says $\iota_r$ behaves nicely with increasing chains.

\begin{lem}\label{L: iota preserves inc}
Let \(u,v\in P_r\) and \(k\in[a,p_{r+1})\).
Then
\[
u \xrightarrow{k} v \quad\Longleftrightarrow\quad \iota_r(u) \xrightarrow{k} \iota_r(v),
\]
and in this case
\[
\fix_{(k,n]}(u,v)=\fix_{(k,r]}(\iota_r(u),\iota_r(v))\sqcup [r+1,n].
\]
\end{lem}
\begin{proof}
By repeated application of Lemma~\ref{L: Insert large number}.
\end{proof}

\begin{lem}\label{L: Chain fix points}
If \((u_n,\dots,u_1)\in C(U,W,\rev(\alpha))\), then \(u_r\in P_{r-1}\) and
\(\iota_{r-1}(u_r)\) is decreasing on \([p_r,r-1]\).
\end{lem}
\begin{proof}
By Lemma~\ref{L: finiteness of inc chain}, each \(u_i\in S_{[a,n]}\).
We argue by induction on \(r\).
For \(r=1\) this is Corollary~\ref{C: W in P0}.
Assume the claim holds for \(r\).
For \(j\in[r+1,n]\), we have \(u_r(p_j)=j\).
Since \(p_j\ge p_{r+1}>\alpha_r\), Lemma~\ref{L: LIRD} with \(u_{r+1}\xrightarrow{\alpha_r}u_r\) gives
\(u_{r+1}(p_j)\ge u_r(p_j)=j\).
On the other hand, 
$u_{r+1} \in S_{[a,n]}$
gives \(u_{r+1}(p_j)\le n\).
Another induction on $j = n, \dots, r+1$
gives \(u_{r+1}(p_j)=j\),
so \(u_{r+1}\in P_r\).

By the inductive hypothesis \(\iota_{r-1}(u_r)\) decreases on \([p_r,r-1]\).
By Lemma~\ref{L: r-1 vs r},
\(\iota_r(u_r)\) decreases on \((p_r,r]\) which contains \([p_{r+1},r]\).
By Lemma~\ref{L: iota preserves inc}, \(\iota_r(u_{r+1}) \xrightarrow{\alpha_r} \iota_r(u_r)\).
Lemma~\ref{L: decreasing preserved} then implies \(\iota_r(u_{r+1})\) decreases on \([p_{r+1},r]\).
\end{proof}

Now we can prove Proposition~\ref{P: TBPD bijection}.

\begin{proof}[Proof of Proposition~\ref{P: TBPD bijection}]
Let \(D\in\TBPD_{a,n}(\gamma)\) and \((u_n,\dots,u_1)=\chain(D)\).
Write \(\varphi_r\) for the cross section of \(D\) above row \(r\).
By Lemma~\ref{L: U W}, \(u_n=U\) and \(u_1=W\).
To prove \((u_n,\dots,u_1)\in C(U,W,\rev(\alpha))\) and \eqref{EQ: TBPD chain wt}, it suffices to show that for each \(r\in[n-1]\),
\begin{equation}\label{EQ: TBPD proof}
u_{r+1}\xrightarrow{\alpha_r} u_r
\quad\text{and}\quad
\fix_{(\alpha_r,n]}(u_{r+1},u_r)=\{\,c\in[a,r]: D(r,c)=\btile\,\}\sqcup [r+1,n].
\end{equation}
Consider row \(r\) of \(D\) with its labeling
inherited from $D$. It is a BPD of
\(\{r\}\times[a,r]\) with boundary condition \(\theta^r\) given by
\begin{equation}\label{EQ: theta}
\theta_N^r=\varphi_r,\qquad
\theta_S^r=\varphi_{r+1},\qquad
\theta_W^r(r)=\varnothing,\qquad
\theta_E^r(r)=\gamma_E(r).
\end{equation}
Here \(\theta_E^r(r)=p_{r+1}-1=\alpha_r\) if a pipe enters from row \(r\), and \(\theta_E^r(r)=\varnothing\) otherwise.
By Definition~\ref{D: u from TBPD}, \(\iota_r(u_r)\in S_{[a,r]}\) agrees with \((\theta_N^{r})^{-1}\) on \([a,p_r]\) and is decreasing on \((p_r,r]\).
By Lemma~\ref{L: Chain fix points}, \(\iota_r(u_{r+1})\) agrees with \((\theta_S^{r})^{-1}\) on \([a,\alpha_r]\) and is decreasing on \((\alpha_r,r]\).
Lemma~\ref{L: one row BPD, right entry} and Lemma~\ref{L: one row BPD, no right entry} then yields
\begin{equation}\label{EQ: TBPD proof iota}
\iota_r(u_{r+1})\xrightarrow{\alpha_r}\iota_r(u_r)
\quad\text{and}\quad
\fix_{(\alpha_r,r]}(\iota_r(u_{r+1}),\iota_r(u_r))
=\{\,c\in[a,r]: D(r,c)=\btile\,\}.
\end{equation}
Applying Lemma~\ref{L: iota preserves inc} to~\eqref{EQ: TBPD proof iota} gives \eqref{EQ: TBPD proof}.

Thus \(\chain(\cdot)\) maps \(\TBPD_{a,n}(\gamma)\) into \(C(U,W,\rev(\alpha))\).
Fix \((u_n,\dots,u_1)\in C(U,W,\rev(\alpha))\).
For bijectivity, 
it remains to show there
is a unique $D \in \TBPD_{a,n}(\gamma)$ 
sent to $(u_n, \dots, u_1)$.
By Lemma~\ref{L: Chain fix points},  \(u_r\in P_{r-1}\) and \(\iota_{r-1}(u_r)\) decreases on \([p_r,r-1]\). Thus, \(\iota_r(u_r)\) decreases on \((p_r,r]\) by Lemma~\ref{L: r-1 vs r}.
Define cross sections \(\varphi_r:[a,r]\to [a,p_r]\sqcup\{\varnothing\}\) by
\[
\varphi_r(c)=
\begin{cases}
u_r^{-1}(c), & u_r^{-1}(c)\in[a,p_r],\\
\varnothing, & \text{otherwise}.
\end{cases}
\]
Then \eqref{EQ: define u_r} holds.
Let \(\theta^r\) be as in \eqref{EQ: theta}. 
Notice that
\[ \begin{aligned} & \chain(D) = (u_n, \dots, u_1). \\[4pt] \Leftrightarrow\;& \text{$D$ has $\varphi_r$ as the cross section above row $r$ for all $r \in [n]$} \\[4pt] \Leftrightarrow\;& \text{$D$ restricted to $\{r\} \times [a,r]$ is a BPD satisfying $\theta^r$ for $r \in [n-1]$.} \end{aligned} \]
By Lemma~\ref{L: iota preserves inc},
$u_{r+1} \xrightarrow{\alpha_r} u_r$
gives
\(\iota_r(u_{r+1})\xrightarrow{\alpha_r}\iota_r(u_r)\) .
By
Lemma~\ref{L: one row BPD, right entry} and Lemma~\ref{L: one row BPD, no right entry}, there is a unique BPD on \(\{r\}\times[a,r]\) with boundary \(\theta^r\) for each \(r\in[n-1]\).
Concatenating these rows and using Remark~\ref{R: Last row} determines a unique \(D\in\TBPD_{a,n}(\gamma)\) with \(\chain(D)=(u_n,\dots,u_1)\).
\end{proof}

\medskip
Our next goal is Lemma~\ref{L: in fixed for increasing}.
We begin with a monotonicity fact for elements of \(C(U,W,\alpha)\).

\begin{lem}\label{L: increasing u1 to un}
Let \((u_1,\dots,u_n)\in C(U,W,\alpha)\).
For any \(i\in[n-1]\) and \(j\in[a,\alpha_i]\),
\[
u_i(j)\le u_{i+1}(j)\le \cdots \le u_n(j).
\]
\end{lem}
\begin{proof}
We have
\[
u_i \le_{\alpha_i} u_{i+1} \le_{\alpha_{i+1}} \cdots \le_{\alpha_{n-1}} u_n,
\qquad
\alpha_i<\alpha_{i+1}<\cdots<\alpha_{n-1}.
\]
Iterating Lemma~\ref{L: LIRD} gives the claim.
\end{proof}

We now establish the property of $C(U, W, \alpha)$ that is analogous to Lemma~\ref{L: Chain fix points} and is observed in Example~\ref{Ex: TBPD reversed chains}.

\begin{lem}\label{L: 3 statements}
Let \((u_1,\dots,u_n)\in C(U,W,\alpha)\).
Then:
\begin{enumerate}
\item For \(i\in[n]\) and \(j\in[0,i]\),
\[
u_i([a,\alpha_j])\subseteq [a,j],
\]
where by Definition~\ref{D: U W alpha} we interpret \(\alpha_0:=p_1-1\).
\item For every \(i\in[n]\), \(u_i^{-1}(i)\ge p_i\).
\item For every \(i\in[n]\), one has
\[
u_1^{-1}(i)=u_2^{-1}(i)=\cdots=u_i^{-1}(i).
\]
\end{enumerate}
\end{lem}
\begin{proof}
(1) If \(u_i([a,\alpha_j])\not\subset [a,j]\), then by Lemma~\ref{L: increasing u1 to un}, \(u_n([a,\alpha_j])\) contains a value \(>j\).
By Corollary~\ref{C: W in P0}, \(u_n=W\in P_0\), so $j+1, \dots, n$
are on indices $p_{j+1},
\dots, p_n$ of $u_n$.
However, the interval \([a,\alpha_j]=[a,p_{j+1})\) does not contain any of \(p_{j+1},\dots,p_n\). Contradiction.

(2) Apply (1) with \(j=i-1\) to get \(u_i([a,\alpha_{i-1}]) = u_i([a,p_i))\subseteq [a,i-1]\), hence \(i\in u_i([p_i,n])\).

(3) Suppose not. We can find \(j<i\) with \(u_j^{-1}(i)\ne u_{j+1}^{-1}(i)\).
By (1), \(u_j^{-1}(i)>\alpha_j\).
Since \(i\notin\fix_{(\alpha_j,n]}(u_j,u_{j+1})\), Lemma~\ref{L: not fixed point} yields a \(t\in[a,\alpha_j]\) with \(u_{j+1}(t)\ge i\).
However, by (1),
$u_{j+1}([a,\alpha_j]) \subseteq [a, j]$ which does not contain
any number at least $i$. Contradiction. 
\end{proof}

Finally, we can prove the two lemmas
in the previous section. 

\begin{proof}[Proof of Lemma~\ref{L: in fixed for increasing}]
Fix \(i\in[n-1]\) and \(j\in[i+1,n]\).
By Lemma~\ref{L: 3 statements}(3),
\(u_i^{-1}(j)=u_{i+1}^{-1}(j)=u_j^{-1}(j)\).
By Lemma~\ref{L: 3 statements}(2), this common value is \(\ge p_j>\alpha_{j-1}\ge \alpha_i\).
Hence \(j\in \fix_{(\alpha_i,n]}(u_i,u_{i+1})\).
\end{proof}

\begin{proof}[Proof of Lemma~\ref{L: fixed numbers in the final}]
By Lemma~\ref{L: 3 statements}(2), \(u_i^{-1}(i)\ge p_i\).
We consider two cases.

\emph{Case 1:} \(u_i^{-1}(i)>\alpha_i\).
Since \((u_1,\dots,u_n)\in\widetilde{C}(U,W,\alpha)\), we have \(i\notin \fix_{(\alpha_i,n]}(u_i,u_{i+1})\).
Lemma~\ref{L: not fixed point} implies that there exists \(t\in[a,\alpha_i]\) with \(u_{i+1}(t)\ge i\).
Then by Lemma~\ref{L: increasing u1 to un},
\[
i \le u_{i+1}(t)\le \cdots \le u_n(t).
\]
The only \(t\in[a,\alpha_i]\) with \(u_n(t)\ge i\) is \(t=p_i\), where \(u_n(p_i)=i\).
Hence \(u_{i+1}(p_i)=\cdots=u_n(p_i)=i\).

\emph{Case 2:} \(u_i^{-1}(i)\le \alpha_i=p_{i+1}-1\).
Since also \(u_i^{-1}(i)\ge p_i\) and \(p_{i+1}-p_i\in\{1,2\}\), we must have \(u_i^{-1}(i)=p_i\) or \(p_i+1\).
The latter is impossible: if \(u_i(p_i+1)=i\), then Lemma~\ref{L: increasing u1 to un} gives \(u_n(p_i+1)\ge i\), a contradiction to $u_n = W \in P_0$.
Once we have $u_i(p_i) = i$,
Lemma~\ref{L: increasing u1 to un} implies
$u_{i+1}(p_i) = \cdots = u_n(p_i) =  i$.

Finally, fix \(i\in[n-1]\).
In both \(u_i\) and \(u_{i+1}\) the values \(1,\dots,i-1\) lie at positions \(p_1,\dots,p_{i-1}\), so none of these values belongs to \(\fix_{(\alpha_i,n]}(u_i,u_{i+1})\).
\end{proof}

\bibliographystyle{alpha}
\bibliography{citation}{}

\end{document}